\documentclass{article}

\usepackage[a4paper]{geometry}
\usepackage{booktabs}
\usepackage{graphicx}
\usepackage{amsmath}
\usepackage{amssymb}
\usepackage{mathrsfs}
\usepackage{latexsym}
\usepackage{subfigure}
\usepackage{floatrow}
\usepackage{crop}
\usepackage{algpseudocode,algorithm}
\usepackage{multirow}
\usepackage{bm}
\usepackage{bbm}
\usepackage{enumerate}
\usepackage{url}
\usepackage{array}
\usepackage{paralist}
\usepackage{diagbox}
\usepackage{booktabs}
\usepackage[dvipsnames]{xcolor}
\usepackage[colorlinks = true, pdfstartview = FitV, linkcolor = blue, citecolor = blue, urlcolor = blue]{hyperref}
\usepackage{authblk}

\usepackage{derivative}
\usepackage{placeins}
\usepackage{parskip}

\usepackage{lmodern}

\usepackage{rotating}

\usepackage[capitalise]{cleveref}
\crefname{equation}{}{}
\crefname{figure}{Figure}{Figures}
\creflabelformat{equation}{\textup{(#2#1#3)}}
\crefname{assumption}{Assumption}{Assumptions}
\algrenewcommand{\algorithmiccomment}[1]{\hfill\texttt{//} #1}
\newcommand{\LeftComment}[1]{\Statex\(\texttt{//}\) #1 \(\texttt{//}\)}

\makeatletter

\providecommand*{\theHalgorithm}{\arabic{algorithm}}
\providecommand*{\theHALG@line}{\theHalgorithm.\arabic{ALG@line}}

\AddToHook{env/algorithmic/before}{\def\@currentcounter{ALG@line}}

\makeatother

\crefalias{ALG@line}{line}
\crefname{line}{line}{lines}
\Crefname{line}{Line}{Lines}

\usepackage{fullpage}
\usepackage{multirow}

\usepackage[sort,numbers]{natbib}

\usepackage{arydshln}
\usepackage{enumitem}
\setlist[enumerate]{leftmargin=*,wide=0em, itemsep = 3pt,topsep= 3pt, label = {\bfseries \arabic*.}}
\setlist[itemize]{leftmargin=*,wide=0em, itemsep= 3pt,topsep= 3pt}

\usepackage{pifont}
\newcommand{\cmark}{\textcolor{green!60!black}{\ding{51}}} 
\newcommand{\xmark}{\textcolor{red!70!black}{\ding{55}}}   

\usepackage{xspace}

\usepackage{accents}

\usepackage{stackengine}
\stackMath
\newcommand\tsup[2][2]{%
	\def\useanchorwidth{T}%
	\ifnum#1>1%
	\stackon[-.5pt]{\tsup[\numexpr#1-1\relax]{#2}}{\scriptscriptstyle\sim}%
	\else%
	\stackon[.5pt]{#2}{\scriptscriptstyle\sim}%
	\fi%
}
\makeatletter
\newcommand{\longdash}[1][2em]{%
	\makebox[#1]{$\m@th\smash-\mkern-7mu\cleaders\hbox{$\mkern-2mu\smash-\mkern-2mu$}\hfill\mkern-7mu\smash-$}}
\makeatother
\newcommand{\omitskip}{\kern-\arraycolsep}

\newcommand{\real}{\mathbb{R}}

\DeclareMathOperator*{\argmin}{arg\,min}

\newcommand*\tageq{\refstepcounter{equation}\tag{\theequation}}

\newcommand{\tol}{\text{tol}}
\newcommand{\tint}{\text{int}}
\newcommand{\bndry}{\text{bndry}}

\newcommand{\flow}{f_{\text{low}}}

\newcommand{\LHess}{\grad_{\xx \xx}^2 \sL}
\newcommand{\Lgrad}{\gradx \sL}
\newcommand{\gradf}{\grad f }
\newcommand{\Hessf}{\grad^2 f}
\newcommand{\primeb}[1]{\left[#1\right]'}

\newcommand{\SOL}{\ensuremath{\texttt{SOL}}}
\newcommand{\LPC}{\ensuremath{\texttt{LPC}}}
\newcommand{\FLAG}{\ensuremath{\texttt{FLAG}}}

\newcommand{\bSSh}{\hat{\bSS}}

\newcommand{\hPP}{\hat{\PP}}
\newcommand{\hPPtt}{\hPP^{(t)}}

\newcommand{\PPNJ}{\PP_{\Null(\JJ)}}
\newcommand{\PPNJk}{\PP_{\Null(\JJ_k)}}
\newcommand{\PPNJxx}{\PP_{\Null(\JJ(\xx))}}
\newcommand{\PPRJT}{\PP_{\Range(\JJ^\transpose)}}

\newcommand{\tJJ}{\tilde{\JJ}}
\newcommand{\tbgg}{\tilde{\bgg}}

\newcommand{\yyzero}{\yy^{(0)}}

\newcommand{\vvtt}{\vv^{(t)}}
\newcommand{\vvii}{\vv^{(i)}}

\newcommand{\tdd}{\tilde{\dd}}
\newcommand{\bdd}{\bar{\dd}}
\newcommand{\ddtt}{\dd^{(t)}}
\newcommand{\bddtt}{\bar{\dd}^{(t)}}
\newcommand{\ddii}{\dd^{(i)}}
\newcommand{\bddii}{\bar{\dd}^{(i)}}
\newcommand{\ddss}{\dd^{(s)}}
\newcommand{\bddss}{\bar{\dd}^{(s)}}
\newcommand{\ddttm}{\dd^{(t-1)}}

\newcommand{\tddtt}{\tilde{\dd}^{(t)}}

\newcommand{\bddttm}{\bar{\dd}^{(t-1)}}

\newcommand{\bddttp}{\bar{\dd}^{(t+1)}}

\newcommand{\ddone}{\dd^{(1)}}

\newcommand{\bddzero}{\bar{\dd}^{(0)}}

\newcommand{\bddone}{\bar{\dd}^{(1)}}

\newcommand{\rrzero}{\rr^{(0)}}

\newcommand{\brrzero}{\bar{\rr}^{(0)} }

\newcommand{\rrii}{\rr^{(i)}}

\newcommand{\brrii}{\bar{\rr}^{(i)}}
\newcommand{\rrtt}{ \rr^{(t)} }
\newcommand{\brrtt}{ \bar{\rr}^{(t)} }

\newcommand{\trrtt}{\tilde{\rr}^{(t)}}

\newcommand{\brrttm}{\bar{\rr}^{(t-1)}}

\newcommand{\normal}{\text{norm}}
\newcommand{\tangent}{\text{tan}}

\newcommand{\omegatan}{\omega_{\tangent}}
\newcommand{\gammatan}{\gamma_{\tangent}}
\newcommand{\gammanorm}{\gamma_{\normal}}

\newcommand{\tautan}{\tau_{\tangent}}
\newcommand{\taunorm}{\tau_{\normal}}
\newcommand{\taunull}{\tau_{\text{null}}}
\newcommand{\pitrial}{\pi^{\text{trial}}}

\newcommand{\pib}{\bar{\pi}}
\newcommand{\pit}{\tilde{\pi}}

\newcommand{\alphat}{\tilde{\alpha}}
\newcommand{\alphab}{\bar{\alpha}}
\newcommand{\Db}{\bar{D}}
\newcommand{\epsc}{\varepsilon_c}
\newcommand{\epsp}{\varepsilon_p}

\newcommand{\tbomega}{\tilde{\bomega}}
\newcommand{\tbOmega}{\tilde{\bOmega}}
\newcommand{\hbOmega}{\hat{\bOmega}}
\newcommand{\hbOmegatt}{\hat{\bOmega}^{(t)}}

\newcommand{\sB}{\mathcal{B}}

\newcommand{\sD}{\mathcal{D}}

\newcommand{\sF}{\mathcal{F}}

\newcommand{\sK}{\mathcal{K}}
\newcommand{\sR}{\mathcal{R}}

\newcommand{\sL}{\mathcal{L}}
\newcommand{\sN}{\mathcal{N}}

\newcommand{\sO}{\mathcal{O}}
\newcommand{\sP}{\mathcal{P}}

\newcommand{\sV}{\mathcal{V}}

\newcommand{\sX}{\mathcal{X}}

\newcommand{\sZ}{\mathcal{Z}}

\renewcommand {\AA}  { {\mathbf{A}} }

\newcommand {\HH}  { {\mathbf{H}} }

\newcommand {\UU}  { {\mathbf{U}} }

\newcommand {\PP}  { {\mathbf{P}} }

\newcommand {\bSS}  { {\mathbf{S}} }
\newcommand {\JJ}  { {\mathbf{J}} }

\newcommand {\ZZ}  { {\mathbf{Z}} }
\newcommand {\bOmega}  { {\mathbf{\Omega}} }

\newcommand {\bSigma}  { {\mathbf{\Sigma}} }
\newcommand{\eye}{\mathbf{I}}

\newcommand {\bb}  { {\bf b} }
\newcommand {\cc}  { {\bf c} }
\newcommand {\dd}  { {\bf d} }

\newcommand {\bgg}  { {\bf g} }

\newcommand {\yy}  { {\bf y} }

\newcommand {\rr}  { {\bf r} }
\newcommand {\uu}  { {\bf u} }
\newcommand {\qq}  { {\bf q} }
\newcommand {\pp}  { {\bf p} }

\newcommand {\vv}  { {\bf v} }
\newcommand {\ww}  { {\bf w} }
\newcommand {\xx}  { {\bf x} }
\newcommand {\zz}  { {\bf z} }

\newcommand {\blambda} {\bm \lambda}

\newcommand {\bbeta} {\bm \beta}

\newcommand {\bomega} {\bm \omega}
\newcommand {\zero}  { {\bf 0} }

\newcommand {\Rank}  { {\textnormal{Rank}} }

\newcommand {\Range}  { \sR }
\newcommand {\Null}  { \sN }
\newcommand {\Span}  { {\textnormal{Span}} }

\newcommand {\diag}  { {\textnormal{diag}} }

\newcommand{\HHt}{\tilde{\HH}}
\newcommand{\HHb}{\bar{\HH}}

\newcommand{\bbgg}{\bar{\bgg}}

\newcommand{\vvec}[2]{
\begin{pmatrix} #1 \\ #2 \end{pmatrix}
}

\makeatletter
\newcommand*{\transpose}{%
	{\mathpalette\@transpose{}}%
}
\newcommand*{\@transpose}[2]{%
	\raisebox{\depth}{$\m@th#1\intercal$}%
}
\makeatother

\newcommand {\tHH}  { {\tilde{\HH}} }

\newcommand {\yytt}  { {\yy^{(t)}} }
\newcommand {\yyttm}  { {\yy^{(t-1)}} }

\newcommand {\tPP}  { {\tilde{\PP} }}

\renewcommand{\vec}[1]{\ensuremath{\mathbf{#1}}}
\newcommand{\grad}{\ensuremath {\vec \nabla}}
\newcommand{\gradx}{\ensuremath {\vec \nabla}_\xx}

\newcommand{\defeq}{\triangleq}

\definecolor{forestgreen}{rgb}{0.13, 0.55, 0.13}

\definecolor{amber}{rgb}{1.0, 0.75, 0.0}

\definecolor{bananayellow}{rgb}{.8, 0.6, 0}

\newcounter{comment}

\usepackage{amsthm}
\usepackage[framemethod=TikZ]{mdframed}

\newmdtheoremenv[%
linewidth = 1pt,%
roundcorner = 10pt,%
leftmargin = 0,%
rightmargin = 0,%
backgroundcolor = green!3,%
outerlinecolor = blue!70!black,%
innertopmargin = \topskip,%
innerbottommargin=\topskip,%
splittopskip = \topskip,%
skipabove=\baselineskip,
skipbelow=\baselineskip,
]{theorem}{Theorem}

\newmdtheoremenv[%
linewidth = 1pt,%
roundcorner = 10pt,%
leftmargin = 0,%
rightmargin = 0,%
backgroundcolor = green!3,%
outerlinecolor = blue!70!black,%
innertopmargin = \topskip,%
innerbottommargin=\topskip,%
splittopskip = \topskip,%
skipabove=\baselineskip,
skipbelow=\baselineskip,
]{corollary}{Corollary}

\newmdtheoremenv[%
linewidth = 1pt,%
roundcorner = 10pt,%
leftmargin = 0,%
rightmargin = 0,%
backgroundcolor = green!3,%
outerlinecolor = blue!70!black,%
innertopmargin = \topskip,%
innerbottommargin=\topskip,%
splittopskip = \topskip,%
skipabove=\baselineskip,
skipbelow=\baselineskip,
]{lemma}{Lemma}

\newmdtheoremenv[%
linewidth = 1pt,%
roundcorner = 10pt,%
leftmargin = 0,%
rightmargin = 0,%
backgroundcolor = blue!3,%
outerlinecolor = blue!70!black,%
innertopmargin = \topskip,%
innerbottommargin=\topskip,%
splittopskip = \topskip,%
skipabove=\baselineskip,
skipbelow=\baselineskip,
]{definition}{Definition}

\newmdtheoremenv[%
linewidth = 1pt,%
roundcorner = 10pt,%
leftmargin = 0,%
rightmargin = 0,%
backgroundcolor = green!3,%
outerlinecolor = blue!70!black,%
innertopmargin = \topskip,%
innerbottommargin=\topskip,%
splittopskip = \topskip,%
skipabove=\baselineskip,
skipbelow=\baselineskip,
]{proposition}{Proposition}

\newmdtheoremenv[%
linewidth = 1pt,%
roundcorner = 10pt,%
leftmargin = 0,%
rightmargin = 0,%
backgroundcolor = green!3,%
outerlinecolor = blue!70!black,%
innertopmargin = \topskip,%
innerbottommargin=\topskip,%
splittopskip = \topskip,%
skipabove=\baselineskip,
skipbelow=\baselineskip,
]{condition}{Condition}

\crefname{condition}{Condition}{Conditions}
\Crefname{condition}{Condition}{Conditions}
\newlist{conditems}{enumerate}{1}
\setlist[conditems,1]{label=(\roman*), ref=\thecondition-(\roman*)}
\crefname{conditemsi}{Condition}{Conditions}
\Crefname{conditemsi}{Condition}{Conditions}

\newlist{assitems}{enumerate}{1}
\setlist[assitems,1]{label=(\roman*), ref=\theassumption-(\roman*)}
\crefname{assitemsi}{Assumption}{Assumptions}
\Crefname{assitemsi}{Assumption}{Assumptions}

\newmdtheoremenv[%
linewidth = 1pt,%
roundcorner = 10pt,%
leftmargin = 0,%
rightmargin = 0,%
backgroundcolor = yellow!3,%
outerlinecolor = blue!70!black,%
innertopmargin = \topskip,%
innerbottommargin=\topskip,%
splittopskip = \topskip,%
skipabove=\baselineskip,
skipbelow=\baselineskip,
]{assumption}{Assumption}

\theoremstyle{definition}
\newmdtheoremenv[%
linewidth = 1pt,%
roundcorner = 10pt,%
leftmargin = 0,%
rightmargin = 0,%
backgroundcolor = cyan!3,%
outerlinecolor = blue!70!black,%
innertopmargin = \topskip,%
innerbottommargin=\topskip,%
splittopskip = \topskip,%
skipabove=\baselineskip,
skipbelow=\baselineskip,
]{example}{Example}

\theoremstyle{definition}
\newmdtheoremenv[%
linewidth = 1pt,%
roundcorner = 10pt,%
leftmargin = 0,%
rightmargin = 0,%
backgroundcolor = red!3,%
outerlinecolor = blue!70!black,%
innertopmargin = \topskip,%
innerbottommargin=\topskip,%
splittopskip = \topskip,%
skipabove=\baselineskip,
skipbelow=\baselineskip,
]{remark}{Remark}

\newmdtheoremenv[%
linewidth = 1pt,%
roundcorner = 10pt,%
leftmargin = 0,%
rightmargin = 0,%
backgroundcolor = gray!3,%
outerlinecolor = blue!70!black,%
innertopmargin = \topskip,%
innerbottommargin=\topskip,%
splittopskip = \topskip,%
skipabove=\baselineskip,
skipbelow=\baselineskip,
]{fact}{Fact}

\usepackage{tikz}
\usepackage{xparse}

\NewDocumentCommand\DownArrow{O{2.0ex} O{black}}{%
	\mathrel{\tikz[baseline] \draw [<-, line width=0.5pt, #2] (0,0) -- ++(0,#1);}
}

\usepackage{listings} 

\definecolor{mygreen}{rgb}{0,0.6,0}
\definecolor{mygray}{rgb}{0.5,0.5,0.5}
\definecolor{mymauve}{rgb}{0.58,0,0.82}

\floatstyle{ruled}
\newfloat{terminationblock}{t}{lotb}
\floatname{terminationblock}{Termination Block}

% ...
\NewDocumentEnvironment{terminationblockalg}{O{t} m o}
{
    \begin{terminationblock}[#1]
    \caption{#2}
    \IfValueT{#3}{\label{#3}}
    \begin{algorithmic}[1]
}
{
    \end{algorithmic}
    \end{terminationblock}
}
\crefname{terminationblock}{termination block}{termination blocks}
\Crefname{terminationblock}{Termination Block}{Termination Blocks}

\newcommand*\dotprod[1]{\left\langle #1\right\rangle}
\newcommand*\vnorm[1]{\left\| #1\right\|}

\newcommand*\bigO[1]{\mathcal O\left( #1\right)}

\usepackage{dsfont}


\title{Primal-Dual Inexact Newton-MR for Nonconvex Optimization with Equality Constraints}

\author{
Oscar Smee\textsuperscript{$\dagger$*}
\hspace{2.5em}
Fred Roosta\textsuperscript{$\dagger$}
}

\date{September 9, 2026}

\begin{document}

\maketitle

\begingroup
\renewcommand{\thefootnote}{$\dagger$}
\footnotetext{School of Mathematics and Physics, The University of Queensland, Brisbane, Australia.}
\endgroup

\begingroup
\renewcommand{\thefootnote}{*}
\footnotetext{Corresponding author: \texttt{oscar.smee@uq.edu.au}.}
\endgroup

\section*{Abstract}

Optimization problems with nonlinear equality constraints arise throughout science, engineering, and increasingly in machine learning. Prominent methods for solving such problems include sequential quadratic programming and, more broadly, primal-dual Newton methods. Classical analyses of these methods typically rely on strong assumptions, perhaps most notably positive definiteness of the Lagrangian Hessian on the null space of the constraint Jacobian. In practice, this assumption often necessitates strong regularization or the use of a positive definite Hessian surrogate. Moreover, in large-scale settings, solving the primal-dual Newton subproblem exactly is often computationally infeasible. To address these issues, we propose an inexact primal-dual Newton method with an inner solver based on the conjugate residual (CR) method. Exploiting recently established properties of CR, including negative-curvature detection, iterate monotonicity, and descent guarantees, our method handles indefiniteness in the subproblem directly as it arises. Our method thereby avoids detrimental regularization of the Lagrangian Hessian while naturally accommodating inexact solves. We establish worst-case global convergence guarantees and demonstrate strong empirical performance on large-scale, nonconvex problems.

\section{Introduction}

The problem we consider is
\begin{align}
\min_{\xx \in \real^d} \quad & f(\xx) \quad \text{s.t.} \quad  \cc(\xx) = 0,
\label{eqn:equality constrained problem}
\end{align}
where $f : \real^d \to \real$ is a twice continuously differentiable objective function and $\cc : \real^d \to \real^c$ is a twice continuously differentiable constraint mapping representing $c$ equality constraints. Problems of the form \cref{eqn:equality constrained problem} arise throughout science, engineering, and machine learning, where an objective must be optimized subject to physical laws or structural constraints that must be satisfied exactly. Examples include physics-informed machine learning \citep{raissi_physics_informed_2019,basir_physics_2022,basir_investigating_2023,son_enhanced_2023}, optimization layers in neural networks \citep{amosOptNet2017}, PDE-constrained optimization \citep{hinze2008optimization}, and engineering design and control \citep{hinze2008optimization,bieglerNonlinearProgramming2010,rao2019engineering}. Our focus is on developing computationally feasible algorithms for the modern large-scale setting in which the number of the decision variable is very large, i.e., $d \gg 1$.

Identifying global solutions of \cref{eqn:equality constrained problem} is typically computationally intractable. Therefore, a more appropriate algorithmic target is an approximate local solution, as characterized by first-order necessary conditions. Recall that the Lagrangian associated with \cref{eqn:equality constrained problem} is given by
\[
\sL(\xx, \blambda) \defeq f(\xx) + \dotprod{\blambda, \cc(\xx)},
\]
where $\blambda \in \real^c$ denotes the vector of Lagrange multipliers. 
Under an appropriate constraint qualification, any local solution of \cref{eqn:equality constrained problem} must be a first-order stationary point of the Lagrangian \citep{nocedal_numerical_2006}: 
Accordingly, for a pair of tolerances $(\epsc,\epsp)\in(0,1)^2$, we seek a pair $(\xx,\blambda)$ satisfying the \emph{$(\epsc,\epsp)$-stationarity condition}
\begin{align*}
    \vnorm{\cc(\xx)} \leq \epsc
    \qquad \text{and} \qquad
    \vnorm{\Lgrad(\xx, \blambda)} \leq \epsp.
    \tageq\label{eqn:epsilon KKT stationary point}
\end{align*}
The first condition in \cref{eqn:epsilon KKT stationary point} captures feasibility, while the second controls primal stationarity.

The primal-dual Newton approach to identify stationary points of the Lagrangian involves applying Newton's method directly to the Lagrangian stationary point equation, generating the so-called Karush-Kuhn-Tucker (KKT) system defined by
\begin{align*}
    \begin{pmatrix}
        \LHess(\xx, \blambda) & \JJ^\transpose(\xx)\\
        \JJ(\xx) & 0 
    \end{pmatrix}
    \begin{pmatrix}
        \pp \\ \qq
    \end{pmatrix}
    =
    -\begin{pmatrix}
        \Lgrad(\xx, \blambda) \\ 
        \cc(\xx)
    \end{pmatrix}, \tageq\label{eqn:KKT-system}
\end{align*}
where $\pp$ and $\qq$ are the primal and dual updates, respectively, and $\JJ:\real^d \to \real^{c \times d}$, is the Jacobian of the constraint function. From \cref{eqn:KKT-system} we can iteratively update the primal-dual variables. While the system \cref{eqn:KKT-system} need not admit a solution in general, it has a unique solution under the following regularity condition.

\begin{condition}[KKT Regularity Condition]
\label{cond:KKT-Regularity}
\hfill
\begin{conditems}
    
    \item \label{cond:kkt:full-rank-Jacobian}
    The Jacobian matrix of the constraints, $\JJ(\xx)$, has full row rank.
    
    \item \label{cond:kkt:positive-definite-Lagrangian-Hessian}
    The Hessian of the Lagrangian, $\LHess(\xx,\blambda)$, is positive definite on the tangent space of the constraints, i.e., 
    \[
        \dotprod{\dd,\nabla_{\xx\xx}^2 \sL(\xx,\blambda)\dd} > 0,
        \qquad
        \forall \dd \neq 0
        \quad \text{such that} \quad
        \JJ(\xx)\dd = 0.
    \]
\end{conditems}
\end{condition}
Additionally, under \cref{cond:KKT-Regularity}, the primal update in \cref{eqn:KKT-system} is the unique global minimizer of the subproblem associated with the celebrated sequential quadratic programming (SQP) method. This subproblem combines a quadratic model of the objective with a linearization of the constraints: 
\begin{align*}
    \min_{\pp} \quad 
    \dotprod{\gradf(\xx),\pp}
    +
    \frac12
    \dotprod{\pp,\LHess(\xx,\blambda)\pp}
    \qquad
    \text{s.t.}
    \qquad
    \JJ\pp = -\cc.
    \tageq \label{eqn:SQP}
\end{align*}
The associated dual update is recovered from the Lagrange multipliers of \cref{eqn:SQP}. 
Under \cref{cond:KKT-Regularity}, SQP enjoys strong local convergence guarantees \citep{nocedal_numerical_2006}. 
Moreover, SQP methods can be globalized using either trust-region or line-search strategies \citep{nocedal_numerical_2006,sunOptimizationTheoryMethods2006a,connTrustRegionMethods2000,andreiModernNumericalNonlinear2022}. In this paper, we focus specifically on the line-search framework.

\paragraph{Key Limitation of Prior Works.}
Unfortunately, while \cref{cond:KKT-Regularity} may hold locally around a solution satisfying suitable regularity assumptions, it is generally unrealistic to expect this condition to hold globally. In particular, \cref{cond:kkt:positive-definite-Lagrangian-Hessian} cannot be guaranteed even in the ``best-case'' scenario where both the objective and constraint Hessians are positive definite, since the multiplier vector $\blambda$ may contain entries of arbitrary sign. Consequently, the Lagrangian Hessian $\nabla_{\xx\xx}^2 \sL(\xx,\blambda)$  may still be indefinite.
For this reason, many SQP methods replace the Lagrangian Hessian with a positive definite surrogate, such as a quasi-Newton approximation \citep{nocedal_numerical_2006,gillSNOPTSQPAlgorithm2005,gillSequentialQuadraticProgramming2010}, or heavily regularize the Hessian whenever negative curvature is detected (e.g., \citep{byrdInexactNonconvexNewtonMethod2010,curtisMatrixFreeAlgorithmEquality2010}). However, both strategies substantially alter the curvature information encoded in the true Lagrangian Hessian, which can significantly degrade the performance of second-order methods \citep{lim_complexity_2024}.
Indeed, such SQP variants can be viewed as closer to first-order or ``1.5th-order'' methods than to fully second-order algorithms.

\subsection{Our Method and Contributions.}
In this work we develop a \emph{fully second-order} method that makes direct use of the Lagrangian Hessian, even in the absence\footnote{In a slight abuse of terminology and following \citep{byrdInexactNonconvexNewtonMethod2010}, we will generally refer to the failure of \cref{cond:kkt:positive-definite-Lagrangian-Hessian} as `non-convexity'.} of \cref{cond:kkt:positive-definite-Lagrangian-Hessian}. Rather than modifying the Lagrangian Hessian through regularization or positive definite approximations, our approach is built around a reformulation of the primal-dual step computation that allows the curvature of the true Hessian to be exploited directly. The resulting framework combines classical primal-dual structure with an inexact Krylov inner solve, yielding a method that is both computationally practical in the large-scale setting and benefits from strong global convergence guarantees under standard assumptions. 

Our framework builds on the classical primal-dual Newton paradigm: at each iteration a primal search direction is obtained by approximately solving the KKT system and the resulting step is globalized using a merit function and line search. Following a standard approach to SQP (see, e.g., \citep[Chapter 18]{nocedal_numerical_2006}), we decompose the primal step into tangent and normal components.
The normal component primarily reduces the violation of the linearized constraints, whereas the tangent component improves the objective while maintaining linearized feasibility. 

Applying this decomposition, alongside a series of simplifications, results in two structured least-squares problems corresponding to the normal and tangent components. The normal step may be computed either directly or indirectly, depending on whether the number of constraints is small or large relative to the problem dimension. In addition to restoring feasibility, solving the normal problem reveals the null space of the Jacobian, which provides the structural information required for the tangent computation. The tangent step then reduces to a least-squares system involving the Lagrangian Hessian and gradient restricted to this null space. Remarkably, this tangent step formulation implicitly targets least-squares multipliers, providing additional flexibility in the choice of dual update; in this sense, the resulting algorithm can be viewed as essentially \textit{primal only}.

The central ingredient of our method is the observation that the tangent computation can be performed \textit{inexactly} using a projected conjugate residual (CR) iteration, which is closely related to the minimum residual (MINRES) algorithm \cite{limConjugateDirectionMethods2024}. Recent work \citep{liu_minres_2022,liu_obtaining_2024} has shown that CR/MINRES possess powerful properties when applied to optimization problems, including reliable detection of negative curvature and strong descent guarantees. These properties have recently been incorporated into several optimization algorithms for nonconvex problems \citep{smee_inexact_2024,lim_complexity_2024,liu_newton-mr_2023,roosta_newton-mr_2022,Smee2025FirstishOM}. 
Importantly, these results effectively ``open the black box'' of the inner Krylov solver, allowing the behavior of the iterative method to be incorporated directly into the analysis of the outer primal-dual scheme. By exploiting these properties, our method can handle an indefinite Lagrangian Hessian without resorting to costly regularization or positive definite approximations. Moreover, the Krylov framework naturally accommodates inexact solutions of the tangent subproblem, allowing the inner iteration to terminate early once easy-to-verify and intuitive termination conditions are satisfied. The resulting search directions are incorporated into a standard $\ell_2$ merit function framework with backtracking line search, ensuring global convergence. The result is our primal-dual Newton-MR method which can handle both nonconvexity and inexactness out-of-the-box. 

\paragraph{Contributions.}
Under standard smoothness and regularity assumptions commonly used in the SQP literature, we establish that the proposed line-search primal-dual Newton-MR method identifies a point satisfying \cref{eqn:epsilon KKT stationary point} within $\sO(\max\{\epsp^{-2}, \epsc^{-1}\})$ iterations. Our analysis allows the Lagrangian Hessian to be indefinite and accommodates inexact computation of the tangent step via projected Krylov iterations, while only accessing Hessian-vector products. To the best of our knowledge, this is the first iteration complexity result for a line-search SQP-type method that simultaneously permits both inexact inner solves and nonconvex Lagrangian curvature. Moreover, we demonstrate that our method performs effectively in practice on large-scale nonconvex optimization problems with highly nonlinear constraints. We summarize our contribution in comparison to the recent literature in \cref{tab:sqp-summary}.

\begin{table}[ht]
\centering
\caption{Summary of recent line-search SQP literature. Under the ``Nonconvexity'' column, ``\cmark'' indicates that the method can handle nonconvexity without any modification or regularization, while ``\cmark\textsuperscript{*}'' indicates that the method relies on an explicit regularization scheme. For an overview of the complexity guarantees obtained by line-search SQP methods we refer the reader to our literature review in \cref{sec:literature review}.}
\label{tab:sqp-summary}
\begin{tabular}{lcccc}
\toprule
\textbf{Paper} & \textbf{Nonconvexity} & \textbf{Inexactness}  & \textbf{Complexity guarantees}  \\
\midrule
\textbf{This Work}  & \cmark  & \cmark &  \cmark   \\
\midrule
\citep{curtis_inexact_2021,berahasModifiedLineSearch2024}   & \xmark & \cmark & \xmark   \\
\midrule
\citep{curtisWorstCaseComplexitySQP2022,berahasSequentialQuadraticOptimization2020,oneill_two_2024,berahasSequentialQuadraticProgramming2025}   & \xmark & \xmark & \cmark   \\
\midrule
\citep{byrdInexactNonconvexNewtonMethod2010,curtisMatrixFreeAlgorithmEquality2010}   & \cmark\textsuperscript{*} & \cmark & \xmark   \\
\bottomrule
\end{tabular}
\end{table}

\paragraph{Outline.} 
The remainder of this paper is organized as follows. We conclude this section by introducing notation and reviewing the relevant literature. In \cref{sec:building blocks}, we lay out the basic building blocks of our algorithm, including the reformulation of the primal-dual Newton system into tangent and normal subproblems and the properties of the CR inner solver used for the tangent problem. We then discuss the globalization of the method through a merit function and line search. In \cref{sec:direct algorithm statement}, we present the proposed algorithm for the case where the normal component subproblem can be solved directly via a factorization, and establish convergence guarantees for this direct variant. In \cref{sec:indirect normal approach}, we extend the method to the large-constraint setting where the Jacobian cannot be factorized. Using inexact normal steps and approximate null-space projections that require only Jacobian-vector products, we show that similar convergence guarantees can be obtained in this indirect case. Finally, \cref{sec:numerical} reports numerical experiments demonstrating strong performance on large-scale nonconvex problems.

\subsection{Notation}

Let $\AA^\dagger$ denote the Moore-Penrose pseudoinverse and $\AA^\ddagger = (\AA^\dagger)^\transpose = (\AA^\transpose)^\dagger$. For a square matrix $\AA$ and vector $\bb$, we denote the $t$-th Krylov subspace as $\sK_t(\AA, \bb) \defeq \Span\{\bb, \AA \bb, \ldots,\AA^{t-1}\bb \}$. Define the grade of $\AA$ with respect to $\bb$ as the integer $g(\bb, \AA)$ such that 
\begin{align*}
    \dim(\sK_t(\AA, \bb)) = \min\{t, g(\bb, \AA)\}.
\end{align*}
We denote the null space of the matrix $\AA$ by $\Null(\AA)$ and the range space by $\Range(\AA)$. For a subspace, $\sV$, we denote the orthogonal projector onto $\sV$ by $\PP_\sV$, e.g., $\PPNJxx$ denotes the orthogonal projector onto the null space of $\JJ(\xx)$. The norm $\vnorm{\cdot}$ denotes the Euclidean norm for a vector input and the operator norm for a matrix. For a positive definite matrix $\AA$, we denote the $\AA$-inner product as $\dotprod{\cdot, \cdot}_\AA$. Meanwhile, for positive semidefinite $\AA$ we denote the $\AA$-seminorm as $\vnorm{\cdot}_\AA = \sqrt{\dotprod{\cdot, \AA \cdot}}$.

When context is clear, we will often suppress the dependence of certain functions on their arguments, for example, we will often write $\JJ(\xx)$ as $\JJ$ and $\JJ(\xx_k)$ as $\JJ_k$. In general, we use bracketed superscripts to denote quantities arising from inner iterations and subscripts to denote quantities arising from outer iterations.

A key tool in our analysis is the economy SVD of $\JJ$, which we write as
\begin{align*}
    \JJ = \UU \bSigma \bOmega^\transpose, \tageq\label{eqn:SVD of jacobian}
\end{align*}
where $r= \Rank{(\JJ)}$, $\UU \in \real^{c \times r}$, $\bOmega \in \real^{d \times r}$ are orthogonal matrices, and $\bSigma = \diag(\sigma_1, \ldots, \sigma_r) \in \real^{r \times r}$ is a diagonal matrix of nonzero singular values $0<\sigma_r\leq \ldots\leq \sigma_1$. 
Finally, for a nonnegative integer $K$, we use square brackets to denote the index set of integers up to $K$, i.e.,
\[
[K] \defeq \{0,1,\ldots,K\}.
\]

\subsection{Literature Review} \label{sec:literature review}

SQP methods are among the most effective approaches for solving nonlinear programs with equality constraints \citep{boggsSQP1995,nocedal_numerical_2006}. Classical SQP methods \citep{hanGloballyConvergentMethod1977,hanSuperlinearlyConvergentVariable1976,powellFastNonlinearConstrained1978,palomaresSuperlinearlyConvergentQuasinewton1976} compute a primal update by solving a quadratic subproblem based on the linearized constraints and a second-order model of the Lagrangian. Global convergence is obtained either through trust-region frameworks \citep{connTrustRegionMethods2000}, a filter method \citep{fletcherNonlinearProgrammingPenalty2002}, or through merit functions combined with line-search procedures \citep{nocedal_numerical_2006}. While trust-region SQP methods have received significant attention due to their inherent ability to handle nonconvexity, line-search variants remain widely used in practice due to their conceptual simplicity and strong empirical performance \citep{nocedal_numerical_2006,curtisWorstCaseComplexitySQP2022}.

A number of works have considered SQP methods in the large-scale setting, where the Newton system arising in the step computation is solved only approximately; see, for example, \citep{byrdInexactSQPMethod2008,byrdInexactNonconvexNewtonMethod2010,curtisMatrixFreeAlgorithmEquality2010,laleeImplementationAlgorithmLargeScale1998,gillSNOPTSQPAlgorithm2005} and more recent developments in \citep{curtis_inexact_2021,berahasModifiedLineSearch2024}. In these approaches, the inner linear solver is typically treated as a ``black-box" and is assumed to produce sufficiently accurate approximations of the Newton step. The inner iterations are therefore terminated once the update satisfies a number of termination conditions derived from models of the objective and constraints. As a result, the termination conditions for the inner solver are often unintuitive and difficult to implement. 

In practical implementations the linear systems that arise from the SQP subproblems are usually solved using Krylov subspace methods such as CG \cite{curtis_inexact_2021}, GMRES \cite{byrdInexactSQPMethod2008,byrdInexactNonconvexNewtonMethod2010}, or MINRES \cite{curtis_inexact_2021,berahasOptimisticNoiseAwareSequential2025,curtisMatrixFreeAlgorithmEquality2010}. However, the theoretical analyses of these SQP frameworks generally do not exploit the specific properties of these iterative solvers, treating them instead as generic linear algebra routines. Consequently, the behavior of the inner Krylov iterations remains largely disconnected from the analysis of the outer optimization algorithm.

In recent years, there has been increasing interest in establishing worst-case complexity guarantees for optimization algorithms. Such results provide theoretical bounds on the number of iterations or function evaluations required to reach approximate optimality and have become an important tool for comparing algorithms; see, for example, the foundational results of \citet{nesterovIntroductoryLecturesConvex2004,nesterovLecturesConvexOptimization2018} and the modern survey of evaluation complexity results in \citet{cartisEvaluationComplexityAlgorithms2022}. Motivated by these advances, a growing body of work has sought to establish similar guarantees for constrained optimization methods.

For general nonlinear programs, several frameworks have been shown to attain worst-case complexity guarantees. These include adaptive regularization and two-phase algorithms \citep[Chapters~6-7]{cartisEvaluationComplexityAlgorithms2022}, as well as related approaches analyzed in \citet{birginEvaluationComplexityNonlinear2016}. While these methods enjoy strong theoretical guarantees, they often require the solution of relatively expensive subproblems at each iteration, and in the case of two-phase approaches, treat feasibility and optimality separately rather than addressing the objective and constraints simultaneously. Other approaches with complexity guarantees include inexact restoration methods \citep{buenoComplexityInexactRestoration2020}, which separate feasibility and optimality phases, augmented Lagrangian frameworks \citep{xie_complexity_2020,grapigliaComplexityAugmentedLagrangian2021}, and a ``trust-funnel" method \citep{curtisComplexityAnalysisTrust2018} inspired by trust-region SQP. 

In contrast, complexity results for line-search SQP methods are relatively scarce, with most existing analyses focusing on asymptotic convergence properties. Only recently have worst-case complexity guarantees begun to appear for line-search-style SQP frameworks \citep{curtisWorstCaseComplexitySQP2022,berahasSequentialQuadraticProgramming2025,oneill_two_2024}. In the deterministic setting, these methods obtain complexity guarantees of order $\sO(\max\{\epsp^{-2},\epsc^{-1}\})$ for identifying an approximate first-order point of the form \cref{eqn:epsilon KKT stationary point}. The dependence on $\epsp$ matches the classical lower bounds for first-order methods and Newton's method applied to unconstrained nonconvex optimization \citep{cartisEvaluationComplexityAlgorithms2022}. 
To the best of our knowledge, no existing result provides a worst-case complexity guarantee for an SQP-type method that simultaneously allows for inexact step computation and nonconvex Lagrangian curvature. This work provides such a result.

The recent sequential \emph{cubic} programming method of \citet{dimou2026sequential} improves the deterministic first-order stationarity rate to $\sO(\max\{\epsp^{-3/2},\epsc^{-1}\})$
and also provides guarantees for convergence to approximate second-order stationary points of \cref{eqn:equality constrained problem}. This improved dependence on $\epsp$ is obtained by replacing the quadratic SQP subproblem with a cubic-regularized tangential subproblem, together with second-order correction steps. However, this comes at the cost of solving a more challenging cubic subproblem subject to linear constraints at each iteration. For a detailed survey of complexity rates in equality-constrained optimization, we refer the reader to the literature review in \citet{dimou2026sequential}.

Another recent trend in the SQP literature is consideration of the stochastic setting, where the objective oracle provides noisy estimates of function, gradient and Hessian values \citep{naAdaptiveStochasticSequential2022,berahasRetrospectiveApproximationSequential2025,berahasAdaptiveSamplingSequential2023,berahasOptimisticNoiseAwareSequential2025,naStatisticalInferenceConstrained2025,curtis_inexact_2021,berahasSequentialQuadraticOptimization2020,berahasModifiedLineSearch2024,berahasSequentialQuadraticProgramming2025}. Related complexity guarantees have also been developed for stochastic SQP methods \citep{oneill_two_2024}, where the use of noisy objective information leads to different stationarity measures and complexity rates.
Our work focuses on the deterministic setting, we leave the extension of our algorithm to the stochastic case for future work.

\section{Basic Building Blocks of Our Framework} \label{sec:building blocks}

In this section, we present the basic building blocks of the algorithmic framework underlying our proposed methods.
We begin in \cref{sec:step derivation} by reformulating the primal-dual Newton system via a decomposition of the step into tangent and normal components as a structured least-squares problem.
To compute the tangent component, we employ the CR algorithm. In particular, \cref{sec:CR algorithm background} introduces the (projected) CR method, which serves as the core inner solver for the tangent-space subproblem.
Finally, \cref{sec:merit} reviews the merit function and the line-search-based globalization strategy we consider in this work.

\subsection{Primal Step Derivation} \label{sec:step derivation}
Our starting point is a reformulation of the KKT system \cref{eqn:KKT-system} as the least-squares problem
\begin{align*}
    \min_{\pp,\qq}
    \;
    \vnorm{
        \begin{pmatrix}
            \LHess & \JJ^\transpose \\
            \JJ & 0
        \end{pmatrix}
        \vvec{\pp}{\qq}
        +
        \vvec{\Lgrad}{\cc}
    }^2
    =
    \min_{\pp,\qq}
    \;
    \vnorm{
        \LHess \pp + \JJ^\transpose \qq + \Lgrad
    }^2
    +
    \vnorm{
        \JJ\pp + \cc
    }^2.
    \tageq\label{eqn:KKT-least-squares}
\end{align*}
Unlike the linear system formulation in \cref{eqn:KKT-system}, which need not admit a solution, the least-squares problem \cref{eqn:KKT-least-squares} always admits at least one minimizer, regardless of whether \cref{cond:KKT-Regularity} holds.
Eliminating the dual variable $\qq$ yields a reduced primal-only formulation. Indeed, the optimal multiplier satisfies
\[
    \qq^\star
    =
    -\JJ^\ddag(\LHess\pp + \Lgrad).
\]
Substituting this expression into \cref{eqn:KKT-least-squares}, and using the symmetry of $\JJ^\dagger \JJ$, gives
\begin{align*}
    \LHess\pp
    -
    \JJ^\transpose\bigl(\JJ^\ddag(\LHess\pp + \Lgrad)\bigr)
    +
    \Lgrad
    &=
    (\eye - \JJ^\dagger \JJ)(\LHess\pp + \Lgrad).
\end{align*}
Recognizing
\[
    \PP_{\Null(\JJ)}
    =
    \eye - \JJ^\dagger \JJ,
\]
as the orthogonal projector onto $\Null(\JJ)$, the least-squares problem reduces to
\begin{align*}
    \min_{\pp}
    \;
    \vnorm{
        \PP_{\Null(\JJ)}
        (\LHess\pp + \Lgrad)
    }^2
    +
    \vnorm{
        \JJ\pp + \cc
    }^2.
\end{align*}

Motivated by classical SQP approaches \citep{nocedal_numerical_2006,byrdTrustRegionAlgorithm1987,omojokun1989trust}, we decompose the primal step into tangent and normal components, $\pp = \ww + \vv$,
where $\ww \in \Null(\JJ)$ and $\vv \in \Range(\JJ^\transpose)$. 
Substituting this decomposition into the reduced least-squares formulation yields
\begin{align*}
    \min_{\substack{
        \vv \in \Range(\JJ^\transpose) \\
        \ww \in \Null(\JJ)
    }}
    \;
    \vnorm{
        \PP_{\Null(\JJ)}
        \bigl(
            \LHess(\vv+\ww) + \Lgrad
        \bigr)
    }^2
    +
    \vnorm{
        \JJ \vv + \cc
    }^2.
\end{align*}
Following \cite[Chapter 18]{nocedal_numerical_2006}, we approximately decouple the tangent and normal components by neglecting the contribution of $\vv$ in the first term. This leads to the pair of subproblems for the normal and tangent components
\begin{subequations}
\label{eqn:exact formulation}
\begin{align}
    \vv^\star
    &=
    \argmin_{\vv \in \Range(\JJ^\transpose)}
    \;
    \vnorm{
        \JJ \vv + \cc
    },
    \label{eqn:normal component problem}
    \\
    \ww^\star
    &=
    \argmin_{\ww \in \Null(\JJ)}
    \;
    \vnorm{
        \PP_{\Null(\JJ)}
        \left(
            \LHess \ww + \Lgrad
        \right)
    }.
    \label{eqn:tangent component problem}
\end{align}
\end{subequations}
The normal and tangent subproblem pair in \cref{eqn:exact formulation} forms the foundation of our approach. When $c \ll d$, it is computationally feasible to solve \cref{eqn:normal component problem} directly via factorization of the constraint Jacobian. Importantly, this factorization simultaneously reveals the structure of $\Range(\JJ^\transpose)$, which in turn enables efficient construction of the null space projector $\PP_{\Null(\JJ)}$ required in the tangent subproblem; see \cref{sec:direct algorithm statement} for details. When $c \lesssim d$, direct factorization may no longer be computationally feasible; we discuss this setting in \cref{sec:indirect normal approach}. For now, we focus on the case where $c \ll d$ and $\PPNJ$ is available.
In contrast to the normal problem \cref{eqn:normal component problem}, \cref{eqn:tangent component problem} corresponds to a large-scale least-squares problem that is generally too expensive to solve exactly in our setting. Consequently, the tangent subproblem must be reformulated in a form amenable to efficient inexact solution by a Krylov subspace method. Developing this reformulation is the focus of the remainder of this section.

We begin by observing that, since $\PP_{\Null(\JJ)}$ is an orthogonal projector, \cref{eqn:tangent component problem} may be interpreted as a positive semidefinite \emph{left preconditioning} of the symmetric system with matrix $\LHess$ and the right-hand side vector $- \Lgrad$.
Following \citet{liu_obtaining_2024}, the associated Krylov subspace formulation for an inexact tangent step can be written as
\begin{align*}
    \ddtt
    =
    \argmin_{\dd}
    \;
    \vnorm{
        \LHess \dd + \Lgrad
    }_{\PP_{\Null(\JJ)}}
    \qquad
    \text{s.t.}
    \qquad
    \dd \in \sK_t\!\left(
        \PP_{\Null(\JJ)}\LHess,
        \PP_{\Null(\JJ)}\Lgrad
    \right),
\end{align*}
where we distinguish the inexact inner iterate, $\ddtt$, from the outer tangent update, $\ww$.
Although this formulation characterizes the desired inexact tangent step, it does not yet coincide with the Krylov subproblem associated with a standard iterative solver. To obtain such a formulation, we further exploit the structure of the null space projector.
Specifically, let
\[
    \PP_{\Null(\JJ)}
    =
    \bSS \bSS^\transpose,
\]
where $\bSS \in \real^{d \times (d-r)}$ has orthonormal columns spanning $\Null(\JJ)$. Then
\[
    \sK_t\!\left(
        \PP_{\Null(\JJ)}\LHess,
        \PP_{\Null(\JJ)}\Lgrad
    \right)
    =
    \bSS
    \sK_t\!\left(
        \bSS^\transpose \LHess \bSS,
        \bSS^\transpose \Lgrad
    \right),
\]
and hence the projected problem is equivalent to the reduced subproblem
\begin{align*}
    \tddtt
    =
    \argmin_{\tdd}
    \;
    \vnorm{
        \tHH \tdd + \tbgg
    }, \qquad \text{s.t.} \qquad \tdd \in \sK_t(\tHH,\tbgg),
    \tageq\label{eqn:inexact tangent component CR subproblem}
\end{align*}
where
\[
    \tHH \defeq \bSS^\transpose \LHess \bSS,
    \qquad
    \tbgg \defeq \bSS^\transpose \Lgrad,
\]
and the full-space iterate is recovered via
\[
    \ddtt = \bSS \tddtt.
\]
The problem in \cref{eqn:inexact tangent component CR subproblem} coincides precisely with the Krylov subproblem underlying the MINRES method \citep{liu_minres_2022,choiMINRESQLPKrylovSubspace2011,paigeSolutionSparseIndefinite1975}, which is applicable to the indefinite linear systems that arise when $\HHt$ fails to be positive definite. Consequently, the tangent step may be computed inexactly using MINRES, which, as discussed in the following section, is effectively equivalent to the CR algorithm \citep{limConjugateDirectionMethods2024} in our setting. 
This observation is central to our framework: it allows us to retain the algorithmic simplicity of CR while simultaneously leveraging powerful theoretical properties of MINRES, which we review in the following section.

\subsection{The Workhorse: Projected Conjugate Residual} \label{sec:CR algorithm background}

The preceding discussion shows that the tangent computation in \cref{eqn:inexact tangent component CR subproblem} can be interpreted as a MINRES Krylov subproblem, while also suggesting a simpler CR implementation. We now make this connection precise and describe the projected CR method used as the inner solver throughout this paper. The key idea is that MINRES provides the residual-minimization viewpoint and curvature-detection properties needed for the analysis, whereas CR provides an equivalent and more transparent implementation in the optimization setting considered here.

To ensure that our results are also applicable to \cref{sec:indirect normal approach}, where the projection onto the null space may be an approximation, we study a slightly more general formulation of \cref{eqn:inexact tangent component CR subproblem}. In particular, let $\HH \in \real^{d \times d}$ be symmetric, $\bgg \in \real^d$, and $\PP \in \real^{d \times d}$ be any orthogonal projection matrix. For analytical purposes only, we write
\[
    \PP = \ZZ \ZZ^\transpose,
\]
where $\ZZ \in \real^{d \times k}$ has orthonormal columns spanning $\Range(\PP)$. Crucially, access to this factorization is not required to perform the projected CR iterations; see \cref{remark:projected CR without factorisation}. Define\footnote{In this section we use the bar, ``$\bar{\cdot}$'', when we need to distinguish reduced quantities from their unreduced counterparts.} the reduced matrix and vector as
\[
    \HHb \defeq \ZZ^\transpose \HH \ZZ,
    \qquad
    \bbgg \defeq \ZZ^\transpose \bgg, \tageq\label{eqn:projected CR reduced Hessian and gradient}
\]
and the corresponding residual minimization problem
\begin{align*}
    \bddtt
    =
    \argmin_{\bdd}
    \;
    \frac12\vnorm{\HHb \bdd + \bbgg}^2,
    \qquad
    \text{s.t.}
    \qquad
    \bdd \in \sK_t(\HHb,\bbgg).
    \tageq\label{eqn:projected CR formulation}
\end{align*}
The application of this formulation to the tangent problem in \cref{eqn:inexact tangent component CR subproblem} with $\ZZ = \bSS$, $\HH = \LHess$, etc., is immediate.

The remainder of this section proceeds as follows. We first recall the formulation of MINRES and the limited positive curvature detection property that makes MINRES useful in nonconvex Newton-type methods. We then formalize the equivalence between MINRES and CR that was used to motivate our implementation. Finally, we state the projected CR method and study the key properties of the algorithm required in the outer convergence analysis.

\paragraph{MINRES Review.} 

The standard MINRES method is recovered from \cref{eqn:projected CR formulation} as the special case $\ZZ = \eye$, $\HHb = \HH$, and $\bbgg = \bgg$. In this case, \cref{eqn:projected CR formulation} is the unreduced MINRES problem
\begin{align*}
    \ddtt
    =
    \argmin_{\dd}
    \;
    \frac12\vnorm{\HH \dd + \bgg}^2,
    \qquad
    \text{s.t.}
    \qquad
    \dd \in \sK_t(\HH,\bgg).
    \tageq\label{eqn:minres subproblem}
\end{align*}
When $\HH$ and $\bgg$ are the Hessian and gradient, respectively, of an objective function, \cref{eqn:minres subproblem} corresponds to the subproblem of an inexact Newton method; see \citep{liu_newton-mr_2023,lim_complexity_2024}. Recent work has highlighted several properties that make MINRES particularly effective as an inner solver for optimization \citep{liu_minres_2022}. A key property is \emph{limited positive curvature} detection. Specifically, suppose that for some $\sigma > 0$, the residuals generated by MINRES
\[
    \rrii \defeq -\HH \ddii - \bgg,
\]
satisfy
\[
    \dotprod{\rrii,\HH \rrii}
    >
    (d+1)\sigma \vnorm{\rrii}^2,
    \qquad
    i = 1,\ldots,t-1.
\]
Then $\HH$ is certified to be $\sigma$-strongly positive definite when restricted to $\sK_t(\HH,\bgg)$ \citep[Theorem 3.3]{liu_minres_2022}. Intuitively, as long as the LPC test is not triggered, the MINRES iterates, $\ddtt \in \sK_t(\HH,\bgg)$, behave as though the underlying optimization problem were locally convex along the generated subspace. On the other hand, when the LPC condition is triggered, the residual, $\rrtt$, itself can serve as a good search direction. Therefore, when MINRES is utilized as an inner solver for inexact Newton methods, reliable search directions can be generated even when $\HH$ is indefinite, without requiring distortionary regularization. See the works in \cite{smee_inexact_2024,lim_complexity_2024,liu_newton-mr_2023} for applications of MINRES with curvature detection to second-order optimization. 

Despite these advantages, MINRES is often viewed as having a less transparent implementation than the simpler CG and CR algorithms. Fortunately, MINRES and CR are essentially equivalent in the optimization setting. The connection with CR is classical when $\HH$ is positive definite. In this case, the CR subproblem associated with the linear system $\HH \dd = -\bgg$ is
\begin{align*}
    \ddtt
    =
    \argmin_{\dd}
    \;
    \frac12
    \dotprod{\dd,\HH\dd}_{\HH}
    +
    \dotprod{\dd,\bgg}_{\HH},
    \qquad
    \text{s.t.}
    \qquad
    \dd \in \sK_t(\HH,\bgg).
    \tageq\label{eqn:CR quadratic formulation}
\end{align*}
The objectives in \cref{eqn:CR quadratic formulation} and \cref{eqn:minres subproblem} differ only by a constant term $\tfrac12\vnorm{\bgg}^2$ and hence MINRES and CR produce equivalent iterates. However, recent work has revealed that this equivalence extends far beyond the positive definite setting. In particular, \citet[Theorem 5]{limConjugateDirectionMethods2024} show that CR and MINRES generate identical iterates until an \emph{unlucky breakdown} where
\[
    \dotprod{\rrtt,\HH \rrtt} = 0.
\]
In optimization applications, the inner iteration is typically terminated as soon as
\[
    \dotprod{\rrtt,\HH \rrtt} \leq 0,
\]
due to the curvature detection procedure outlined previously. Thus, for the purposes of Newton-type optimization methods, CR and MINRES may be regarded as practically equivalent. This observation has motivated a ``second generation'' of Newton-MR methods based on the simpler CR implementation \citep{lim_faithful-newton_2025}. For the remainder of this work, we therefore focus on the equivalent CR interpretation of \cref{eqn:projected CR formulation}.

The preceding discussion applies to the unreduced pair $(\HH,\bgg)$, but it applies verbatim to the reduced pair $(\HHb,\bbgg)$ in \cref{eqn:projected CR reduced Hessian and gradient}. Thus, \cref{eqn:projected CR formulation} may be interpreted either as a reduced MINRES subproblem or, equivalently within the optimization regime described above, as a reduced CR subproblem. The projected method is obtained by running this CR iteration in the reduced coordinates and then lifting either the approximate solution or the detected curvature direction back to the original space through $\ZZ$.

\paragraph{Projected Conjugate Residual.} We now give the details of the projected CR method associated with \cref{eqn:projected CR formulation}. Let
\[
    \brrtt \defeq -\HHb\bddtt - \bbgg,
\]
denote the reduced residual. After verifying that the positive curvature certificate is valid for $\brrzero$ (if not, $\rrzero = \ZZ \brrzero$ is returned), the basic projected CR method proceeds until either the termination condition
\begin{align*}
    \vnorm{\HHb \brrtt}
    \leq
    \tau \vnorm{\HHb \bddtt},
    \tageq\label{eqn:tangent component termination condition}
\end{align*}
is satisfied or limited positive curvature is detected in the reduced residuals, i.e.,
\begin{align*}
    \dotprod{\brrtt, \HHb \brrtt}
    \leq
    (d+1)\sigma \vnorm{\brrtt}^2.
    \tageq\label{eqn:LPC condition}
\end{align*} 
If \cref{eqn:tangent component termination condition} holds, then $\bddtt$ is accepted as an inexact solution of the reduced tangent system. The corresponding lifted direction
\[
    \ddtt = \ZZ \bddtt,
\]
is returned with $\FLAG{} = \SOL{}$. Alternatively, if \cref{eqn:LPC condition} is triggered first, the positive-curvature certificate no longer holds for the next iteration, and we terminate. In this case the lifted residual
\[
    \rrtt = \ZZ \brrtt,
\]
is returned as the search direction with $\FLAG{} = \LPC{}$. In either case, the returned direction lies in $\Range(\PP)$ (i.e., $\Null(\JJ)$ if $\PP = \PPNJ$) by construction. This termination scheme is detailed in \Cref{term:direct}.

We state the projected CR iteration in \cref{alg:conjugate residual}. The CR update is written in a compact form using both $\HHb\yyttm$ and $\HHb\brrttm$. In practice, however, $\HHb\yy_t$ can be updated as
\[
    \HHb\yytt = \HHb\brrtt + \beta_{t-1}\HHb\yyttm,
\]
so that, as with CG and MINRES, each CR iteration requires only one Hessian-vector product. We leave the termination block in \cref{alg:conjugate residual} flexible, since we vary the termination condition for the direct (\Cref{term:direct}) and indirect (\Cref{term:indirect}) cases.

\begin{algorithm}[ht]
\renewcommand{\baselinestretch}{1.25}\selectfont 
\begin{algorithmic}[1]
    \Require{$\HH \in \real^{d \times d}$, $\bgg \in \real^d$, an orthogonal projection matrix $\PP \in \real^{d\times d}$, parameters for \Cref{term:direct} or \Cref{term:indirect}.}
    \State $\PP = \ZZ \ZZ^\transpose$, $\HHb = \ZZ^\transpose \HH \ZZ$ and $\bbgg = \ZZ^\transpose \bgg$. \Comment{Factorization is not required in practice; see \cref{remark:projected CR without factorisation}.}
    \State Set $\bddzero = 0$, $\brrzero = \yyzero = -\bbgg$ and $t = 0$.
    \If{\cref{eqn:LPC condition} holds} \Comment{Certify the first Krylov subspace.}
        \State \Return{$(\rrzero, \FLAG=\LPC{})$} \label{line:early LPC termination}
    \EndIf
    \While{True} 
    \LeftComment{CR iteration.}
        \State $t \gets t + 1$
        \State $\zeta_{t-1} = \frac{\dotprod{\brrttm, \HHb \brrttm}}{\vnorm{\HHb \yyttm}^2}$
        \State $\bddtt = \bddttm + \zeta_{t-1} \yyttm$
        \State $\brrtt = \brrttm - \zeta_{t-1} \HHb \yyttm$
        \State $\beta_{t-1} = \frac{\dotprod{\brrtt, \HHb \brrtt}}{\dotprod{\brrttm, \HHb \brrttm}} $
        \State $\yytt = \brrtt + \beta_{t-1} \yyttm$
        \State Check termination via $\left\{
\begin{aligned}
&\text{\Cref{term:direct}} && \text{if ``direct'' (see \cref{sec:direct algorithm statement})},\\
&\text{\Cref{term:indirect}} && \text{if ``indirect'' (see \cref{sec:indirect normal approach}).}
\end{aligned}
\right.$
    \EndWhile
\end{algorithmic}
\caption{Projected Conjugate Residual}
\label{alg:conjugate residual}
\end{algorithm}

\begin{terminationblockalg}[ht]{Direct 
Termination Block}[term:direct]
    \Require LPC tolerance $\sigma > 0$ and termination tolerance $\tautan>0$.
    \If{\cref{eqn:tangent component termination condition} holds with $\tau = \tautan$}
        \State \Return{$(\ddtt, \FLAG=\SOL{})$} 
    \ElsIf{\cref{eqn:LPC condition} holds}
        \State \Return{$(\rrtt, \FLAG=\LPC{})$}
    \Else
        \State \textbf{continue}
    \EndIf 
\end{terminationblockalg}

We now provide some remarks on the computational aspects of \cref{alg:conjugate residual}.

\begin{remark} \label{remark:projected CR termination condition}
    The termination condition in \cref{eqn:tangent component termination condition} is non-standard but particularly well suited to singular settings. In particular, the normal equation residual associated with \cref{eqn:projected CR formulation}, $\HHb \brrtt$, converges to zero even when $\HHb$ is indefinite, while $\vnorm{\HHb \bddtt}$ increases monotonically to $\vnorm{\HHb\HHb^\dagger\bbgg}$ provided that negative curvature is not detected; see \citep{liu_newton-mr_2023}. By contrast, a more conventional stopping rule based on the relative residual, $\vnorm{\brrtt} \leq \tau \vnorm{\bbgg}$, may be unreliable when $\bbgg \notin \Range(\HHb)$. In particular, $\vnorm{\brrtt}$ can remain bounded away from zero regardless of the quality of $\bddtt$, so that the relative residual tolerance may never be satisfied for small choices of $\tau$.
\end{remark}

\begin{remark} \label{remark:projected CR without factorisation}
    The formulation of projected CR in \cref{eqn:projected CR formulation} appears to require access to a factorization of the projection matrix of the form $\PP = \ZZ \ZZ^\transpose$. 
    However, in both the direct and indirect cases, the projector is constructed as $\PP = \eye - \bOmega \bOmega^\transpose$, where $\Range(\bOmega) \subseteq \Range(\JJ^\transpose)$, 
    and explicitly forming a basis matrix $\ZZ$ for $\Range(\PP)$ would incur substantial additional computational cost.
    
    Fortunately, a careful inspection of the projected CR recursions (\cref{alg:conjugate residual}) reveals that the method can be implemented without ever explicitly constructing $\ZZ$. In particular, by appropriately tracking auxiliary quantities throughout the iterations, all CR updates may be expressed solely in terms of 
    \begin{itemize}
        \item a projection-product oracle $\vv \mapsto \PP \vv$, and
        \item a Hessian-vector product oracle $\vv \mapsto \HH \vv$.
    \end{itemize}
    These operations are sufficient for performing the projected CR updates, evaluating the termination criteria, detecting LPC, and constructing the final search directions. Moreover, with a small amount of auxiliary storage, each projected CR iteration requires only a single application of both the projection-product oracle and the Hessian-vector product oracle; see \cref{apx:CR algorithm details} for details. 
    
    Both the Hessian-vector product oracle and projection-product oracle calls can be performed efficiently, and without excessive storage costs. In particular, Hessian-vector products can be computed without forming the Hessian using automatic differentiation \citep{blondelElementsDifferentiableProgramming2024}. Meanwhile, the projection-products can be computed without forming $\PP$ explicitly via
    \[
    \PP \vv
    =
    \vv-\bOmega(\bOmega^\transpose\vv).
    \]
    Consequently, the resulting tangent-space computation remains computationally efficient in the large-scale setting. 
\end{remark}

\bigskip

Finally, the following lemma shows that the favorable descent and curvature-detection properties of CR are preserved by projected CR. In particular, irrespective of the value of $\FLAG$, the returned direction satisfies the required descent properties with respect to the unreduced quantities $\HH$ and $\bgg$. This fact is central to the global convergence analysis developed in the subsequent sections. We state \cref{lemma:projected CR properties} in a general form so that it remains applicable to the modified termination conditions considered in \cref{sec:indirect normal approach}. However, the applicability of \cref{lemma:projected CR properties} to \cref{alg:conjugate residual} with \Cref{term:direct} is immediate.

\begin{lemma}\label{lemma:projected CR properties}
Let $\vnorm{\HHb} \leq \kappa < \infty$,  $\tau>0$, and $g$ denote the grade of $\bbgg$ with respect to $\HHb$. Let $\bddii$ and $\brrii$ be the iterates and residuals produced by projected CR with LPC detection. Suppose that 
\[
\dotprod{\brrii, \HHb \brrii } > (d+1)\sigma \vnorm{\brrii}^2, \quad i=0,\ldots,t,
\]
for some $0 \leq t \leq g-1$, i.e., the LPC condition has not been triggered through iteration $t$. Then for every $s=1, \ldots, t+1$, we have 
    \begin{align*}
        \dotprod{\bgg, \ddss} &\leq -\dotprod{\ddss, \HH \ddss} \leq -\sigma \vnorm{\ddss}^2, \tageq \label{eqn:SOL second order descent}\\ 
        \vnorm{\ddss} &\geq  \frac{\sigma}{\kappa^2} \vnorm{\bbgg}. \tageq\label{eqn:SOL step length}
    \end{align*}
    On the other hand, suppose LPC is detected at iteration $t$, either with $t = 0$ or before \cref{eqn:tangent component termination condition} is satisfied. Then, 
    \begin{align*}
        \dotprod{\bgg, \rrtt} &= - \vnorm{\rrtt}^2 ,\tageq\label{eqn:LPC step descent}\\ 
        \vnorm{\rrtt} &\geq \frac{\tau}{\sqrt{\tau^2 + \kappa^2}}\vnorm{\bbgg}. \tageq\label{eqn:LPC step length} 
    \end{align*}
\end{lemma}
\begin{proof}
We first establish \cref{eqn:SOL second order descent,eqn:SOL step length}. Since  $\dotprod{\brrii, \HHb \brrii} > (d+1)\sigma \vnorm{\brrii}^2 $ holds for $i=0,\ldots,t$, \citep[Theorem 5]{limConjugateDirectionMethods2024} implies that the CR iterates $\bddss$ are equivalent to that of MINRES for $s=1,\ldots,t+1$. Moreover, by \citep[Theorem 3.8]{liu_minres_2022} we have 
\begin{align*}
    \dotprod{\bgg, \ddss} + \dotprod{\ddss, \HH \ddss} 
    &= \dotprod{\bgg, \ZZ \bddss} + \dotprod{\ZZ\bddss, \HH \ZZ\bddss} \\
    &= \dotprod{\bbgg,\bddss} + \dotprod{ \bddss, \HHb\bddss} \leq 0.
\end{align*}
Moreover, since $\bddss \in \sK_{t+1}(\HHb,\bbgg)$ for $s = 1,\ldots,t+1$, \citep[Lemma B.1]{smee_inexact_2024} gives  
\begin{align*}
    \dotprod{ \bddss, \HHb\bddss} \geq \sigma \vnorm{\bddss}^2.
\end{align*}
Combining this with the previous inequality we have 
\begin{align*}
    \dotprod{\bgg, \ddss} = \dotprod{\bgg, \ZZ\bddss} = \dotprod{\bbgg, \bddss} \leq - \dotprod{ \bddss, \HHb\bddss} \leq - \sigma \vnorm{\bddss}^2 = - \sigma \vnorm{\ddss}^2,
\end{align*}
where in the final inequality we use the fact that $\ZZ$ is an orthogonal matrix, i.e., $\vnorm{\ddss} = \vnorm{\ZZ \bddss} = \vnorm{\bddss}$. Together these equalities establish \cref{eqn:SOL second order descent}. Next, suppose that the LPC condition is not triggered at $t = 0$. By direct calculation and the curvature test for $\brrzero = - \bbgg$ we have 
\begin{align*}
    \vnorm{\ddone}= \vnorm{\bddone} = \frac{\dotprod{\bbgg, \HHb \bbgg}}{\vnorm{\HHb\bbgg}^2} \vnorm{\bbgg} \geq \frac{\sigma}{\kappa^2} \vnorm{\bbgg}.
\end{align*}
This establishes \cref{eqn:SOL step length} for $t=0$. The result for $t>0$ follows from the monotonicity of the MINRES iterates \citep[Theorem 3.11]{liu_minres_2022}, that is, as long as LPC remains undetected through iteration $t$, we have 
\begin{align*}
    \vnorm{\bddttp} \geq \vnorm{\bddtt} \geq \ldots \geq \vnorm{\bddone} = \frac{\sigma}{\kappa^2} \vnorm{\bbgg}.
\end{align*}
The result in \cref{eqn:SOL step length} follows from $\ZZ$ being orthogonal. 

Next, we establish \cref{eqn:LPC step descent,eqn:LPC step length}. At $t=0$, $\rrzero = -\ZZ\bbgg$. So by the orthogonality of $\ZZ$ we have
\begin{align*}
    \vnorm{\rrzero} = \vnorm{\ZZ \bbgg} = \vnorm{\bbgg},
\end{align*}
and 
\begin{align*}
    \dotprod{\bgg, \rrzero} = -\dotprod{\bgg, \ZZ \bbgg} = - \vnorm{\bbgg}^2 = -\vnorm{\rrzero}^2.
\end{align*}
Therefore, assume $t>0$. If LPC is first detected at iteration $t$, then in all prior iterations $i=0, \ldots, t-1$ we must have $\dotprod{\brrii, \HHb \brrii} \neq 0$. Therefore, by \citep[Theorem 5]{limConjugateDirectionMethods2024} the CR iterate $\bddtt$ and the corresponding residual $\brrtt$ are equivalent to those from MINRES. By \citep[Lemma 3.1]{liu_minres_2022}, the residuals of MINRES satisfy $\dotprod{\brrtt, \bbgg} = -\vnorm{\brrtt}^2$. We therefore obtain \cref{eqn:LPC step descent} via
\begin{align*}
    \dotprod{\bgg, \rrtt}= \dotprod{\bgg, \ZZ\brrtt} = \dotprod{\bbgg, \brrtt} = -\vnorm{\brrtt}^2 = -\vnorm{\rrtt}^2, 
\end{align*}
where in the fourth equality we apply orthogonality of $\ZZ$. 

For \cref{eqn:LPC step length}, the failure of \cref{eqn:tangent component termination condition} to hold implies
\begin{align*}
    \vnorm{\HHb \brrtt}^2 
    \geq \tau^2 \vnorm{\HHb \bddtt}^2
    = \tau^2 \vnorm{\brrtt + \bbgg}^2
    = \tau^2\left(\vnorm{\bbgg}^2 - \vnorm{\brrtt}^2\right),
\end{align*}
where the second equality follows from $\HHb\bddtt+\bbgg=-\brrtt$, and the third equality uses the identity $\dotprod{\brrtt,\bbgg}=-\vnorm{\brrtt}^2$ established in \cref{eqn:LPC step descent}.
Combining this with
\[
    \vnorm{\HHb \brrtt}^2
    \leq
    \kappa^2 \vnorm{\brrtt}^2,
\]
and rearranging gives
\[
    (\tau^2+\kappa^2)\vnorm{\brrtt}^2
    \geq
    \tau^2\vnorm{\bbgg}^2.
\]
Therefore,
\begin{align*}
    \vnorm{\rrtt}
    =
    \vnorm{\brrtt}
    \geq
    \frac{\tau}{\sqrt{\tau^2+\kappa^2}}
    \vnorm{\bbgg},
\end{align*}
where the equality follows from the orthonormality of $\ZZ$.
\end{proof}

\subsection{Merit Function and Line Search}
\label{sec:merit}
We globalize our methods by employing an Armijo line search \citep{armijo_line_search_1966} on an $\ell_2$-penalty merit function 
\begin{align}
    \phi(\xx; \pi) = f(\xx) + \pi \vnorm{\cc(\xx)}, \tageq\label{eqn:merit function}
\end{align}
where $\pi >0$ is a penalty parameter that balances the contribution of constraint violation and the objective. The $\ell_2$ penalty merit function in \cref{eqn:merit function} is nonsmooth, so, in the following lemma, we bound it using a proxy for the directional derivative under smoothness conditions on the objective and constraints. This result is standard, so we defer the proof to \cref{apx:proof of descent lemma}.

\begin{lemma}[Merit Function Descent Lemma] \label{lemma:descent lemma}
Suppose that the objective function $f$ and constraint mapping $\cc$ are Lipschitz smooth with smoothness constants $L_f$ and $L_c$, respectively. Then, for any direction $\pp$, the directional derivative of the merit function \cref{eqn:merit function} satisfies
\begin{align*}
    D\phi(\xx;\pi)[\pp]
    \leq
    \dotprod{\grad f(\xx),\pp}
    +
    \pi\bigl(
        \vnorm{
            \cc(\xx) + \JJ(\xx)\pp
        }
        -
        \vnorm{
            \cc(\xx)
        }
    \bigr)
    \defeq
    \Db(\xx,\pp;\pi).
    \tageq\label{eqn:directional derivative upper bound}
\end{align*}
Moreover, for any $\alpha \in [0,1]$,
\begin{align*}
    \phi(\xx+\alpha\pp;\pi)
    \leq
    \phi(\xx;\pi)
    +
    \alpha \Db(\xx,\pp;\pi)
    +
    \frac{\alpha^2 \left( \pi L_c +  L_f \right)}{2}
    \vnorm{\pp}^2.
    \tageq\label{eqn:merit function upper bound}
\end{align*}
\end{lemma}

The upper bound established in \cref{lemma:descent lemma} forms the basis for the line-search globalization strategy employed by our method. In particular, at iteration $k$, given a primal update $\pp_k$, we select a step size $\alpha_k$ via a backtracking line search such that
\begin{align*}
    \phi(\xx_k+\alpha_k\pp_k;\pi_k)
    \leq
    \phi(\xx_k;\pi_k)
    +
    \alpha_k \eta
    \Db(\xx_k,\pp_k;\pi_k),
    \tageq\label{eqn:line search descent criteria}
\end{align*}
where $\eta \in (0,1)$ is a user-defined line-search parameter.
A central component of the convergence analysis is showing that the primal update $\pp_k$, together with the choice of merit parameter $\pi_k$, ensures that the upper bound $\Db(\xx_k,\pp_k;\pi_k)$ is strictly negative. Consequently, \cref{eqn:line search descent criteria} guarantees a monotonic decrease of the merit function $\phi(\cdot;\pi_k)$ at every iteration.
To this end, at each iteration we define the trial merit parameter
\begin{align}
    \pitrial_k
    =
    \begin{cases}
        \displaystyle
        \frac{ \dotprod{\grad f_k,\vv_k}}{(1-\rho)\vnorm{\cc_k}},
        &
        \text{if }
        \vnorm{\cc_k} \neq 0, \\[2ex]
        0,
        &
        \text{if }
        \vnorm{\cc_k} = 0,
    \end{cases}
    \tageq\label{eqn:direct merit parameter trial}
\end{align}
where $\rho \in (0,1)$ is a user-specified parameter. Starting from an initial value $\pi_{-1} > 0$, we then update the merit parameter according to
\begin{align*} 
\pi_{k} = \max\{\pitrial_k, \pi_{k-1} \}. \tageq\label{eqn:direct merit parameter update} 
\end{align*}
Intuitively, the update strategy in \cref{eqn:direct merit parameter trial,eqn:direct merit parameter update} ensures that the merit parameter is monotonically nondecreasing and sufficiently large to balance any potential \emph{increase} in the objective caused by the normal step, quantified by $\dotprod{\grad f_k,\vv_k}$, against the corresponding \emph{decrease} in the constraint violation, which is proportional to $\vnorm{\cc_k}$.

\section{Our Algorithm: Direct Normal Computation} \label{sec:direct algorithm statement}
To enhance the clarity of the presentation and the underlying ideas, in this section we focus exclusively on the setting where the normal component of the SQP step is computed directly via a factorization of the constraint Jacobian. This approach is particularly practical when the number of constraints is small relative to the ambient dimension of the problem, i.e., \(c \ll d\). The case of indirect normal computation is deferred to \cref{sec:indirect normal approach}.

We start by discussing the tangent and normal component calculations, followed by details on termination conditions and the dual update in \cref{sec:direct algorithm setup}. We then present the global complexity guarantee of the resulting algorithm in \cref{sec:direct convergence analysis}.  

\subsection{Algorithm and Setup} \label{sec:direct algorithm setup}
\paragraph{Normal Component (Exact Computation).}
We first consider the normal component problem \cref{eqn:normal component problem}. By computing the economy SVD of the Jacobian as in \cref{eqn:SVD of jacobian}, we can simultaneously solve \cref{eqn:normal component problem} and explicitly construct the null space projector $\PP_{\Null(\JJ)}$.
This approach requires the Jacobian to be formed and stored explicitly. When the number of constraints is small relative to the ambient problem dimension, i.e., $c \ll d$, this is computationally practical. In particular, in this regime the Jacobian may be assembled row-by-row using automatic differentiation techniques \citep{blondelElementsDifferentiableProgramming2024}, requiring only $c$ evaluations of the constraint gradients.
Once $\JJ$ has been formed, we compute the economy SVD\footnote{Equivalently, a thin QR factorization of $\JJ^\transpose$ could be employed without affecting the analysis.} of $\JJ^\transpose$ at a computational cost of $\sO(dc^2)$, which is linear in the ambient dimension $d$ when $c \ll d$ \citep{golub_matrix_2013}.

With the economy SVD of the Jacobian available, the unique solution of \cref{eqn:normal component problem} can be written explicitly as $\vv^\star = -\JJ^\dagger \cc$. Indeed, using the SVD representation of $\JJ$ in \cref{eqn:SVD of jacobian}, we obtain
\begin{align*}
    \vv^\star
    =
    -
    \bOmega
    \bSigma^{-1}
    \UU^\transpose
    \cc.
    \tageq\label{eqn:exact normal update}
\end{align*}
From this expression, it follows immediately that $\vv^\star \in \Range(\JJ^\transpose)$.
Moreover, the columns of the SVD factor $\bOmega$ form an orthonormal basis for $\Range(\JJ^\transpose)$. Consequently, they can be used to construct the orthogonal projector onto the null space of the Jacobian:
\begin{align*}
    \PP_{\Null(\JJ)}
    =
    \eye - \bOmega \bOmega^\transpose.
    \tageq\label{eqn:exact normal null space projection}
\end{align*}

\paragraph{Tangent Component.} 

With the projected CR framework established in \cref{sec:CR algorithm background} and the derivation of \cref{eqn:inexact tangent component CR subproblem} in \cref{sec:step derivation}, the tangent computation reduces to a call to the projected CR algorithm (\cref{alg:conjugate residual} with \Cref{term:direct}) with $\HH = \LHess$, $\PP = \PP_{\Null(\JJ)}$, and $\bgg = \gradf$, where $\LHess$ is the current Lagrangian Hessian and $\PP_{\Null(\JJ)}$ is given by \cref{eqn:exact normal null space projection}.
Note that, since $\Range(\bSS) = \Null(\JJ)$, $\bSS^\transpose \Lgrad = \bSS^\transpose \gradf$, and hence it suffices to specify $\bgg = \gradf$. Because the projected CR method operates directly on the reduced Lagrangian Hessian, $\HHt$, the curvature tests described in \cref{sec:CR algorithm background} naturally exploit the curvature of the Lagrangian Hessian in null-space directions.
Moreover, irrespective of the value of $\FLAG$, the iterates generated by projected CR automatically satisfy $\ww \in \Null(\JJ)$. 

\paragraph{Termination Condition.}
The goal of our algorithm is to identify an \emph{$(\epsc,\epsp)$-stationary point} in the sense of \cref{eqn:epsilon KKT stationary point}. While this definition uses the full primal Lagrangian gradient norm $\vnorm{\Lgrad(\xx,\blambda)}$, our algorithm instead monitors the projected gradient norm
\[\vnorm{\PPNJxx \Lgrad(\xx, \blambda)}= \vnorm{\PPNJxx(\gradf + \JJ^\transpose \blambda)} = \vnorm{\PPNJxx \gradf(\xx)}.\] 
Fortunately, as the following lemma demonstrates, the projected gradient and the primal Lagrangian gradient are intimately connected.

\begin{lemma}\label{lemma:termination condition equivalence}
    The primal Lagrangian gradient and null space projected gradient are equivalent termination conditions in the sense that, for $\epsp \geq 0$
    \begin{align*}
        \vnorm{\PPNJxx\grad f(\xx)} \leq \epsp \iff \exists\, \blambda^\circ, \ \vnorm{\Lgrad(\xx, \blambda^\circ)} \leq \epsp.
    \end{align*}
\end{lemma}
\begin{proof}
    For any $\xx$, let $\blambda^\star = - \JJ^\ddag \grad f$ be the pseudoinverse solution to $\min_{\blambda} \vnorm{\JJ^\transpose \blambda + \grad f}$. We have
    \begin{align*}
        \vnorm{\Lgrad(\xx, \blambda^{\star})}  = \vnorm{\JJ^\transpose \blambda^\star + \grad f}= \vnorm{ - \JJ^\transpose \JJ^\ddag \grad f + \grad f} = \vnorm{(\eye - \JJ^\dagger \JJ) \grad f} = \vnorm{\PPNJ \grad f}.
    \end{align*}
     Now, the forward implication is immediate with $\blambda^\circ = \blambda^\star$. For the reverse implication, it also follows that 
    \begin{align*}
        \vnorm{\PPNJ\gradf } = \vnorm{\JJ^\transpose \blambda^\star + \grad f} \leq  \vnorm{\JJ^\transpose \blambda^\circ + \grad f} = \vnorm{\Lgrad(\xx, \blambda^\circ)}  \leq \epsp.
    \end{align*} 
\end{proof}
The multiplier $\blambda^\star = -\JJ^\ddagger \gradf$ in the proof of \cref{lemma:termination condition equivalence} is known as the \emph{least-squares multiplier} and has been used explicitly in the literature to construct the dual update \citep[Chapter 18]{nocedal_numerical_2006}. Remarkably, \cref{lemma:termination condition equivalence} shows that any algorithm targeting the projected gradient is implicitly targeting the primal Lagrangian gradient with the corresponding least-squares multiplier. In fact, \cref{lemma:termination condition equivalence} becomes more intuitive when we note that, if $\vnorm{\Lgrad} = \vnorm{\gradf + \JJ^\transpose \blambda}$ is small, then we expect $\gradf \approx - \JJ^\transpose \blambda$, i.e., the gradient should lie approximately in $\Range(\JJ^\transpose)$ (hence approximately orthogonal to $\Null(\JJ(\xx))$).

Motivated by these observations, we adopt the termination condition
\begin{align*}
    \vnorm{\cc(\xx)} \leq \epsc \quad \text{and} \quad  \vnorm{\PPNJxx \gradf(\xx)} \leq \epsp. \tageq\label{eqn:termination condition}
\end{align*}

\paragraph{Dual Update.} Notably, the termination condition in \cref{eqn:termination condition} depends only on the primal variables and exhibits no explicit dependence on the dual variables. Indeed, the key first-order components of our method, such as the normal update \cref{eqn:exact normal update} and the reduced gradient $\bSS^\transpose \gradf$ used in the tangent computation, are also independent of the dual variables. In fact, the multipliers enter the algorithm only through the second-order Lagrangian Hessian $\LHess$ used for the tangent computation. However, the specific choice of multipliers does not affect the convergence guarantees, provided the Lagrangian Hessian remains bounded (see \cref{remark:Lagrangian Hessian boundedness}). In this sense, the method may be viewed as \emph{dual-update agnostic}: the projection onto the null space implicitly encodes the least-squares multiplier.

Nevertheless, in practice it may be beneficial to select the dual variables in a way that yields meaningful curvature for the Lagrangian. We consider two simple choices. First, recalling that the projected gradient implicitly reflects the least-squares multiplier, we may reuse the SVD factors from the normal step to compute this multiplier explicitly,
\begin{align*}
    \blambda_{k+1} = - \JJ_k^\ddagger \gradf_k
    = - \UU_k \bSigma_k^{-1}\bOmega_k^\transpose \gradf_k. 
    \tageq\label{eqn:least-squares-multipliers}
\end{align*}
This update ensures that the Lagrangian Hessian reflects the curvature of the Lagrangian evaluated at the least-squares multiplier.
Alternatively, we may set $\blambda_{k+1} = 0$, in which case
\begin{align*}
    \LHess(\xx_k, \blambda_{k+1}) = \Hessf_k,
\end{align*}
so that the tangent computation uses only the curvature of the objective function. Since the tangent step is ultimately intended to reduce the objective value (\cref{sec:direct convergence analysis}), this choice may be interpreted as favoring the curvature of the objective over that of the constraints.
Both the least-squares update \cref{eqn:least-squares-multipliers} and the choice $\blambda_{k+1}=0$ preserve the boundedness of the Lagrangian Hessian under the regularity assumptions (\cref{ass:basic assumptions,ass:jacobian regularity}) used in our analysis; see \cref{remark:Lagrangian Hessian boundedness}. Regardless of the specific dual update, once a point $\xx$ satisfying \cref{eqn:termination condition} is obtained, \cref{lemma:termination condition equivalence} implies that a pair $(\xx,\blambda)$ satisfying \cref{eqn:epsilon KKT stationary point} can be recovered by solving
\[
\min_{\blambda}\vnorm{\JJ^\transpose \blambda + \gradf}^2.
\]

Our complete method is presented in \cref{alg:direct primal-dual algorithm}.
\begin{algorithm}
\renewcommand{\baselinestretch}{1.25}\selectfont 
\begin{algorithmic}[1]
\Require{Outer termination tolerances $(\epsc, \epsp) \in (0,1)^2$, tangent termination tolerance $\tautan \in (0,1)$, curvature tolerance $\sigma > 0$, initial point $(\xx_0, \blambda_0)$, initial merit parameter $\pi_{-1} > 0$, merit parameter update factor $\rho \in (0,1)$, line search parameter $\eta \in (0,1)$, and backtracking parameter $\theta \in (0,1)$.}

\For{$k=0, 1, 2, \ldots$}
    \State Compute the economy SVD of the Jacobian: $\JJ_k = \UU_k \bSigma_k \bOmega_k^\transpose$
    \State Compute $\vv_k$ from \cref{eqn:exact normal update} and $\PPNJk$ from \cref{eqn:exact normal null space projection}

    \If{$\vnorm{\PPNJk \gradf_k} \leq \epsp$ and $\vnorm{\cc_k} \leq \epsc$}
        \State (Optional) compute $\blambda_k^\star = \arg\min_{\blambda} \vnorm{\Lgrad(\xx_k, \blambda)}^2$
        \State \Return $(\xx_k, \blambda_k^\star)$ \Comment{Terminate with $(\epsc,\epsp)$-stationary point.}
    \EndIf

    \State Update $\blambda_{k+1}$, e.g., the least-squares multiplier \cref{eqn:least-squares-multipliers} or $\blambda_{k+1} = 0$ \Comment{Optional dual update.}

    \State Set $\bgg_k = \gradf_k$, $\HH_k = \LHess(\xx_k, \blambda_{k+1})$

    \State $(\ww_k, \FLAG_k)\gets$[\cref{alg:conjugate residual}]$(\HH_k, \bgg_k, \PPNJk, \sigma, \tautan)$ \Comment{\cref{alg:conjugate residual} with \Cref{term:direct}.}

    \State Update merit parameter $\pi_k$ using \cref{eqn:direct merit parameter trial,eqn:direct merit parameter update}

    \State Set $\pp_k = \vv_k + \ww_k$
    \State Use backtracking line search to find $\alpha_k \leq 1$ satisfying \cref{eqn:line search descent criteria}

    \State $\xx_{k+1} = \xx_k + \alpha_k \pp_k$
\EndFor

\caption{Direct Primal-Dual Newton-MR Algorithm}
\label{alg:direct primal-dual algorithm}
\end{algorithmic}
\end{algorithm}

\subsection[
    Convergence Analysis of Algorithm~\ref{alg:direct primal-dual algorithm}
]{
    Convergence Analysis of \cref{alg:direct primal-dual algorithm}
} \label{sec:direct convergence analysis}
We now establish the global convergence properties of \cref{alg:direct primal-dual algorithm}. To this end, we first introduce and discuss the assumptions underpinning our analysis. Since our primary objective is to address the nonconvex setting, that is, to dispense with \cref{cond:kkt:positive-definite-Lagrangian-Hessian}, we use a strengthened version of \cref{cond:kkt:full-rank-Jacobian}.

\begin{assumption}[Jacobian Regularity] \label{ass:jacobian regularity}
    Let $\sX$ denote the set containing all iterates and trial points generated by the algorithm. There exist constants $0 < \sigma_\JJ \leq \kappa_\JJ < \infty$ such that
    \[
        \sigma_\JJ \leq \sigma_c(\xx) \leq \sigma_1(\xx) \leq \kappa_\JJ, \qquad \forall \xx \in \sX.
    \]
\end{assumption}

\cref{ass:jacobian regularity} is stronger than \cref{cond:kkt:full-rank-Jacobian}, as it additionally requires a uniform lower bound on the smallest singular value of the Jacobian. Nevertheless, assumptions of this form are standard throughout the SQP literature in the full-rank setting \citep{nocedal_numerical_2006,byrdInexactSQPMethod2008,byrdInexactNonconvexNewtonMethod2010,curtisWorstCaseComplexitySQP2022}.
Indeed, when $\sigma_c(\xx)$ is not uniformly bounded away from zero, the problem may be more naturally viewed as belonging to a rank-deficient regime, in which changes in the rank of the Jacobian are permitted and must be accounted for explicitly in the analysis.

\begin{assumption} \label{ass:basic assumptions}
On $\sX$, the following holds:
\begin{assitems}
    \item The objective $f$ and constraint mapping $\cc$ are twice continuously differentiable. Moreover, $\grad f$ and $\JJ$ are $L_f$- and $L_c$-Lipschitz continuous, respectively, and $f$ is bounded below by $\flow$.
    \label{ass:item:function constraint smoothness}

    \item The objective gradient is uniformly bounded: $\vnorm{\grad f(\xx)} \leq \kappa_g$.
    \label{ass:item:gradient is bounded}

    \item The constraints are uniformly bounded: $\vnorm{\cc(\xx)} \leq \kappa_c$.
    \label{ass:item:constraint is bounded}
\end{assitems}
\end{assumption}
While \cref{ass:basic assumptions} is relatively restrictive, such conditions are standard in the analysis of SQP methods (see, e.g., \citep{berahasSequentialQuadraticOptimization2020,curtisWorstCaseComplexitySQP2022,byrdInexactSQPMethod2008,byrdInexactNonconvexNewtonMethod2010,curtisMatrixFreeAlgorithmEquality2010}). Moreover, both the SQP literature and practical implementations suggest that SQP algorithms are significantly more robust to violations of \cref{ass:basic assumptions} than, for example, \cref{cond:kkt:positive-definite-Lagrangian-Hessian}. For this reason, and since our primary goal is to relax \cref{cond:kkt:positive-definite-Lagrangian-Hessian}, we proceed under \cref{ass:basic assumptions}.

\cref{ass:basic assumptions} omits some standard restrictions that appear in related SQP analyses. In particular, we do not impose an explicit boundedness assumption on the Lagrangian Hessian (or on the dual variables). The following remark shows that boundedness of the Lagrangian Hessian follows from \cref{ass:basic assumptions} under natural choices of the dual variables used in evaluating $\LHess(\xx,\blambda)$.

\begin{remark}[Boundedness of the Primal Lagrangian Hessian] \label{remark:Lagrangian Hessian boundedness}
\cref{ass:item:function constraint smoothness} implies that both the objective Hessian and constraint Hessians are uniformly bounded on $\sX$. Since
\begin{align*}
    \LHess(\xx,\blambda)
    =
    \Hessf(\xx)
    +
    \sum_{i=1}^c \lambda_i \grad^2 \cc_i(\xx),
\end{align*}
it follows that boundedness of the primal Lagrangian Hessian is guaranteed provided the dual variables $\{\blambda_k\}$ remain bounded.
Clearly, choosing $\blambda_{k+1} = 0$ ensures boundedness. On the other hand, for the least-squares multipliers $\blambda^\star$ defined in \cref{eqn:least-squares-multipliers}, \cref{ass:item:gradient is bounded} and \cref{ass:jacobian regularity} imply
\begin{align*}
    \vnorm{\blambda^\star(\xx)}
    &=
    \vnorm{\JJ^\ddag(\xx)\grad f(\xx)} 
    \leq
    \frac{\kappa_g}{\sigma_\JJ}.
\end{align*}
Since $\blambda_{k+1}=0$ and the least-squares multipliers are the only dual updates we consider in this section, we may therefore conclude that the Lagrangian Hessian is uniformly bounded, i.e.,
\begin{align*}
    \vnorm{\LHess(\xx,\blambda)} \leq \kappa_{\sL}.
    \tageq\label{eqn:Lagrangian Hessian boundedness}
\end{align*}
\end{remark}

We now state the main convergence result and provide brief remarks before presenting the proof.
\begin{theorem}\label{thm:direct case convergence}
Suppose \cref{ass:basic assumptions,ass:jacobian regularity} hold and let $(\epsc,\epsp)\in(0,1)^2$. Then \cref{alg:direct primal-dual algorithm} reaches a point satisfying \cref{eqn:termination condition} in at most
\[
K \defeq \left\lceil
\frac{(f_0-\flow)/\pi_{-1} + \vnorm{\cc_0}}
{\min\{\Gamma_c,\Gamma_p\}\min\{\epsc,\epsp^2\}}
\right\rceil
+1,
\]
iterations, where $\Gamma_c,\Gamma_p>0$ are constants defined in \cref{eq:direct case Gamma}.
\end{theorem}

\cref{thm:direct case convergence} implies that a primal iterate satisfying \cref{eqn:termination condition} is obtained within
$\sO(\max\{\epsc^{-1},\epsp^{-2}\})$ iterations.
Setting $\epsc=\epsp^2=\varepsilon$, the resulting complexity becomes $\sO(\varepsilon^{-2})$, matching the rate obtained in deterministic SQP frameworks such as \citep{curtisWorstCaseComplexitySQP2022,berahasSequentialQuadraticOptimization2020}. In contrast to those approaches, however, the present method does not require exact subproblem solves and does not rely on a positive definite Lagrangian Hessian proxy.
Additionally, the resulting dependence on $\epsp$ obtained is consistent with known lower bounds for first-order stationarity of nonconvex smooth optimization in the unconstrained setting, and matches the rates achieved by both gradient descent and Newton-type methods under appropriate smoothness assumptions \citep[Theorems 2.2.3 and 3.1.1]{cartisEvaluationComplexityAlgorithms2022}.

We prove \cref{thm:direct case convergence} via a sequence of auxiliary lemmas (\cref{lemma:direct normal component descent and norm properties,lemma:direct case descent of tangent component,lemma:direct primal update descent on merit function,lemma:direct merit parameter boundedness,lemma:direct line search termination}). 
We first analyze the decrease in the merit function induced by the normal update (\cref{lemma:direct normal component descent and norm properties}), followed by the contribution of the tangent update (\cref{lemma:direct case descent of tangent component}). We then show that the merit parameter defined in \cref{eqn:direct merit parameter trial,eqn:direct merit parameter update} ensures that the resulting search direction is a descent direction for the merit function (\cref{lemma:direct primal update descent on merit function}). Finally, we establish that the merit parameter remains bounded (\cref{lemma:direct merit parameter boundedness}) and that the line search terminates for a sufficiently small step size (\cref{lemma:direct line search termination}), which together yield the result in \cref{thm:direct case convergence}.

We begin with the properties of the normal update.

\begin{lemma}[Normal Component Properties] \label{lemma:direct normal component descent and norm properties}
Under \cref{ass:jacobian regularity}, the normal component $\vv_k$ defined by \cref{eqn:exact normal update} satisfies
\begin{align*}
\vnorm{\JJ_k \vv_k + \cc_k}
&= 0,
\\
\frac{\vnorm{\cc_k}}{\kappa_\JJ}
\leq
\vnorm{\vv_k}
&\leq
\frac{\vnorm{\cc_k}}{\sigma_\JJ}.
\end{align*}
\end{lemma}

\begin{proof}
    The equality holds immediately from \cref{ass:jacobian regularity}, since $\JJ^\dagger_k$ acts as a right inverse, i.e., $\JJ_k\vv_k + \cc_k = - \JJ_k\JJ^\dagger_k \cc_k + \cc_k=\zero$. To show the inequalities, we note that 
    \begin{align*}
        \vnorm{\JJ^\dagger_k \cc_k} = \vnorm{\bSigma_k^{-1} \UU_k^\transpose \cc_k} \leq \frac{1}{\sigma_\JJ} \vnorm{\UU_k^\transpose \cc_k} \leq \frac{1}{\sigma_{\JJ}} \vnorm{\cc_k},
    \end{align*}
    and the lower bound follows similarly.
\end{proof}

Next, we consider the properties of the tangent update. In particular, we show that the tangent search direction satisfies key descent and norm bounds regardless of the value of $\FLAG$ returned by the projected CR algorithm (\cref{alg:conjugate residual}).

\begin{lemma}[Descent Properties of Tangent Component] \label{lemma:direct case descent of tangent component}
    In \cref{alg:direct primal-dual algorithm}, under \cref{ass:basic assumptions}, the call to \cref{alg:conjugate residual} with \Cref{term:direct} returns a tangent search direction $\ww_k$ satisfying
    \begin{align*}
        \dotprod{\gradf_k, \ww_k} &\leq - \gammatan \vnorm{\ww_k}^2, \tageq\label{eqn:exact tangent component descent}\\
        \vnorm{\ww_k} &\geq \omegatan \vnorm{\PPNJk \gradf_k }, \tageq\label{eqn:exact tangent component norm bound}
    \end{align*}
    where
    \begin{align*}
        \gammatan \defeq \min\left\{1, \sigma \right\},
        \qquad
        \omegatan \defeq \min\left\{\frac{\sigma}{\kappa_{\sL}^2}, \frac{\tautan}{\sqrt{\tautan^2 + \kappa_{\sL}^2}} \right\}.
    \end{align*}
\end{lemma}

\begin{proof}
    Suppose the call to \cref{alg:conjugate residual} returns $\FLAG=\SOL$. Then, \cref{eqn:Lagrangian Hessian boundedness} together with \cref{eqn:SOL second order descent,eqn:SOL step length} from \cref{lemma:projected CR properties} imply
    \begin{align*}
        \dotprod{\gradf_k, \ww_k} &\leq - \sigma \vnorm{\ww_k}^2, \\
        \vnorm{\ww_k} \geq \frac{\sigma}{\kappa_{\sL}^2} \vnorm{\bSS^\transpose_k \gradf_k} &= \frac{\sigma}{\kappa_{\sL}^2} \vnorm{\PPNJk \gradf_k},
    \end{align*}
    where the equality follows from orthogonality of $\bSS_k$. 
    On the other hand, suppose $\FLAG=\LPC$. Then \cref{eqn:Lagrangian Hessian boundedness} as well as \cref{eqn:LPC step descent,eqn:LPC step length} from \cref{lemma:projected CR properties} imply 
    \begin{align*}
        \dotprod{\ww_k, \gradf_k} &= - \vnorm{\ww_k}^2,\\
        \vnorm{\ww_k} &\geq \frac{\tautan}{\sqrt{\tautan^2 + \kappa_{\sL}^2}} \vnorm{\PPNJk \gradf_k },
    \end{align*}
    where we again use the orthogonality of $\bSS_k$. Combining the above bounds establishes the result.
\end{proof}
\cref{lemma:direct case descent of tangent component} reveals the utility of \cref{eqn:termination condition} as the outer termination condition, since the descent due to the tangent component is directly proportional to the projected gradient.

We now combine \cref{lemma:direct case descent of tangent component,lemma:direct normal component descent and norm properties} with the merit parameter update \cref{eqn:direct merit parameter trial,eqn:direct merit parameter update} to show that the resulting primal update is a descent direction for the merit function \cref{eqn:merit function}.

\begin{lemma}[Descent on the Merit Function] \label{lemma:direct primal update descent on merit function}
    Under \cref{ass:jacobian regularity,ass:basic assumptions}, the primal search direction $\pp_k$ and merit parameter $\pi_k$ computed in \cref{alg:direct primal-dual algorithm} satisfy
    \begin{align*}
        \Db(\xx_k, \pp_k; \pi_k)
        \leq
        -\gammatan \vnorm{\ww_k}^2
        - \rho \pi_k \vnorm{\cc_k},
    \end{align*}
    where $\Db$ is defined in \cref{lemma:descent lemma}.
\end{lemma}

\begin{proof}
From \cref{lemma:descent lemma} and $\ww_k \in \Null(\JJ_k)$, we have
\begin{align*}
    \Db(\xx_k, \pp_k; \pi_k)
    =
    \dotprod{\grad f_k, \vv_k}
    +
    \dotprod{\grad f_k, \ww_k}
    +
    \pi_k\big(\vnorm{\cc_k + \JJ_k\vv_k} - \vnorm{\cc_k}\big).
\end{align*}

Applying \cref{lemma:direct normal component descent and norm properties,lemma:direct case descent of tangent component} gives
\begin{align*}
    \Db(\xx_k, \pp_k; \pi_k)
    \leq
    -\gammatan \vnorm{\ww_k}^2
    + \dotprod{\grad f_k, \vv_k}
    - \pi_k \vnorm{\cc_k}.
\end{align*}

If $\vnorm{\cc_k} = 0$, then $\vv_k = 0$ (see \cref{eqn:exact normal update}), and the result follows immediately. Otherwise, rearranging \cref{eqn:direct merit parameter trial} and using \cref{eqn:direct merit parameter update} yields
\[
\dotprod{\grad f_k, \vv_k} \leq (1-\rho)\pi_k \vnorm{\cc_k}.
\]
Substituting into the previous bound gives
\begin{align*}
    \Db(\xx_k, \pp_k; \pi_k)
    &\leq
    -\gammatan \vnorm{\ww_k}^2
    + (1-\rho)\pi_k \vnorm{\cc_k}
    - \pi_k \vnorm{\cc_k} \\
    &=
    -\gammatan \vnorm{\ww_k}^2
    - \rho \pi_k \vnorm{\cc_k}.
\end{align*}
\end{proof}

\cref{lemma:direct primal update descent on merit function} implies that the primal update computed in \cref{alg:direct primal-dual algorithm} is a descent direction for the merit function \cref{eqn:merit function}, ensuring that the Armijo line search in \cref{eqn:line search descent criteria} is well defined. The next step in our analysis is to show that both the merit parameter and the step size selected by the line search remain bounded.

\begin{lemma}[Merit Parameter Boundedness]\label{lemma:direct merit parameter boundedness}
    Under \cref{ass:jacobian regularity,ass:basic assumptions}, the penalty parameter selected in \cref{alg:direct primal-dual algorithm} is uniformly bounded, i.e., there exists $0 < \pib  < \infty$ such that 
    $$\pi_{-1} \leq \pi_k \leq \pib, \quad \forall k.$$
\end{lemma}
\begin{proof}
    Using \cref{eqn:direct merit parameter trial}, together with \cref{ass:item:gradient is bounded} and \cref{lemma:direct normal component descent and norm properties}, we obtain
    \begin{align*}
        \pitrial_k
        =
        \frac{\dotprod{\grad f_k, \vv_k}}{(1-\rho)\vnorm{\cc_k}}
        \leq
        \frac{\vnorm{\grad f_k}\,\vnorm{\vv_k}}{(1-\rho)\vnorm{\cc_k}}
        \leq
        \frac{\kappa_g}{(1-\rho)\sigma_{\JJ}}.
    \end{align*}
    Define
    \begin{align*}
        \pib \defeq \max\left\{\frac{\kappa_g}{(1-\rho)\sigma_{\JJ}}, \pi_{-1}\right\}.
    \end{align*}
    Then the update rule in \cref{eqn:direct merit parameter update} implies
    \begin{align*}
        0 < \pi_{-1} \leq \pi_k \leq \pib < \infty.
    \end{align*}
\end{proof}

\begin{lemma}[Line Search Termination] \label{lemma:direct line search termination}
    Under \cref{ass:basic assumptions,ass:jacobian regularity}, there exists $\alphab > 0$ such that, for all iterations of \cref{alg:direct primal-dual algorithm}, the step size selected by the line search satisfies $\alpha_k \geq \alphab$.
\end{lemma}
\begin{proof}
Let $\Delta_k(\alpha) \defeq \phi(\xx_k+\alpha\pp_k; \pi_k) - \phi(\xx_k; \pi_k) - \eta \alpha \Db_k$. Applying \cref{eqn:merit function upper bound}, we have
\begin{align*}
    \Delta_k(\alpha) 
    &\leq \alpha(1-\eta)  \Db_k + \frac{\alpha^2}{2} \left( \pi_k L_c + L_f\right) \vnorm{\pp_k}^2. 
\end{align*}
Additionally, from $\vnorm{\pp_k}^2 = \vnorm{\ww_k}^2 + \vnorm{\vv_k}^2$, \cref{lemma:direct primal update descent on merit function}, the boundedness of the merit parameter (\cref{lemma:direct merit parameter boundedness}), the upper bound on the normal update (\cref{lemma:direct normal component descent and norm properties}), and the upper bound on the constraints (\cref{ass:item:constraint is bounded}), we have
\begin{align*}
    \Delta_k(\alpha) 
    &\leq -\alpha(1-\eta) (\gammatan \vnorm{\ww_k}^2 +\pi_{k}\rho\vnorm{\cc_k}) + \frac{\alpha^2}{2} \left( \pi_k L_c + L_f\right) \left( \vnorm{\ww_k}^2 + \vnorm{\vv_k}^2 \right) \\
    &\leq -\alpha(1-\eta) (\gammatan \vnorm{\ww_k}^2 +\pi_{-1}\rho\vnorm{\cc_k}) + \frac{\alpha^2}{2} \left( \pib L_c + L_f\right) \left( \vnorm{\ww_k}^2 + \frac{1}{\sigma_{\JJ}^2} \vnorm{\cc_k}^{2} \right) \\
    &\leq -\alpha(1-\eta) (\gammatan \vnorm{\ww_k}^2 +\pi_{-1}\rho\vnorm{\cc_k}) + \frac{\alpha^2}{2} \left( \pib L_c + L_f\right) \left( \vnorm{\ww_k}^2 + \frac{\kappa_c}{\sigma_{\JJ}^2} \vnorm{\cc_k} \right) \\
    &\leq -\alpha(1-\eta)\min\{\gammatan, \pi_{-1} \rho \}(\vnorm{\ww_k}^2 +\vnorm{\cc_k}) + \frac{\alpha^2}{2}\left( \pib L_c + L_f\right) \max \left\{1, \frac{\kappa_c}{\sigma_{\JJ}^2} \right\}(\vnorm{\ww_k}^2 + \vnorm{\cc_k})  \\
    &= \alpha\left(-(1-\eta)\min\{\gammatan, \pi_{-1} \rho \} + \frac{\alpha}{2}\left( \pib L_c + L_f\right) \max \left\{1, \frac{\kappa_c}{\sigma_{\JJ}^2} \right\}   \right)\left(\vnorm{\ww_k}^2 +\vnorm{\cc_k} \right).
\end{align*}
Hence, $\Delta_k(\alpha) \leq 0$ whenever
\begin{align*}
    \alpha \leq
    \frac{2(1-\eta)\min\{ \gammatan, \pi_{-1} \rho \}}
    {\left( \pib L_c + L_f\right) \max \left\{1, \kappa_c/\sigma_{\JJ}^2 \right\}}.
\end{align*}
Since the line search uses backtracking, it may undershoot this threshold by at most a factor $\theta$. Therefore, the step size returned by \cref{alg:direct primal-dual algorithm} satisfies
\begin{align}
    \label{eq:alphab}
    \alpha_k \geq \min\left\{
    1,\;
    \theta
    \frac{2(1-\eta)\min\{ \gammatan, \pi_{-1} \rho \}}
    {\left( \pib L_c + L_f\right) \max \left\{1, \kappa_c/\sigma_{\JJ}^2 \right\}}
    \right\}
    \defeq \alphab.
\end{align}
\end{proof}

We are now ready to prove \cref{thm:direct case convergence}.

\begin{proof}[Proof of \cref{thm:direct case convergence}]
    Define 
    \begin{align}
    \label{eq:direct case Gamma}
        \Gamma_c \defeq \alphab \eta \rho, \quad \text{and} \quad \Gamma_p \defeq \frac{\alphab \eta \gammatan \omegatan^2}{\pib}, 
    \end{align}
    where $\gammatan$ and $\omegatan$ are as in \cref{lemma:direct case descent of tangent component} and $\alphab$ is as in \cref{eq:alphab}.
    Suppose that the termination condition \cref{eqn:termination condition} fails to hold for iterations $k = 0,\ldots,K-1$. We partition the iterations up to $K-1$ into two sets 
    \begin{align*}
        \sK_c = \left\{ k \in [K-1] \mid \vnorm{\cc_k} > \epsc \right\}, \qquad \sK_p = [K-1]\setminus \sK_c.
    \end{align*}
    Suppose $k \in \sK_c$. 
    By \cref{eqn:line search descent criteria} and \cref{lemma:direct line search termination,lemma:direct primal update descent on merit function}, we have
    \begin{align*}
        \frac{\phi(\xx_{k+1}; \pi_k) - \phi(\xx_k; \pi_k)}{\pi_k} 
        &\leq \frac{\alpha_k \eta}{\pi_k}\Db(\xx_k, \pp_k; \pi_k) \\
        &\leq -\alphab \eta \rho \vnorm{\cc_k} \\ 
        &\leq -\alphab \eta \rho \epsc = - \Gamma_c \epsc.
    \end{align*}
    On the other hand, for $k \in \sK_p$, we have $\vnorm{\PPNJk \gradf_k} > \epsp $. Applying \cref{lemma:direct primal update descent on merit function,lemma:direct merit parameter boundedness,lemma:direct line search termination}, the lower bound on $\vnorm{\ww_k}$ from \cref{lemma:direct case descent of tangent component}, and the line search criteria \cref{eqn:line search descent criteria}, we have
    \begin{align*}
        \frac{\phi(\xx_{k+1}; \pi_k) - \phi(\xx_k; \pi_k)}{\pi_k} 
        &\leq \frac{\alpha_k \eta}{\pi_k}\Db(\xx_k, \pp_k; \pi_k) \\
        &\leq -\frac{\alphab \eta \gammatan}{\pib} \vnorm{\ww_k}^2 \\ 
        &\leq - \frac{\alphab \eta \gammatan \omegatan^2 \epsp^2}{\pib} = - \Gamma_p \epsp^2.
    \end{align*}
    Since $\sK_c$ and $\sK_p$ partition $[K-1]$, the above bounds give
    \begin{align*}
        \sum_{k=0}^{K-1} \frac{\phi(\xx_{k+1}; \pi_k) - \phi(\xx_k; \pi_k)}{\pi_k} &= \sum_{k \in \sK_c} \frac{\phi(\xx_{k+1}; \pi_k) - \phi(\xx_k; \pi_k)}{\pi_k} + \sum_{k \in \sK_p} \frac{\phi(\xx_{k+1}; \pi_k) - \phi(\xx_k; \pi_k)}{\pi_k} \\
         &\leq -|\sK_c| \Gamma_c \epsc - |\sK_p| \Gamma_p \epsp^2 \\
        &\leq
        -\min\{\Gamma_c,\Gamma_p\}
        \left(
            |\sK_c|\epsc + |\sK_p|\epsp^2
        \right) \\
        &\leq
        -\min\{\Gamma_c,\Gamma_p\} \min\{\epsc,\epsp^2\}
        \left(
            |\sK_c| + |\sK_p|
        \right) \\
        &= -K\min\{\Gamma_c, \Gamma_p\} \min\{\epsc,\epsp^2\}.
    \end{align*}
    On the other hand, from the definition of the merit function \cref{eqn:merit function},
    \begin{align*}
        \sum_{k=0}^{K-1} \frac{\phi(\xx_{k+1}; \pi_k) - \phi(\xx_k; \pi_k)}{\pi_k} 
        &= \sum_{k=0}^{K-1} \frac{f_{k+1}}{\pi_k} - \frac{f_k}{\pi_k} +  \vnorm{\cc_{k+1}} - \vnorm{\cc_k}  \\
        &= \sum_{k=0}^{K-1} \frac{f_{k+1}}{\pi_{k+1}} - \frac{f_{k}}{\pi_k} + f_{k+1}\left(\frac{1}{\pi_{k}} -\frac{1}{\pi_{k+1}}\right) + \vnorm{\cc_{k+1}} - \vnorm{\cc_k}.         
    \end{align*}
    Applying a telescoping argument, the monotonicity of $\pi_{k}$, $f_k \geq \flow$ and $\vnorm{\cc_k} \geq 0$, we obtain
    \begin{align*}
        \sum_{k=0}^{K-1} \frac{\phi(\xx_{k+1}; \pi_k) - \phi(\xx_k; \pi_k)}{\pi_k} 
        &\geq \frac{f_K}{\pi_K} - \frac{f_0}{\pi_0} + \flow \left(\frac{1}{\pi_0} - \frac{1}{\pi_K}\right) + \vnorm{\cc_{K}} - \vnorm{\cc_0} \\
        &\geq \frac{\flow}{\pi_K} - \frac{f_0}{\pi_0} + \flow \left(\frac{1}{\pi_0} - \frac{1}{\pi_K}\right) - \vnorm{\cc_0} \\
        &\geq  \frac{1}{\pi_{-1}}\left(\flow - f_0 \right) - \vnorm{\cc_0}.
    \end{align*}
    Combining the lower and upper bounds, we deduce
    \begin{align*}
        K \leq \frac{(f_0 - \flow)/\pi_{-1} + \vnorm{\cc_0}}{\min\{\Gamma_c, \Gamma_p\} \min\{\epsc,\epsp^2\}}.
    \end{align*}
    This contradicts the definition of $K$ in \cref{thm:direct case convergence}.
\end{proof}

\section{Our Algorithm: Indirect Normal Computation} \label{sec:indirect normal approach}
We now extend the framework of \cref{sec:direct algorithm statement} to the \emph{indirect} setting, suitable for problems with a large number of constraints, where forming and factorizing $\JJ$ is impractical. In this regime, the normal component problem \cref{eqn:normal component problem} is solved inexactly using only Jacobian- and Jacobian-transpose-vector products. Consequently, a complete orthonormal basis for $\Range(\JJ^\transpose)$ is not readily available, and the projection $\PPNJ$ required for the tangent computation cannot be formed explicitly.
Nevertheless, we show that the projection can be approximated by combining information from the inexact normal solve with iterative range-finding techniques. Crucially, we derive practical termination conditions for the inner solvers that ensure the approximation is sufficiently accurate to preserve the convergence guarantees established for the direct approach.
In \cref{sec:indirect algorithm setup}, we motivate this construction and highlight the key differences from the direct case. The corresponding convergence analysis is presented in \cref{sec:indirect convergence analysis}.

\subsection{Algorithm and Setup} \label{sec:indirect algorithm setup}

We now describe the indirect algorithm (\cref{alg:indirect primal-dual algorithm}). Much of the underlying setup remains the same as in the direct case: our goal is to compute a primal search direction that is a descent direction for the merit function \cref{eqn:merit function}. An Armijo \citep{armijo_line_search_1966} line search is then used to select a step size that produces sufficient decrease in the merit function according to \cref{eqn:line search descent criteria}. The key differences from the direct case in \cref{alg:direct primal-dual algorithm} are twofold. First, we no longer solve the normal component problem exactly. Second, we replace the exact null space projector with an approximate projection, $\hPP \approx \PPNJ$, with corresponding adjustments to the tangent computation and outer termination conditions.

To motivate the modifications required in the indirect setting, we revisit the upper bound on the directional derivative derived in \cref{lemma:descent lemma}. For a search direction decomposed as $\pp = \ww + \vv$, we have
\begin{align*}
    \Db(\xx, \pp; \pi) &= \dotprod{\gradf, \pp} + \pi (\vnorm{\JJ \pp + \cc} - \vnorm{\cc}) \\ 
    &\leq \dotprod{\gradf, \ww} + \dotprod{\gradf, \vv} + \pi(\vnorm{\JJ\ww} + \vnorm{\JJ\vv + \cc} - \vnorm{\cc}). 
    \tageq\label{eqn:indirect case direction derivative Jw bound}
\end{align*}
Focusing on the normal component terms in \cref{eqn:indirect case direction derivative Jw bound}, we see that any inexact normal update for which the residual, $\|\JJ\vv + \cc\|$, is small relative to the constraint violation, $\|\cc\|$, ensures that the normal-component terms provide descent in \cref{eqn:indirect case direction derivative Jw bound}, provided that $\pi$ is sufficiently large.

The picture is more complicated for the tangent component. In particular, the upper bound in \cref{eqn:indirect case direction derivative Jw bound} includes an additional $\vnorm{\JJ\ww}$ term, which is absent in the direct case because $\ww \in \Null(\JJ)$. Consequently, even if the descent conditions from the direct case (i.e., \cref{lemma:direct case descent of tangent component,lemma:direct normal component descent and norm properties}) continue to hold, there may be no choice of merit parameter that guarantees $\pp$ is a descent direction for the merit function. This observation reveals the central difficulty in the indirect setting: the approximate projector we construct, $\hPP \approx \PPNJ$, must control the quantity $\|\JJ\ww\|$ within the tangent iterations.

Our approach to computing the approximate projector $\hPP$ is to employ an auxiliary \textit{range-finding} algorithm that can incorporate range-space information from the inexact normal solve. The algorithm identifies an orthonormal basis $\hbOmega$ such that $\Range(\hbOmega) \subseteq \Range(\JJ^\transpose)$. We then define 
\[\hPP \defeq \eye - \hbOmega\hbOmega^\transpose. \tageq\label{eqn:approximate projector formula}\]
The resulting operator $\hPP$ remains an orthogonal projection and, crucially, \textit{overestimates} the null space in the sense that $\Null(\JJ) \subseteq \Range(\hPP)$. 

In the remainder of this section, we describe the key details of the indirect algorithm. We begin with a discussion of the inexact normal component update algorithm and the relevant termination conditions. We then analyze how the approximate projection, $\hPP$, affects the tangent computation. Specifically, we isolate a simple condition that the approximate projector must satisfy to attain descent in the tangent component, given in \cref{eqn:null space gradient termination condition}. Subsequently, we consider how the outer termination conditions and the dual variable update must be modified in the indirect case. With these conditions on the approximate projector in place, we then discuss a concrete instance (\cref{example:Krylov Range-Finder}) of a range-finding algorithm suitable to our setting. Finally, we present the complete algorithm in \cref{alg:indirect primal-dual algorithm}.

\paragraph{Normal Component (Inexact Computation).}
We solve \cref{eqn:normal component problem} inexactly using the least-squares minimal residual (LSMR) method \citep{fongLSMRIterativeAlgorithm2011}. LSMR is a Krylov subspace method, closely related to LSQR \citep{paigeLSQRAlgorithmSparse1982a}, based on the Golub-Kahan bidiagonalization \citep{golub_matrix_2013} applied to $\JJ$. The bidiagonalization process uses Jacobian-vector and Jacobian-transpose-vector products to iteratively construct a basis for the Krylov subspace $\sK_t(\JJ^\transpose\JJ, \JJ^\transpose\cc)$, thereby avoiding the need to form $\JJ$ explicitly. Using this Krylov basis, LSMR produces a sequence of iterates approximating the solution of \cref{eqn:normal component problem} via the subproblem
\begin{align*}
    \vvtt
    =
    \argmin_{\vv}
    \vnorm{\JJ^\transpose\JJ \vv + \JJ^\transpose \cc}
    \qquad
    \text{s.t.}
    \qquad
    \vv \in \sK_t(\JJ^\transpose\JJ, \JJ^\transpose \cc).
    \tageq\label{eqn:indirect normal component computation}
\end{align*}
From the subproblem in \cref{eqn:indirect normal component computation}, it is clear that LSMR is equivalent to applying the CR algorithm to the normal equations associated with \cref{eqn:normal component problem}. This observation motivates our use of LSMR rather than LSQR, since CR is known to be significantly more stable than CG (recall that LSQR corresponds to CG applied to the normal equations) for positive semidefinite systems \citep{limConjugateDirectionMethods2024}. Moreover, this interpretation allows us to exploit several favorable properties of CR to obtain guarantees for $\vvtt$ analogous to those in \cref{lemma:direct normal component descent and norm properties}.

Computing an approximate solution to \cref{eqn:normal component problem} using LSMR is straightforward. As we show in \cref{lemma:inexact normal component properties}, the iteration can be terminated once the relative residual condition
\begin{align*}
    \vnorm{\JJ \vvtt + \cc}
    \leq
    \taunorm \vnorm{\cc},
    \tageq\label{eqn:inexact normal component termination condition}
\end{align*}
is satisfied, where $\taunorm \in [0,1)$ is a user-chosen parameter. This condition is readily available within the LSMR iterations and is well posed since, under \cref{ass:jacobian regularity}, the normal component problem \cref{eqn:normal component problem} admits a zero-residual solution. The resulting inexact normal update is summarized in \cref{alg:inexact normal component}. Note that in \cref{alg:inexact normal component} we allow for optional storage of the Krylov vectors, $\tbOmega$, generated by the LSMR procedure. These vectors are orthonormal, satisfy $\Range(\tbOmega) \subseteq \Range(\JJ^\transpose)$, and can be utilized to help construct $\hPP$.

\begin{algorithm} 
\renewcommand{\baselinestretch}{1.25}\selectfont 
\begin{algorithmic}[1]
    \Require{$\taunorm \in (0,1)$, Jacobian-vector and Jacobian-transpose-vector product oracles, $\vv \mapsto \JJ\vv$ and $\uu \mapsto \JJ^\transpose \uu$.}
    \State $\tbOmega = [\,]$
    \For{$t=1,2,\ldots$}
        \State Apply one step of LSMR to compute the iterate $\vvtt$ for \cref{eqn:indirect normal component computation} and the  Krylov vector $\tbomega_t$.
        \State $\tbOmega \gets [\tbOmega, \tbomega_t]$. \Comment{Optional: store the Krylov basis $\tbOmega \subseteq \Range(\JJ^\transpose)$.}
        \If{$\vvtt$ satisfies \cref{eqn:inexact normal component termination condition} with $\taunorm$} \Comment{Residual available within LSMR.}
            \State \Return{$\vvtt,\tbOmega$}
        \EndIf
    \EndFor
\caption{Inexact Normal Component}
\label{alg:inexact normal component}
\end{algorithmic}
\end{algorithm}

Once the normal step is obtained, the merit parameter can be updated in a manner analogous to the direct case. Specifically, at each iteration we define
\begin{align*}
    \pitrial_k = 
    \begin{cases}
    \displaystyle
    \frac{\dotprod{\grad f_k, \vv_k}}
         {(1-\rho)\gammanorm\vnorm{\cc_k}} 
    & \text{if} \ \vnorm{\cc_k} \neq 0, \\[2ex]
    0 
    & \text{if} \ \vnorm{\cc_k} = 0,
    \end{cases}
    \tageq\label{eqn:indirect merit parameter trial}
\end{align*}
where $\rho \in (0,1)$ is a user-chosen parameter and $\gammanorm \defeq 1 - \taunorm$. We then update the merit parameter as in \cref{eqn:direct merit parameter update}. Using properties of the LSMR iterates, we can bound $\pi_k$ in a similar way to \cref{lemma:direct merit parameter boundedness}.

\paragraph{Tangent Component.} 
We now investigate the conditions on the approximate projector that ensure the tangent component achieves descent. In the absence of $\PPNJ$, we reformulate \cref{eqn:inexact tangent component CR subproblem} to make use of the approximate projection $\hPP$. Specifically, letting $\hPP = \bSSh\bSSh^\transpose$, we consider a conjugate residual formulation analogous to \cref{eqn:inexact tangent component CR subproblem} with 
\begin{align}
\label{eqn:indirect inexact tangent component CR subproblem}
\tHH = \bSSh^\transpose \LHess \bSSh, \qquad \text{and} \qquad \tbgg = \bSSh^\transpose \gradf.
\end{align}
The final $\FLAG=\SOL{}$ and $\FLAG=\LPC{}$ search directions can be formed as $\ww = \bSSh \tddtt$ or $\ww = \bSSh \trrtt_\tangent$, respectively.

Unlike the direct case, the above formulation no longer corresponds to a Krylov approximation to \cref{eqn:tangent component problem}. In particular, $\tbgg = \bSSh^\transpose\gradf$ need not equal $\bSSh^\transpose\Lgrad$ unless $\bSSh = \bSS$. Nevertheless, we will show that many of the descent properties established in the direct case (see \cref{lemma:direct case descent of tangent component}) still hold for the search direction $\ww$ produced by projected CR with \cref{eqn:indirect inexact tangent component CR subproblem}; see \cref{lemma:indirect case descent of tangent component}. By construction, \cref{eqn:indirect inexact tangent component CR subproblem} produces a direction $\ww \in \Range(\hPP)$. Thus, when $\hPP$ approximates $\PPNJ$ well, we may expect $\ww$ to lie close to $\Null(\JJ)$. As discussed earlier, controlling the quantity $\vnorm{\JJ\ww}$ is crucial for guaranteeing descent of the merit function. Indeed, inspecting \cref{eqn:indirect case direction derivative Jw bound} shows that the term $\vnorm{\JJ\ww}$ can be balanced against the descent due to $\dotprod{\gradf,\ww}$, provided that
\begin{align*}
    \vnorm{\JJ\ww} \leq -\frac{\taunull}{\pi}\dotprod{\gradf,\ww},
    \tageq\label{eqn:null space termination condition}
\end{align*}
where $\taunull \in (0,1)$ is a user-chosen null-space tolerance and $\pi$ is the current merit parameter. This condition is well posed since, by \cref{lemma:projected CR properties}, we expect $\dotprod{\gradf,\ww} \leq 0$. In particular, \cref{eqn:null space termination condition} is automatically satisfied in the exact projection case $\hPP=\PPNJ$.

The primary difficulty arises from the fact that $\hPP$ must be chosen \textit{prior to} the computation of the tangent component. We address this difficulty by choosing $\hPP$ to enforce \cref{eqn:null space termination condition} for the \textit{first} CR iterate, i.e.,
\begin{align*}
\vnorm{\JJ\hPP\gradf}
\le
\frac{\taunull}{\pi}
\dotprod{\gradf,\hPP\gradf}.
\tageq\label{eqn:null space gradient termination condition}
\end{align*}
Note that both $\ddone$ and $\rrzero$ are negative multiples of $\hPP\gradf$, so \cref{eqn:null space gradient termination condition} is sufficient to enforce \cref{eqn:null space termination condition} for the first possible $\FLAG=\SOL{}$ and $\FLAG=\LPC{}$ exits.
We maintain \cref{eqn:null space termination condition} across CR iterations with a simple backtracking safeguard. In particular, for each potential $\FLAG = \SOL{}$ iterate, $\ddtt$, if \cref{eqn:null space termination condition} fails to hold, we terminate early and revert to the previous iterate\footnote{\cref{lemma:projected CR properties} shows that even in the worst case, when \cref{eqn:null space termination condition} holds only for the first iterate, both $\FLAG=\SOL$ and $\FLAG=\LPC$ directions retain sufficient descent properties. See also \citep{Smee2025FirstishOM}.}, $\ddttm$. 
By \cref{eqn:null space gradient termination condition}, the first candidate satisfies \cref{eqn:null space termination condition}; thereafter, the safeguard only advances past iterates that satisfy this condition. Hence, if the current iterate fails the test, the previous iterate is valid. 

Once \cref{eqn:null space termination condition} is satisfied, the termination conditions are the usual SOL/LPC checks outlined in \cref{sec:CR algorithm background} with some minor modification. In particular, if \cref{eqn:tangent component termination condition} is satisfied then $\ddtt$ can be returned. On the other hand, if the LPC condition \cref{eqn:LPC condition} is triggered, we test whether \cref{eqn:null space termination condition} holds for $\rrtt$. If this test passes, then $\rrtt$ is a valid search direction. Otherwise, $\ddtt$ serves as a suitable search direction since the LPC condition has not been triggered through iteration $t-1$ (i.e., \cref{lemma:projected CR properties} is applicable) and $\ddtt$ passed \cref{eqn:null space termination condition}. We incorporate this termination scheme into \cref{alg:conjugate residual} via \Cref{term:indirect}. 
\begin{terminationblockalg}[ht]{Indirect Termination Block}[term:indirect]
    \Require{Current merit parameter $\pi$, LPC tolerance $\sigma>0$, tangent tolerance $\tautan>0$, and null space tolerance $\taunull \in (0,1)$.}
    \If{$\vnorm{\JJ\ddtt} > -\dfrac{\taunull}{\pi}\dotprod{\bgg,\ddtt}$} \Comment{Will not trigger at $t=1$, due to \cref{eqn:null space gradient termination condition}.}
        \State \Return{$\ww=\ddttm$, $\FLAG=\SOL$}
        \Comment{Current iterate violates \cref{eqn:null space termination condition}; return previous iterate.} \label{line:null-space-termination-test}
    \ElsIf{\cref{eqn:tangent component termination condition} is satisfied with $\tau=\tautan$}
        \State \Return{$\ww=\ddtt$, $\FLAG=\SOL$}
        \Comment{Solution identified.} \label{line:SOL termination}
    \ElsIf{\cref{eqn:LPC condition} is satisfied}
        \If{$\vnorm{\JJ\rrtt}
        >
        -\dfrac{\taunull}{\pi}\dotprod{\bgg,\rrtt}$}
            \State \Return{$\ww=\ddtt$, $\FLAG=\SOL$}
            \Comment{Residual candidate violates \cref{eqn:null space termination condition}.} \label{line:LPC detection termination residual not tangent}
        \Else
            \State \Return{$\ww=\rrtt$, $\FLAG=\LPC$} \label{line:LPC detection termination residual is tangent}
        \EndIf
    \Else
        \State \textbf{continue}
    \EndIf
\end{terminationblockalg}

\paragraph{Outer Termination Condition.}
Another consequence of using an inexact projection $\hPP$ is that the primal stationarity test 
$\vnorm{\PPNJ\gradf} \leq \epsp$ can no longer be evaluated directly. Indeed, since $\PPNJ$ is unavailable, the termination condition \cref{eqn:termination condition} is not explicitly computable. Fortunately, the construction of $\hPP$ guarantees that it provides an \emph{over-estimate} of the null space in the sense that $\Null(\JJ) \subseteq \Range(\hPP)$. This immediately implies the bound
\[
\vnorm{\PPNJ \gradf} \leq \vnorm{\hPP \gradf}.
\]
For this reason we adopt the computable proxy for the tangent termination condition as
\begin{align*}
    \vnorm{\hPP\gradf} \leq \epsp. \tageq\label{eqn:indirect primal termination condition}
\end{align*}
The overall outer termination condition (in place of \cref{eqn:termination condition}) for the indirect algorithm is therefore
\begin{align*}
    \vnorm{\cc} \leq \epsc, 
    \quad \text{and} \quad 
    \vnorm{\hPP\gradf} \leq \epsp.
    \tageq\label{eqn:indirect termination condition}
\end{align*}
This termination proxy has implications for the tangent-step computation, which we discuss in the following remark.

\begin{remark}
    \label{rem:tangent zero}
    If $\|\hPP\gradf(\xx)\|$ becomes small, i.e., the problem becomes close to primal stationary, the benefit of computing a tangent step diminishes. Indeed, by \cref{lemma:projected CR properties},
    \[
    \dotprod{\gradf(\xx),\ww}
    \in
    \bigO{-\vnorm{\hPP\gradf(\xx)}^2},
    \]
    so the descent generated by the tangent update scales proportionally to $\|\hPP\gradf(\xx)\|^2$. Therefore, whenever $\|\hPP\gradf(\xx)\| \leq \epsp$, we set $\ww=0$ and skip the tangent computation. If the constraints remain $\epsc$-infeasible, the iteration then proceeds using only the normal update.
\end{remark}

\paragraph{Dual Update.}

In the absence of a direct factorization of $\JJ$, the least-squares multipliers used in \cref{eqn:least-squares-multipliers} are no longer available in closed form for constructing the Lagrangian Hessian. Nevertheless, the flexibility of our framework allows the least-squares multipliers to be replaced by an approximation. For example, applying LSMR \citep{fongLSMRIterativeAlgorithm2011} to the dual least-squares problem, results in the subproblem
\begin{align*}
    \blambda^{(t)}
    =
    \argmin_{\blambda \in \sK(\JJ \JJ^\transpose, \JJ \gradf)}
    \vnorm{\JJ \JJ^\transpose \blambda + \JJ \gradf}.
    \tageq\label{eqn:LSMR least squares multipliers}
\end{align*}
The solution to this subproblem can be computed using only Jacobian-vector and Jacobian-transpose-vector products. Importantly, we do not require a stringent inexactness tolerance on this dual update. As discussed in \cref{sec:direct algorithm statement}, the multipliers only affect the curvature model used in the tangent computation and do not directly influence the descent properties of the algorithm. Indeed, the zero multipliers $\blambda_{k+1}=0$ remain a valid choice.

Recalling \cref{remark:Lagrangian Hessian boundedness}, a key requirement for the dual update is that the dual variables are bounded so that, in turn, the Lagrangian Hessian remains bounded across the iterations. Fortunately, the LSMR approximation to the least-squares dual update preserves this property. Indeed, since LSMR is equivalent to MINRES applied to the normal equations, we can utilize monotonicity properties of the MINRES iterates. In particular, by \cref{ass:jacobian regularity}, the matrix $\JJ\JJ^\transpose$ is positive definite, and hence \citep[Theorem 3.11]{liu_minres_2022} implies that the iterates produced by \cref{eqn:LSMR least squares multipliers} satisfy
\begin{align*}
    \vnorm{\blambda^{(t)}}
    \leq
    \vnorm{(\JJ\JJ^\transpose)^{-1}\JJ \gradf}
    =
    \vnorm{\JJ^\ddagger \gradf}
    =
    \vnorm{\blambda^\star}.
\end{align*}
Therefore, by \cref{remark:Lagrangian Hessian boundedness}, the LSMR dual iterates $\blambda^{(t)}$ and the corresponding Lagrangian Hessians remain uniformly bounded. We therefore take the bound in \cref{eqn:Lagrangian Hessian boundedness} to hold throughout this section.

\paragraph{Inexact Projection.}

We now turn to the construction of an approximate projector \cref{eqn:approximate projector formula} satisfying \cref{eqn:null space gradient termination condition}. One subtlety is that, as the range space basis $\hbOmega$ is expanded, the right-hand side of \cref{eqn:null space gradient termination condition} decreases. Indeed, by orthogonality of the projection,
\[
\dotprod{\gradf, \hPP \gradf} = \vnorm{\hPP \gradf}^2,
\]
which is clearly nonincreasing as $\hbOmega$ expands.
This could suggest that \cref{eqn:null space gradient termination condition} becomes increasingly stringent as the projector improves. However, this issue is avoided because the tangent step is disabled once $\|\hPP\gradf\|$ becomes small; see \cref{rem:tangent zero}. Indeed, the range-finding procedure only needs to continue while a nonzero tangent computation is required, in which case $\| \hPP \gradf \| \geq \epsp$. Therefore, the right-hand side of \cref{eqn:null space gradient termination condition} is bounded below, and the effective worst-case termination threshold is given by a uniform bound
\begin{align*}
    \vnorm{\JJ\hPP\gradf}
    \le
    \frac{\taunull}{\pi}\epsp^2.
\end{align*}
Our framework is flexible with respect to the specific range-finding procedure used to construct $\hbOmega$. Any method that produces a partial orthonormal basis for $\Range(\JJ^\transpose)$ and a corresponding projector that satisfies \cref{eqn:null space gradient termination condition} is sufficient. For instance, randomized range-finding methods \citep{martinssonRandomizedNumericalLinear2021,halkoFindingStructureRandomness2010} provide one possible approach, and may be attractive when block Jacobian-transpose-vector products can be applied efficiently. Thus, the appropriate range-finder may depend not only on the structure of the constraint Jacobian, but also on low-level computational considerations. Here, we focus on one particularly natural choice: a Krylov range-finder that targets the component of $\gradf$ lying in $\Range(\JJ^\transpose)$ using only Jacobian- and Jacobian-transpose-vector products.

\begin{example}[Krylov Range-Finder] \label{example:Krylov Range-Finder}

The Krylov range-finder constructs a (partial) basis for $\Range(\JJ^\transpose)$ from the Krylov vectors
\[
    \sK_t(\JJ^\transpose \JJ, \JJ^\transpose \JJ \gradf)
    \subseteq \Range(\JJ^\transpose).
\]
In particular, a Golub-Kahan bidiagonalization \citep{golub_matrix_2013} can be applied to form an orthonormal basis for the Krylov subspace, i.e.,
\[
    \hbOmegatt
    =
    \texttt{ortho}\bigl(\sK_t(\JJ^\transpose \JJ, \JJ^\transpose \JJ \gradf)\bigr),
\]
using only Jacobian-vector and Jacobian-transpose-vector products. The approximate projection is then computed as
\[
    \hPPtt = \eye - \hbOmegatt(\hbOmegatt)^\transpose.
\]
The key advantage of this construction is that the Krylov space is generated from $\JJ^\transpose\JJ\gradf$ and therefore directly targets the component of $\gradf$ contained in $\Range(\JJ^\transpose)$. This aligns the range-finding process with \cref{eqn:null space gradient termination condition}, rather than attempting to approximate the entire range space uniformly.

To make this precise, observe that the corresponding projection problem can be formulated as
\begin{align*}
    \hbOmegatt (\hbOmegatt)^\transpose \gradf
    =
    \zz^{(t)}
    &=
    \argmin_{\zz \in \sK_t(\JJ^\transpose \JJ, \JJ^\transpose \JJ \gradf)}
    \vnorm{\zz - \gradf} \\
    &=
    \argmin_{\zz \in \sK_t(\JJ^\transpose \JJ, \JJ^\transpose \JJ \gradf)}
    \vnorm{\zz - \PPRJT \gradf},
\end{align*}
with
\begin{align*}
    \hPPtt \gradf = \gradf - \zz^{(t)}.
\end{align*}

Any vector $\zz \in \sK_t(\JJ^\transpose \JJ, \JJ^\transpose \JJ \gradf)$ can be expressed as
\[
    \zz
    =
    q_{t-1}(\JJ^\transpose \JJ)\JJ^\transpose \JJ \PPRJT \gradf,
\]
where $q_{t-1}$ is a polynomial of degree at most $t-1$. Setting $p_t(\lambda) = 1 -  q_{t-1}(\lambda)\lambda$ gives
\begin{align*}
    \PPRJT \gradf - \zz
    =
    p_t(\JJ^\transpose\JJ) \PPRJT \gradf.
\end{align*}
Moreover, $p_t$ is a polynomial of degree at most $t$ satisfying $p_t(0)=1$. Hence
\begin{align*}
    \vnorm{\PPRJT\gradf-\zz^{(t)}}
    &=
    \min_{\zz \in \sK_t(\JJ^\transpose \JJ, \JJ^\transpose \JJ \gradf)}
    \vnorm{\zz - \PPRJT \gradf} \\
    &=
    \min_{\substack{p \in \sP_t\\ p(0)=1}}
    \vnorm{p(\JJ^\transpose\JJ)\PPRJT \gradf},
\end{align*}
where $\sP_t$ is the set of polynomials of degree at most $t$. This polynomial characterization can be used to bound $\|\JJ \hPPtt \gradf\|$ directly. In particular,
\begin{align*}
    \vnorm{\JJ \hPP^{(t)} \gradf}
    &=
    \vnorm{\JJ(\gradf -  \zz^{(t)})} \\
    &= \vnorm{\JJ(\PPRJT\gradf -  \zz^{(t)})} \\
    &\leq
    \vnorm{\JJ} \vnorm{\PPRJT\gradf -  \zz^{(t)}} \\
    &=
    \sigma_1
    \min_{\substack{p \in \sP_t\\ p(0)=1}}
    \vnorm{p(\JJ^\transpose\JJ) \PPRJT \gradf} \\
    &\leq
    \sigma_1
    \min_{\substack{p \in \sP_t\\ p(0)=1}}
    \max_{\lambda \in [\sigma_c^2, \sigma_1^2]}
    |p(\lambda)|\vnorm{\PPRJT \gradf},
\end{align*}
where on the final line we use the fact that the spectrum of $\JJ^\transpose \JJ$, restricted to $\Range(\JJ^\transpose)$, is contained in $[\sigma_c^2,\sigma_1^2]$. Applying a standard minimax polynomial bound, e.g., \citep[Theorem 3.1.1]{greenbaumIterativeMethodsSolving1997}, gives
\begin{align*}
    \vnorm{\JJ \hPPtt \gradf}
    \leq
    2\sigma_1
    \left(
        \frac{\sigma_1/\sigma_c - 1}
        {\sigma_1/\sigma_c + 1}
    \right)^{t}
    \vnorm{\PPRJT\gradf}.
\end{align*}
The rate indicated by this bound is geometric, so the number of Krylov iterations required to satisfy \cref{eqn:null space gradient termination condition} scales logarithmically with the desired projection accuracy. Thus, the Krylov range-finder need not identify the full range space $\Range(\JJ^\transpose)$; it can terminate as soon as the induced projection is sufficiently accurate. Moreover, in exact arithmetic, finite termination occurs once the Krylov space contains the active range space component $\PPRJT\gradf$.

Finally, this Krylov construction naturally accommodates warm starts. A warm-start basis $\tbOmega$, for example obtained from the inexact normal solve, can be incorporated by orthogonalizing the Krylov vectors away from $\Range(\tbOmega)$. Specifically, define
\[
    \tPP = \eye-\tbOmega\tbOmega^\transpose,
\]
and
\[
    \tJJ = \JJ\tPP.
\]
We then construct the Krylov basis as
\[
    \hbOmegatt
    =
    \texttt{ortho}\bigl(\sK_t( \tJJ^\transpose\tJJ, \tJJ^\transpose\tJJ\gradf)\bigr).
\]
The projection problem is given by
\[
    \zz^{(t)}
    =
    \argmin_{\zz \in \sK_t( \tJJ^\transpose\tJJ, \tJJ^\transpose\tJJ\gradf)}
    \vnorm{\zz - \gradf}.
\]
This construction ensures that $\zz^{(t)} \in \Range(\JJ^\transpose)$ and $\tbOmega^\transpose \hbOmegatt = 0$. The final approximate projector is then obtained by augmenting the Krylov basis with the warm-start basis:
\begin{align*}
    \hPP = \eye - \tbOmega\tbOmega^\transpose - \hbOmegatt(\hbOmegatt)^\transpose.
\end{align*}
\end{example}

\paragraph{The Algorithm.}

We summarize the developments of this section in \cref{alg:indirect primal-dual algorithm}.

\begin{algorithm}[ht]
\renewcommand{\baselinestretch}{1.25}\selectfont
\begin{algorithmic}[1]
\Require{Termination tolerances $(\epsc,\epsp)\in(0,1)^2$, normal and null space tolerances $\taunorm,\taunull\in(0,1)$, tangent tolerance $\tautan > 0$, curvature tolerance $\sigma>0$, initial point $(\xx_0,\blambda_0)$, initial merit parameter $\pi_{-1}>0$, merit update factor $\rho\in(0,1)$, line search parameter $\eta\in(0,1)$, and backtracking factor $\theta\in(0,1)$.}

\For{$k=0,1,2,\ldots$}

    \State $(\vv_k,\tbOmega_k)=$ [\cref{alg:inexact normal component}]($\JJ_k,\taunorm$).
    \Comment{Inexact normal component solve.}

    \State Compute $\pi_k$ according to \cref{eqn:indirect merit parameter trial,eqn:direct merit parameter update}.

    \State Apply range-finder (e.g., \cref{example:Krylov Range-Finder}) to identify $\hPP_k = \eye - \hbOmega_k\hbOmega_k^\transpose$ satisfying \cref{eqn:null space gradient termination condition} or \cref{eqn:indirect primal termination condition}.
    \If{\cref{eqn:indirect termination condition} is satisfied}
        \State \Return $\xx_k$
        \Comment{Terminate with an $(\epsc,\epsp)$-stationary point.}
    \EndIf

    \If{$\vnorm{\hPP_k\gradf_k}\le\epsp$} \Comment{See \cref{rem:tangent zero}}
        \State $\ww_k=0$
    \Else
        \State Update $\blambda_{k+1}$, e.g., using LSMR as in \cref{eqn:LSMR least squares multipliers}.
        \State Set $\bgg_k=\gradf_k$ and $\HH_k=\LHess(\xx_k,\blambda_{k+1})$.
        \State $(\ww_k, \FLAG_k)\gets$[\cref{alg:conjugate residual}]$(\HH_k,\bgg_k,\hPP_k,\pi_k,\sigma,\tautan,\taunull)$.
        \Comment{\cref{alg:conjugate residual} with \Cref{term:indirect}.}
    \EndIf

    \State $\pp_k=\vv_k+\ww_k$.
    \State Use backtracking line search to find $\alpha_k \leq 1$ satisfying \cref{eqn:line search descent criteria}
    \State $\xx_{k+1}=\xx_k+\alpha_k\pp_k$.

\EndFor

\caption{Indirect Primal-Dual Newton-MR Algorithm}
\label{alg:indirect primal-dual algorithm}
\end{algorithmic}
\end{algorithm}

\subsection{Convergence Analysis of Algorithm~\ref{alg:indirect primal-dual algorithm}}  \label{sec:indirect convergence analysis}
We now establish the global convergence properties of \cref{alg:indirect primal-dual algorithm}.

\begin{theorem} \label{thm:indirect case convergence}
    Suppose \cref{ass:basic assumptions,ass:jacobian regularity} hold and let $(\epsc,\epsp)\in(0,1)^2$. Then \cref{alg:indirect primal-dual algorithm} reaches a point satisfying \cref{eqn:termination condition} in at most 
    \[ K \defeq \left\lceil \frac{(f_0 - \flow)/\pi_{-1} + \vnorm{\cc_0}}{\min\{\Gamma_c', \Gamma_p'\} \min\{\epsc,\epsp^2\}} \right\rceil + 1, \]
    iterations, where $\Gamma_c',\Gamma_p'>0$ are constants defined in \cref{eq:indirect case Gamma}.
\end{theorem}
\cref{thm:indirect case convergence} guarantees that \cref{alg:indirect primal-dual algorithm} identifies a point $\xx_k$ satisfying \cref{eqn:termination condition}. If dual variables are required to form a pair satisfying \cref{eqn:epsilon KKT stationary point}, they can be obtained by approximately solving
\begin{align*}
    \min_{\blambda} \vnorm{\JJ^\transpose \blambda + \gradf},
\end{align*}
by \cref{lemma:termination condition equivalence}. Note that this problem need only be solved inexactly to a tolerance $\vnorm{\JJ^\transpose_k \blambda_k + \gradf_k} \leq \epsp$. Additionally, we remark that the overall dependence on $\epsp$ and $\epsc$ obtained in \cref{thm:indirect case convergence} matches that of \cref{thm:direct case convergence}.

We begin the analysis with the descent properties of the inexactly computed normal component.

\begin{lemma}[Normal Component Properties] \label{lemma:inexact normal component properties}
Under \cref{ass:jacobian regularity,ass:basic assumptions}, the output $\vv_k$ of \cref{alg:inexact normal component} within \cref{alg:indirect primal-dual algorithm} satisfies
\begin{align*}
    \vnorm{\JJ_k \vv_k + \cc_k}
    &\leq
    (1-\gammanorm) \vnorm{\cc_k},
    \tageq\label{eqn:inexact normal component descent}\\
    \nu_l \vnorm{\cc_k}
    &\leq
    \vnorm{\vv_k}
    \leq
    \nu_u \vnorm{\cc_k}.
    \tageq\label{eqn:inexact normal component norm bounds}
\end{align*}
where
\[
\gammanorm = 1-\taunorm,
\qquad
\nu_u \defeq \frac{1}{\sigma_{\JJ}},
\qquad
\nu_l \defeq \frac{\sigma_{\JJ}}{\kappa_{\JJ}^2}.
\]
\end{lemma}
\begin{proof}
To simplify notation, we drop the outer iteration index and work with the inner iterates of \cref{alg:inexact normal component}. The termination condition \cref{eqn:inexact normal component termination condition} gives the first claim.
We now prove the norm bounds. Since $\JJ^\transpose \JJ$ is positive semidefinite and the LSMR subproblem in \cref{eqn:indirect normal component computation} is equivalent to CR applied to the normal equations, we appeal to the monotonicity of the CR iterates \citep{liu_minres_2022}, provided no zero-curvature direction occurs in $\sK_t(\JJ^\transpose\JJ,\JJ^\transpose\cc)$, i.e.,
\[
\dotprod{\rrii_\normal,\JJ^\transpose\JJ\rrii_\normal}>0, \quad \text{where} \quad \rrii_\normal = \JJ^\transpose\JJ\vvii + \JJ^\transpose\cc.
\]

If $\cc=0$, then $\vv^{(0)}=0$ and the result is trivial. Otherwise $\cc\neq 0$ and the method performs at least one iteration. Suppose a zero-curvature direction occurs at some $t>0$ prior to \cref{eqn:inexact normal component termination condition}. Then
\begin{align*}
0
&=
\dotprod{\rrtt_\normal,\JJ^\transpose\JJ\rrtt_\normal}
=
\vnorm{\JJ\rrtt_\normal}^2
=
\vnorm{\JJ\JJ^\transpose(\JJ\vvtt+\cc)}^2,
\end{align*}
which implies $\JJ^\transpose(\JJ\vvtt+\cc)=0$, hence $\rrtt_\normal=0$ and $\vvtt=-\JJ^\dagger\cc$. In this case, the termination condition \cref{eqn:inexact normal component termination condition} is necessarily satisfied, which is a contradiction. Hence, zero curvature cannot occur before termination.

Let $g_\normal$ denote the grade of $\JJ^\transpose\cc$ with respect to $\JJ^\transpose\JJ$. Since $\Range(\JJ^\transpose)=\Range(\JJ^\transpose\JJ)$, \citep[Fact 1]{limConjugateDirectionMethods2024} and \citep[Theorem 3.11]{liu_minres_2022} yield
\begin{align*}
    \vnorm{\vv_1} \leq \cdots \leq \vnorm{\vv_t} \leq \cdots \leq \vnorm{\vv^{(g_\normal)}} = \vnorm{\JJ^\dagger \cc}.
\end{align*}

The upper bound in \cref{eqn:inexact normal component norm bounds} follows as in \cref{lemma:direct normal component descent and norm properties}. For the lower bound, explicitly computing the first CR step gives
\begin{align*}
\vnorm{\vv_1}
&=
\frac{\dotprod{\JJ^\transpose\cc,\JJ^\transpose\JJ\JJ^\transpose\cc}}
     {\vnorm{\JJ^\transpose\JJ\JJ^\transpose\cc}^2}
\,\vnorm{\JJ^\transpose\cc} \\
&=
\frac{\dotprod{\JJ\JJ^\transpose\cc,\JJ\JJ^\transpose\cc}}
     {\dotprod{\JJ\JJ^\transpose\cc,\JJ\JJ^\transpose\JJ\JJ^\transpose\cc}}
\,\vnorm{\JJ^\transpose\cc}
\;\geq\;
\frac{\sigma_{\JJ}}{\kappa_{\JJ}^2}\vnorm{\cc},
\end{align*}
where we used \cref{ass:jacobian regularity} in the final inequality.
\end{proof}

Using \cref{lemma:inexact normal component properties}, we can establish boundedness of the merit parameter in a manner similar to \cref{lemma:direct merit parameter boundedness}.
\begin{lemma}[Merit Parameter Boundedness] \label{lemma:indirect case merit parameter boundedness}
    Under \cref{ass:jacobian regularity,ass:basic assumptions}, the penalty parameter selected in \cref{alg:indirect primal-dual algorithm} is uniformly bounded, i.e., there exists $0 < \pit < \infty$ such that $$\pi_{-1} \leq \pi_k \leq \pit, \quad  \forall k.$$
\end{lemma}
\begin{proof}
    Applying \cref{lemma:inexact normal component properties,ass:basic assumptions} to \cref{eqn:indirect merit parameter trial} yields
    \begin{align*}
        \frac{\dotprod{\gradf_k,\vv_k}}
        {(1-\rho)\gammanorm \vnorm{\cc_k}}
        \leq
        \frac{\vnorm{\gradf_k}\vnorm{\vv_k}}
        {(1-\rho)\gammanorm \vnorm{\cc_k}}
        \leq
        \frac{\kappa_g \nu_u}
        {(1-\rho)\gammanorm}.
    \end{align*}
    Define
    \begin{align*}
        \pit \defeq
        \max\left\{
        \frac{\kappa_g \nu_u}
        {(1-\rho)\gammanorm},
        \pi_{-1}
        \right\}.
    \end{align*}
    The merit parameter update \cref{eqn:direct merit parameter update} then gives the desired bounds.
\end{proof}

Next we consider the descent and null space properties of the tangent step.

\begin{lemma}[Descent Properties of Tangent Component] \label{lemma:indirect case descent of tangent component}
    Suppose \cref{ass:basic assumptions} holds and
    \[
        \vnorm{\hPP_k \gradf_k} > \epsp,
    \]
    in \cref{alg:indirect primal-dual algorithm}. The call to \cref{alg:conjugate residual} with \Cref{term:indirect} returns a tangent search direction $\ww_k$ satisfying
    \begin{align*}
        \dotprod{\gradf_k, \ww_k} &\leq - \gammatan \vnorm{\ww_k}^2, \tageq\label{eqn:tangent component descent}\\
        \vnorm{\ww_k} &\geq \omegatan \vnorm{\hPP_k \gradf_k }, \tageq\label{eqn:tangent component norm bound} \\
        \vnorm{\JJ_k \ww_k} &\leq -\frac{\taunull}{\pi_k}\dotprod{\gradf_k, \ww_k}. \tageq\label{eqn:tangent component null space bound}
    \end{align*}
    The constants, $\gammatan$ and $\omegatan$, are as in \cref{lemma:direct case descent of tangent component}.
\end{lemma}
\begin{proof}
    First consider the case $\FLAG=\SOL$. If \Cref{term:indirect} returns on
    \cref{line:LPC detection termination residual not tangent,line:SOL termination},
    then \cref{eqn:tangent component null space bound} holds by construction.
    Moreover, since LPC has not been detected up to iteration $t-1$,
    \cref{lemma:projected CR properties} together with
    \cref{eqn:Lagrangian Hessian boundedness} yields
    \begin{align*}
        \dotprod{\gradf_k,\ww_k}
        &\leq -\sigma \vnorm{\ww_k}^2,
        \\
        \vnorm{\ww_k}
        &\geq
        \frac{\sigma}{\kappa_{\sL}^2}
        \vnorm{\bSSh_k^\transpose \gradf_k}
        =
        \frac{\sigma}{\kappa_{\sL}^2}
        \vnorm{\hPP_k\gradf_k},
    \end{align*}
    where the equality follows from the orthogonality of $\bSSh_k$.
    Alternatively, if termination occurs on
    \cref{line:null-space-termination-test}, then
    \cref{eqn:null space gradient termination condition} guarantees that
    $t\ge 2$. Since the null-space test \cref{eqn:null space termination condition} was not triggered at iteration
    $t-1$, the returned direction
    $\ww_k=\ddttm_k$ satisfies
    \cref{eqn:tangent component null space bound}. The descent and norm
    bounds above continue to hold because LPC was not detected up to
    iteration $t-2$.

    Next consider the case $\FLAG=\LPC$.
    Since either \cref{eqn:tangent component termination condition} did not
    trigger or $t=0$, \cref{lemma:projected CR properties} implies
    \begin{align*}
        \dotprod{\gradf_k,\ww_k}
        &= -\vnorm{\ww_k}^2,
        \\
        \vnorm{\ww_k}
        &\geq
        \frac{\tautan}
        {\sqrt{\tautan^2+\kappa_{\sL}^2}}
        \vnorm{\bSSh_k^\transpose \gradf_k}
        =
        \frac{\tautan}
        {\sqrt{\tautan^2+\kappa_{\sL}^2}}
        \vnorm{\hPP_k\gradf_k}.
    \end{align*}
    If termination occurs on \cref{line:LPC detection termination residual is tangent}, then \cref{eqn:tangent component null space bound} holds by construction. The only other case to consider is if the LPC condition is triggered at $t=0$ (i.e., on \cref{line:early LPC termination} of \cref{alg:conjugate residual}). However, in this case \cref{eqn:tangent component null space bound} follows directly from \cref{eqn:null space gradient termination condition}.

    Collecting the above bounds establishes
    \cref{eqn:tangent component descent,eqn:tangent component norm bound,eqn:tangent component null space bound}.
\end{proof}

Next, we show that the tolerances enforced on the normal and tangent components, together with the choice of merit parameter, guarantee descent of the merit function. Subsequently, we show that the line search criterion \cref{eqn:line search descent criteria} is satisfied for a small enough step size.

\begin{lemma}[Descent on Merit Function]\label{lemma:indirect primal update descent on merit function}
    Under \cref{ass:jacobian regularity,ass:basic assumptions}, the primal search direction $\pp_k$ and merit parameter $\pi_k$ computed in \cref{alg:indirect primal-dual algorithm} satisfy 
    \begin{align*}
        \Db(\xx_k, \pp_k; \pi_k) \leq -(1- \taunull)\gammatan\vnorm{\ww_k}^2 - \gammanorm \rho \pi_k \vnorm{\cc_k}, \tageq\label{eqn:indirect case directional derivative upper bound}
    \end{align*}
    where $\Db$ is defined in \cref{lemma:descent lemma}.
\end{lemma} 
\begin{proof}
Recall from \cref{lemma:descent lemma} that
\begin{align*}
    \Db(\xx_k,\pp_k;\pi_k)
    &= \dotprod{\gradf_k,\pp_k}
    + \pi_k\big(\vnorm{\JJ_k\pp_k+\cc_k}-\vnorm{\cc_k}\big) \\
    &\leq
    \dotprod{\gradf_k,\vv_k}
    + \dotprod{\gradf_k,\ww_k}
    + \pi_k \vnorm{\JJ_k\ww_k}
    + \pi_k\big(\vnorm{\JJ_k\vv_k+\cc_k}-\vnorm{\cc_k}\big).
\end{align*}

\paragraph{Case 1: $\vnorm{\hPP_k\gradf_k} > \epsp$.}
In this case, \cref{alg:indirect primal-dual algorithm} produces both $\vv_k$ and $\ww_k$ satisfying
\cref{eqn:inexact normal component descent,eqn:tangent component descent,eqn:tangent component null space bound}.
Moreover, the merit parameter update (\cref{eqn:indirect merit parameter trial,eqn:direct merit parameter update}) implies
\begin{align}
\dotprod{\gradf_k,\vv_k}
\leq \pi_k \gammanorm (1-\rho)\vnorm{\cc_k}.
\label{eqn:merit parameter lower bound full rank}
\end{align}

Substituting these bounds gives
\begin{align*}
    \Db(\xx_k,\pp_k;\pi_k)
    &\leq
    \pi_k \gammanorm(1-\rho)\vnorm{\cc_k} + \dotprod{\gradf_k,\ww_k}
    - \taunull \dotprod{\gradf_k,\ww_k} - \pi_k \gammanorm \vnorm{\cc_k} \\
    &\leq
    -(1-\taunull)\gammatan \vnorm{\ww_k}^2
    - \pi_k \gammanorm \rho \vnorm{\cc_k}.
\end{align*}

\paragraph{Case 2: $\vnorm{\hPP_k\gradf_k} \leq \epsp$.}
Here, $\ww_k=0$, and \cref{eqn:inexact normal component descent} together with
\cref{eqn:merit parameter lower bound full rank} yield
\begin{align*}
    \Db(\xx_k,\pp_k;\pi_k)
    &\leq
    \dotprod{\gradf_k,\vv_k}
    + \pi_k\big(\vnorm{\JJ_k\vv_k+\cc_k}-\vnorm{\cc_k}\big) \\
    &\leq
    \pi_k \gammanorm(1-\rho)\vnorm{\cc_k}
    - \pi_k \gammanorm \vnorm{\cc_k} \\
    &=
    -\pi_k \gammanorm \rho \vnorm{\cc_k}.
\end{align*}

This coincides with \cref{eqn:indirect case directional derivative upper bound}, since $\ww_k=0$.
\end{proof}

\begin{lemma}[Line Search Termination] \label{lemma:indirect line search termination}
    Under \cref{ass:basic assumptions,ass:jacobian regularity}, there exists $\alphat > 0$, such that, for all iterations of \cref{alg:indirect primal-dual algorithm}, the step size selected by the line search satisfies $\alpha_k \geq \alphat$.
\end{lemma}
\begin{proof}
    Similarly to the proof of \cref{lemma:direct line search termination}, we have
    \begin{align*}
        \Delta_k(\alpha) \leq \alpha(1-\eta)\Db_k + \frac{\alpha^2}{2} \left(\pi_kL_c+L_f \right) \vnorm{\pp_k}^2.
    \end{align*}
    Since $\vv_k$ and $\ww_k$ may no longer be orthogonal, we bound  
    \[\vnorm{\pp_k}^2 \leq 2\vnorm{\vv_k}^2+2\vnorm{\ww_k}^2.\] 
    This bound, along with
    \cref{lemma:indirect primal update descent on merit function,lemma:indirect case merit parameter boundedness,lemma:inexact normal component properties}, yields
    \begin{align*}
        \Delta_k(\alpha) &\leq -\alpha(1-\eta) \left( (1-\taunull)\gammatan\vnorm{\ww_k}^2 + \pi_k\gammanorm\rho\vnorm{\cc_k} \right) + \alpha^2 \left( \pi_kL_c+L_f \right) \left(
        \vnorm{\ww_k}^2+\vnorm{\vv_k}^2 \right) \\
        &\leq -\alpha(1-\eta) \left((1-\taunull)\gammatan\vnorm{\ww_k}^2 + \pi_{-1}\gammanorm\rho\vnorm{\cc_k} \right) + \alpha^2 \left( \pit L_c+L_f \right) \left( \vnorm{\ww_k}^2+\nu_u^2\vnorm{\cc_k}^2 \right).
    \end{align*}
    By \cref{ass:item:constraint is bounded}, we have
    \[
        \nu_u^2\vnorm{\cc_k}^2 \leq \kappa_c\nu_u^2\vnorm{\cc_k},
    \]
    and therefore
    \[
        \vnorm{\ww_k}^2+\nu_u^2\vnorm{\cc_k}^2 \leq \max\{1,\kappa_c\nu_u^2\}
        \left(\vnorm{\ww_k}^2+\vnorm{\cc_k}\right).
    \]
    Moreover,
    \begin{align*}
        &(1-\taunull)\gammatan\vnorm{\ww_k}^2 + \pi_{-1}\gammanorm\rho\vnorm{\cc_k} \\
        &\qquad \geq \min\left\{ (1-\taunull)\gammatan, \pi_{-1}\gammanorm\rho \right\} \left( \vnorm{\ww_k}^2+\vnorm{\cc_k} \right).
    \end{align*}
    Combining the above estimates gives
    \begin{align*}
        \Delta_k(\alpha) \leq \alpha \Bigg( -(1-\eta) \min\left\{(1-\taunull)\gammatan, \pi_{-1}\gammanorm\rho \right\} + \alpha
        \left(\pit L_c+L_f \right) \max\left\{ 1,\kappa_c\nu_u^2 \right\} \Bigg) \left( \vnorm{\ww_k}^2+\vnorm{\cc_k} \right).
    \end{align*}
    Hence, the line-search criterion is satisfied for any step size $\alpha$ such that
    \[
        \alpha \leq \frac{(1-\eta) \min\left\{ (1-\taunull)\gammatan, \pi_{-1}\gammanorm\rho \right\}}{\left(\pit L_c+L_f\right)\max\left\{1,\kappa_c\nu_u^2\right\}}.
    \]
    The line search in \cref{alg:indirect primal-dual algorithm} either terminates with a unit step size or, after backtracking, returns a step length that undershoots the above bound by at most a factor $\theta$. Consequently,
    \begin{align}
        \label{eq:alphat}
        \alphat \defeq \min\left\{ 1, \frac{\theta(1-\eta) \min\left\{ (1-\taunull)\gammatan, \pi_{-1}\gammanorm\rho \right\} }{ \left( \pit L_c+L_f \right) \max\left\{ 1,\kappa_c\nu_u^2 \right\} } \right\},
    \end{align}
    which establishes the result.
\end{proof}

Finally, we are ready to prove \cref{thm:indirect case convergence}. 

\begin{proof}[Proof of \cref{thm:indirect case convergence}]
Define
\begin{align}
\label{eq:indirect case Gamma}
    \Gamma_c' \defeq \alphat\eta\rho\gammanorm,
    \qquad
    \Gamma_p' \defeq \frac{\alphat\eta(1-\taunull)\gammatan\omegatan^2}{\pit},
\end{align}
where $\gammatan$ and $\omegatan$ are as in
\cref{lemma:direct case descent of tangent component},
$\gammanorm$ is as in
\cref{lemma:inexact normal component properties},
and $\alphat$ is as in \cref{eq:alphat}.
Suppose that the termination conditions in
\cref{eqn:indirect termination condition}
fail to hold for iterations $k=0,\ldots,K-1$.
Partition the iteration indices into
\[
    \sK_p \defeq \left\{k\in[K-1]\mid \vnorm{\hPP_k\gradf_k}>\epsp\right\},
    \qquad
    \sK_c \defeq [K-1]\setminus \sK_p.
\]

First consider an iteration $k\in\sK_c$. Since the algorithm has not
terminated, we must have
\[
    \vnorm{\cc_k}>\epsc,
    \qquad
    \vnorm{\hPP_k\gradf_k}\leq\epsp.
\]
Hence, $\ww_k=\mathbf{0}$ and the decrease is entirely due to the
normal component $\vv_k$. By the line search criterion
\cref{eqn:line search descent criteria}, together with
\cref{lemma:indirect line search termination,lemma:indirect primal update descent on merit function},
we obtain
\begin{align*}
    \frac{\phi(\xx_{k+1};\pi_k)-\phi(\xx_k;\pi_k)}{\pi_k}
    &\leq \frac{\alpha_k\eta}{\pi_k}\Db(\xx_k,\pp_k;\pi_k) \\
    &\leq -\alphat\eta\rho\gammanorm\vnorm{\cc_k} \\
    &\leq -\alphat\eta\rho\gammanorm\epsc \\
    &= -\Gamma_c'\epsc.
\end{align*}

Now consider an iteration $k\in\sK_p$. In this case, we use the decrease contributed by the tangent component, since $\vnorm{\cc_k}$ may already satisfy $\vnorm{\cc_k} \leq \epsc$. Applying
\cref{lemma:indirect primal update descent on merit function,lemma:indirect line search termination,lemma:indirect case merit parameter boundedness,lemma:indirect case descent of tangent component},
together with the line search criterion
\cref{eqn:line search descent criteria}, yields
\begin{align*}
    \frac{\phi(\xx_{k+1};\pi_k)-\phi(\xx_k;\pi_k)}{\pi_k}
    &\leq \frac{\alpha_k\eta}{\pi_k}\Db(\xx_k,\pp_k;\pi_k) \\
    &\leq -\frac{\alphat\eta(1-\taunull)\gammatan}{\pit}\vnorm{\ww_k}^2 \\
    &\leq -\frac{\alphat\eta(1-\taunull)\gammatan\omegatan^2}{\pit}
    \vnorm{\hPP_k\gradf_k}^2 \\
    &\leq -\Gamma_p'\epsp^2.
\end{align*}

Since $\sK_c$ and $\sK_p$ partition $[K-1]$, it follows that
\begin{align*}
    \sum_{k=0}^{K-1}\frac{\phi(\xx_{k+1};\pi_k)-\phi(\xx_k;\pi_k)}{\pi_k}
    &= \sum_{k\in\sK_c}\frac{\phi(\xx_{k+1};\pi_k)-\phi(\xx_k;\pi_k)}{\pi_k}
    + \sum_{k\in\sK_p}\frac{\phi(\xx_{k+1};\pi_k)-\phi(\xx_k;\pi_k)}{\pi_k} \\
    &\leq -|\sK_c|\Gamma_c'\epsc - |\sK_p|\Gamma_p'\epsp^2 \\
    &\leq -K\min\{\Gamma_c',\Gamma_p'\}\min\{\epsc,\epsp^2\}.
\end{align*}

On the other hand, by the same telescoping argument used in the proof
of \cref{thm:direct case convergence}, we have
\begin{align*}
    \sum_{k=0}^{K-1}\frac{\phi(\xx_{k+1};\pi_k)-\phi(\xx_k;\pi_k)}{\pi_k}
    \geq -\vnorm{\cc_0} + \frac{\flow-f_0}{\pi_{-1}}.
\end{align*}

Combining the upper and lower bounds gives
\[
    K
    \leq
    \frac{(f_0-\flow)/\pi_{-1}+\vnorm{\cc_0}}
    {\min\{\Gamma_c',\Gamma_p'\}\min\{\epsc,\epsp^2\}},
\]
which contradicts the definition of $K$ in \cref{thm:indirect case convergence}.
\end{proof}

\section{Numerical Experiments} \label{sec:numerical}
We now compare the practical performance of our primal-dual Newton methods against augmented Lagrangian and SQP-based baselines on a number of nonconvex problems and nonlinear constraints.

\begin{itemize}
\item \textbf{Primal-Dual Newton (PDN, our method).}
For our numerical experiments, we implement the direct variant of our method, \cref{alg:direct primal-dual algorithm}, using both least-squares multipliers and the zero-dual update. Several practical modifications are incorporated to improve robustness.
First, we employ a forward-tracking line search for negative-curvature steps. In particular, if a unit step satisfies \cref{eqn:line search descent criteria}, the step size is increased until the condition fails and the largest acceptable step is taken; see, e.g., \citep[Algorithm 3]{liu_newton-mr_2023}. This often improves the exploitation of negative curvature in practice \citep{liu_newton-mr_2023,curtisExploitingNegativeCurvature2019,gouldExploitingNegativeCurvature2000}. 
Second, we deactivate individual step components once their associated optimality conditions have been satisfied. In particular, the normal step is disabled once the iterate is sufficiently feasible, while the tangent step is disabled once primal stationarity is attained. Each step is reactivated if its corresponding optimality condition is subsequently violated. This allows the algorithm to focus computational effort on the remaining unsatisfied optimality condition rather than continuing to reduce a quantity that is already below tolerance. An important consequence is that, when the normal step is disabled, the merit parameter is no longer updated. This prevents the accumulation of unnecessarily large parameter values when the constraint violation is already small.

Most parameters can be assigned typical values: $\epsp=\epsc=10^{-5}$, $\eta=10^{-4}$, $\theta=0.5$, $\pi_{-1}=10^{-6}$, and $\rho=0.1$. Although the theory assumes $\sigma>0$, we found that setting $\sigma=0$ performs well in practice.
The only parameter that typically requires tuning is $\tautan$, which controls the accuracy of the tangent solve. Empirically, only a moderate number of inner iterations are needed for good performance, though the appropriate choice of $\tautan$ depends on problem conditioning. 
We emphasize that this relatively simple tuning requirement contrasts with previous inexact SQP methods, whose performance depends on a considerably larger collection of coupled inner-solver parameters and termination criteria.

\item \textbf{Classical SQP Approach.}
To benchmark our method, we implemented the inexact SQP algorithm of \citep{byrdInexactNonconvexNewtonMethod2010}, which is well suited to the large-scale, deterministic, nonconvex, full-rank setting considered here. Although \citet{byrdInexactNonconvexNewtonMethod2010} use GMRES as the inner solver, they also report that MINRES produced similar performance. Motivated by this observation, and by the substantial memory overhead of GMRES in large-scale settings, we use CR for the inner solves. This choice is consistent with the equivalence between MINRES and CR discussed in \cref{sec:CR algorithm background}, while also simplifying the implementation and aligning the comparison solver with the CR-based inner solves used by our method. Our CR implementation incorporates safeguards that increase regularization and restart the iteration if CR breaks down numerically or fails to terminate within (d+c) iterations. Motivated by \citep{curtis_inexact_2021}, we also replace the Hessian regularization in \citep{byrdInexactNonconvexNewtonMethod2010} with the blended scheme
\[
\bar{\HH}=10^{-i}\HH + (1-10^{-i})\eye,
\]
where \(i\) is initialized at \(0\) and updated by \(i \gets i+1\) whenever the inner solve is restarted. We found this modification to improve stability on the ill-conditioned problems in \cref{sec:PINN}.

For the nonlinear least-squares problem in \cref{sec:NLLS}, we use the parameters from \citep{byrdInexactNonconvexNewtonMethod2010}, except that we set \(\theta=10^{-3}\), where $\theta$ is a tolerance for sufficiently positive curvature in the search direction, since the original implementation assumes access to the full Lagrangian Hessian. For the machine learning problems in \cref{sec:PINN}, we found the original settings too restrictive, likely due to the severe ill-conditioning, and instead used
\[
\kappa=10^{-1},\qquad
\epsilon=10^{-1},\qquad
\beta=10^3,\qquad
\tau=0.2,\qquad
\theta=10^{-6},\qquad
\psi=10^4,\qquad
\sigma=\tau(1-\epsilon).
\]
Additionally, we initialized the merit parameter to \(\pi_{-1}=10^{-4}\).

\item \textbf{Augmented Lagrangian (ALM).}
We also compare against an augmented Lagrangian algorithm. Specifically, we consider a classical approach (\citep[Chapter 17]{nocedal_numerical_2006}), which consists of applying an inner solver to the unconstrained penalized Lagrangian objective 
\begin{align*}
    \sL_\rho (\xx; \blambda) = f(\xx) + \dotprod{\cc(\xx) , \blambda} + \frac{\rho}{2} \vnorm{\cc(\xx)}^2,
\end{align*}
for some positive penalty parameter $\rho$. Once the penalized objective is close to optimal (as measured by $\vnorm{\grad_x \sL_\rho(\xx;\blambda)}$), the dual variables and penalty parameters can be updated. Following \citep[Algorithm 1]{basir_physics_2022}, and to avoid over-prioritizing feasibility, we update the penalty parameter only when two conditions are met: (i) constraint optimality is not yet satisfied, and (ii) the constraint violation has not decreased sufficiently relative to the previous update. For the outer-loop hyperparameters (initial penalty, penalty growth rate, and tolerance on constraint reduction), we use the choices in \citep[Algorithm 1]{basir_physics_2022}.
For the inner solver, we use L-BFGS with a zoom line search enforcing the Wolfe descent and curvature conditions; we use standard defaults (line search constants $10^{-4}$ and $0.9$, respectively, and memory size $20$). 
We observed that the initial relative tolerance parameter for the subproblem solver, $\tau_{\mathrm{ALM}}$, can have a significant effect on the optimization dynamics, so we tune this parameter on a per-problem basis. For completeness, we state the ALM method in \cref{alg:alm-lbfgs}.

\end{itemize}

All methods are implemented in JAX \citep{jax2018github} and compared using wall-clock time rather than iteration count. While not perfect, this metric provides a more meaningful comparison because the methods access the objective and constraint oracles in different ways. Each run is terminated as soon as either the prescribed optimality and feasibility tolerances are satisfied or the maximum time budget is exceeded. 

The reported ``stationarity" measures differ slightly across methods: for PDN we use the null-space projected gradient norm $\vnorm{\PPNJxx \grad f(\xx)}$, while for ALM and SQP we use the primal Lagrangian gradient norm $\vnorm{\Lgrad(\xx,\blambda)}$. These quantities provide comparable termination criteria under \cref{lemma:termination condition equivalence}. Unless otherwise stated, all experiments are run on a local CPU, except for the Helmholtz PINN experiment, which is run on a cluster GPU.

We manually tune the inner tolerance parameters $\tautan$ (for PDN) and the relative subproblem tolerance $\tau_{\mathrm{ALM}}$ (for ALM). In practice, this is done by sweeping over a range of values and selecting those that yield a reasonable trade-off between constraint satisfaction, primal stationarity, and early termination.

The code for our experiments can be found in our GitHub \href{https://github.com/oscar-99/PrimalDualNewtonRelease}{repository}.

\subsection{Constrained Nonlinear Least Squares}
\label{sec:NLLS}
In this experiment, we consider a constrained nonlinear least squares problem
\[
    \min_{\ww \in \real^d,\; b \in \mathbb{R}}
    \
    \frac{1}{2N}\sum_{i=1}^N
    \left(
    \sigma(\dotprod{\zz_i,\ww} + b) - y_i
    \right)^2 \
    \text{s.t.} \
    \AA \ww = \bbeta, \  \vnorm{\ww}^2 = \rho^2,
\]
where $\{(\zz_i, y_i)\}_{i=1}^N$ represents the dataset, and $\sigma(t) = (1+e^{-t})^{-1}$ is the sigmoid function. The linear constraint matrix is generated as 
$\AA_{ij} \sim \sN(0,1) $, while the right-hand side $\bbeta$ and radius
$\rho^2$ are chosen from a random reference weight vector $\ww_{\mathrm{ref}} \sim \sN(0, \eye/d)$ via
\[
\bbeta = \AA \ww_{\mathrm{ref}},
\quad
\rho^2 = \|\ww_{\mathrm{ref}}\|_2^2,
\]
so that $\ww_{\mathrm{ref}}$ is feasible by construction. We generate the initial point via $(\ww_0, b_0) \sim \sN(0, \eye)$. This problem forms a nonconvex objective with nonlinear constraints, but is relatively well behaved (i.e., well defined solution, full rank Jacobian) and is suitable for testing the functionality of our method. 

In \cref{fig:nls_mnist_alg_comparison}, we consider a problem with relatively few constraints, while in \cref{fig:nls_fashion200_alg_comparison} we consider a significantly larger number of constraints. In both cases, our method outperforms the competitor algorithms, despite the additional cost of Jacobian factorization. The performance of the zero-dual and least-squares dual variants of PDN is also very similar.

In \cref{fig:fashion_regularisation_comparison}, we study PDN (least-squares dual) in the same setting as \cref{fig:nls_fashion200_alg_comparison}, but with Hessian regularization at each iteration, i.e., $\HH \leftarrow \HH + \lambda \eye$. This experiment highlights the importance of detecting and exploiting negative curvature rather than suppressing it through regularization. As shown in \cref{fig:fashion_regularisation_comparison}, regularization degrades performance, and when $\lambda$ is sufficiently large (e.g., $\lambda=0.1,1,10$) to prevent negative curvature detection, the algorithm suffers significantly. This example shows how regularizing away negative curvature can distort crucial curvature information and harm convergence.

\begin{figure}[ht]
    \centering
    \includegraphics[width=\linewidth]{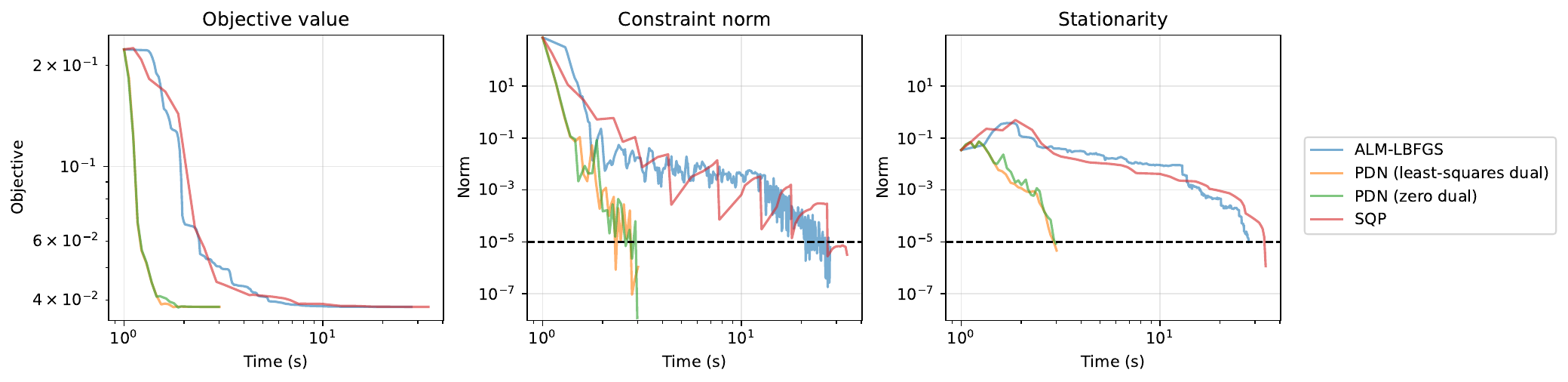}
    \caption{Comparative performance of the algorithms for the experiment in \cref{sec:NLLS} on the \texttt{MNIST} dataset (odd versus even labels) with 20 linear constraints, i.e., $d+1 = 785$ and $c = 21$. The PDN methods clearly outperform the comparison algorithms across all performance metrics, satisfying the termination conditions approximately an order of magnitude faster.}
    \label{fig:nls_mnist_alg_comparison}
\end{figure}

\begin{figure}[ht]
    \centering
    \includegraphics[width=\linewidth]{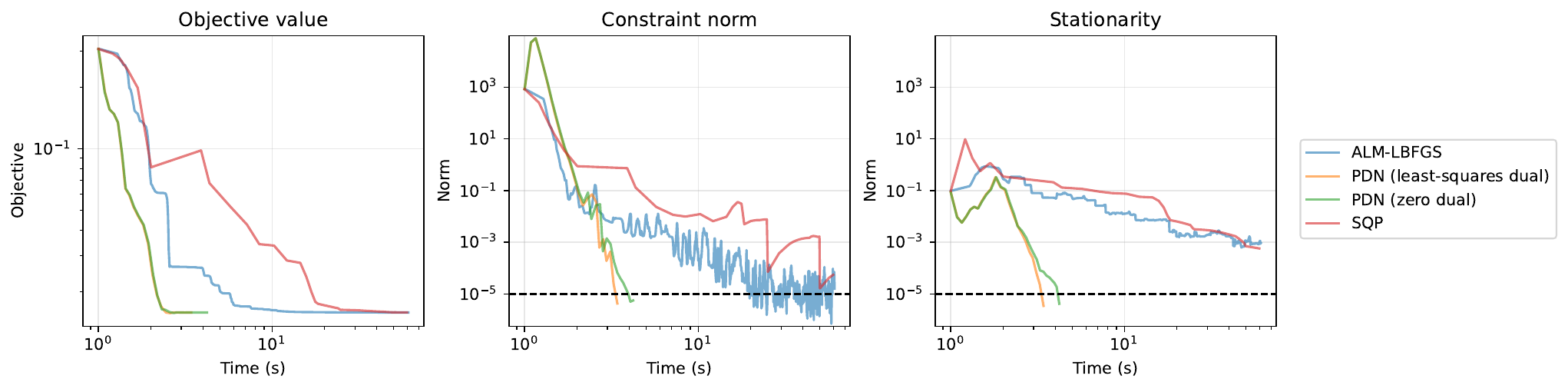}
    \caption{Comparative performance of the algorithms for the experiment in \cref{sec:NLLS} on the \texttt{FashionMNIST} dataset (odd versus even labels) with 200 linear constraints, i.e., $d+1 = 785$ and $c = 201$. The PDN methods clearly outperform the comparison algorithms across all termination metrics and, in particular, are the only methods to satisfy the primal stationarity termination condition within the time limit.}
    \label{fig:nls_fashion200_alg_comparison}
\end{figure}

\begin{figure}[ht]
    \centering
    \includegraphics[width=\linewidth]{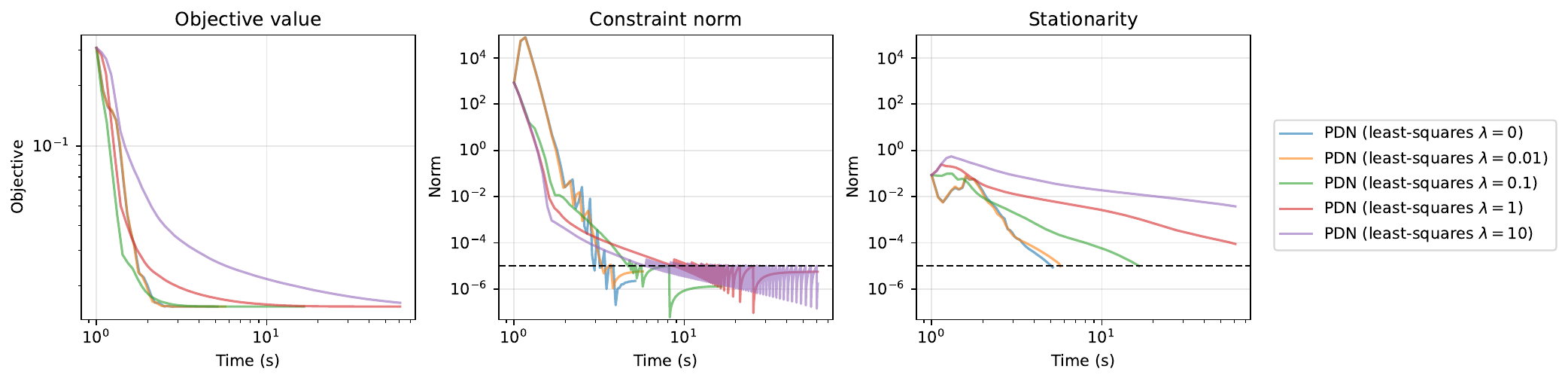}
    \caption{Comparison of PDN with least-squares multipliers under different levels of Lagrangian Hessian regularization in the setting of \cref{fig:nls_fashion200_alg_comparison}. As the regularization strength increases, performance deteriorates, particularly with respect to the primal stationarity metric. This demonstrates that excessive Hessian regularization can be detrimental to convergence.}
    \label{fig:fashion_regularisation_comparison}
\end{figure}

\subsection{Physics-Informed Neural Networks (PINNs)}
\label{sec:PINN}
\paragraph{Introduction to Constrained PINNs.}
Consider a boundary value problem of the form
\begin{align*}
    \sF[u](\zz) &= 0, \quad \zz \in \sD, \\
    \sB[u](\zz) &= 0, \quad \zz \in \partial\sD,
    \tageq\label{eqn:general PDE}
\end{align*}
where $u : \sD \subset \real^m \to \real$ is a scalar field and $\zz \in \sD$ denotes the spatial (or spatio-temporal) coordinate. The operator $\sF[\cdot]$ represents a (generally nonlinear) differential operator governing the PDE in the interior of the domain $\sD$, while $\sB[\cdot]$ is a differential operator that encodes the boundary conditions imposed on $\partial\sD$.

PINNs, originating in \citep{dissanayakeNeuralnetworkbasedApproximationsSolving1994} and popularized by \citet{raissi_physics_informed_2019}, approximate the solution to \cref{eqn:general PDE} using a neural network $\hat u_{\xx}:\real^m \to \real$ with parameters $\xx \in \real^d$. Given interior collocation points
$\sZ_{\tint}=\{\zz_i\}_{i=1}^{N_{\tint}}\subset \sD$ and boundary collocation
points $\sZ_{\bndry}=\{\zz_i\}_{i=1}^{N_{\bndry}}\subset \partial\sD$,
automatic differentiation is used to evaluate the PDE and boundary residuals directly.
A standard PINN formulation trains the network by minimizing a weighted sum of squared residuals,
\begin{align*}
    \sL(\xx)
    =
    \frac{\lambda_\sF}{N_{\tint}}
    \sum_{\zz \in \sZ_{\tint}}
    (\sF[\hat u_{\xx}](\zz))^2
    +
    \frac{\lambda_\sB}{N_{\bndry}}
    \sum_{\zz \in \sZ_{\bndry}}
    (\sB[\hat u_{\xx}](\zz))^2 .
    \tageq\label{eqn:weighted PINN objective}
\end{align*}
This formulation enables training directly from the governing equations without labeled solution data, although data-misfit terms can also be incorporated.

A drawback of this penalized formulation is that the PDE and boundary residuals may have different magnitudes and optimization dynamics, making the weights $\lambda_\sF$ and $\lambda_\sB$ problem-dependent tuning parameters. In particular, the classical PINN approach can fail to properly enforce boundary conditions, which is crucial for obtaining physically meaningful solutions \citep{wangUnderstandingMitigatingGradient2020,basir_investigating_2023}.
Moreover, in many PDE settings a reliable validation set for parameter tuning is not available \textit{a priori}. From an optimization viewpoint, \cref{eqn:weighted PINN objective} corresponds to a quadratic penalty method applied with the boundary conditions as equality constraints, which need not be enforced accurately with finite penalty weights \citep{nocedal_numerical_2006}. 
These limitations have motivated constrained PINN formulations in which the PDE residual is minimized subject to satisfaction of the sampled boundary conditions. Several recent works have explored this perspective, most commonly through augmented Lagrangian methods
\citep{basir_physics_2022,son_enhanced_2023,basirAdaptiveAugmentedLagrangian2023,huConditionallyAdaptiveAugmented2025}.
Correspondingly, such problems provide a useful test bed for our method: even for simple elliptic PDEs, the resulting optimization problems have nonconvex objectives, with highly nonlinear constraints that exhibit significant ill-conditioning in the Jacobian.

\paragraph{Problem Setup.}
We consider two two-dimensional elliptic benchmark problems with manufactured solutions. The first is a relatively benign Dirichlet Poisson equation $\Delta u = g$ on $\sD=[0,\pi]^2$, with true solution $u^\star_{\mathrm{poi}}(z_1,z_2)=\cos(z_1)\sin(z_2).$
The constrained PINN formulation is
\[
    \min_{\xx}
    \frac{1}{N_{\tint}}
    \sum_{\zz_i\in\sZ_{\tint}}
    \left(
        \Delta \hat u_{\xx}(\zz_i) - g_{\mathrm{poi}}(\zz_i)
    \right)^2
    \quad
    \text{s.t.}
    \quad
    \hat u_{\xx}(\zz_i)=h_{\mathrm{poi}}(\zz_i),
    \quad
    \zz_i\in\sZ_{\bndry}, \tageq\label{eqn:poisson pinn problem}
\]
where $\Delta$ is the Laplacian and
$g_{\mathrm{poi}}$ and $h_{\mathrm{poi}}$ are manufactured from $u^\star_{\mathrm{poi}}$.

The second is a more challenging Dirichlet Helmholtz equation
$\Delta u+k^2u=g$ on $\sD=[-1,1]^2$, with wavenumber $k=1$ and true solution $ u^\star_{\mathrm{helm}}(z_1,z_2) = \sin(\pi z_1)\sin(4\pi z_2)$. The corresponding constrained PINN formulation is
\[
    \min_{\xx}
    \frac{1}{N_{\tint}}
    \sum_{\zz_i\in\sZ_{\tint}}
    \left(
        \Delta \hat u_{\xx}(\zz_i)
        + k^2 \hat u_{\xx}(\zz_i)
        - g_{\mathrm{helm}}(\zz_i)
    \right)^2
    \quad
    \text{s.t.}
    \quad
    \hat u_{\xx}(\zz_i)=h_{\mathrm{helm}}(\zz_i),
    \quad
    \zz_i\in\sZ_{\bndry}, \tageq\label{eqn:Helmholtz pinn problem}
\]
where
$g_{\mathrm{helm}}$ and 
$h_{\mathrm{helm}}$ are manufactured from $u^\star_{\mathrm{helm}}$. This pair of examples lets us test the method on both a simple, well-behaved elliptic problem and a more oscillatory Helmholtz problem, which is known to be more difficult for PINN training \citep{wangUnderstandingMitigatingGradient2020}.

In all experiments, we use a fully connected neural network with three hidden layers of width $30$, tanh activations, domain-normalized inputs, and Xavier initialization \citep{glorotUnderstandingDifficultyTraining2010}. The interior and boundary collocation sets are generated using a Sobol sequence once at the start of training and are fixed throughout optimization. We fix the number of interior collocation points to $N_{\tint}=512$. Since our method is designed primarily for the case where the constraint Jacobian has full rank, the number of boundary collocation points must be chosen with some care. In particular, increasing $N_{\bndry}$ strengthens enforcement of the boundary conditions, but can also lead to rank deficiency in the constraint Jacobian. We monitor rank deficiency using the singular values already computed in the SVD-based normal step. Singular values below $\sigma_{\tol}=10^{-10}\max\{\sigma_1,1\}$
are treated as numerically zero, and the same tolerance is used to truncate the SVD when computing the normal step and least-squares dual update. We found that this truncation sufficiently stabilized the iterations.
For the Poisson problem, we use $N_{\bndry}=32$, for which we found that the constraint Jacobian remained full rank along the optimization trajectory. 
For the Helmholtz problem, more boundary information is needed, so we use $N_{\bndry}=128$. This results in significant rank deficiency in the boundary Jacobian.

In the Helmholtz experiments, we also use a simple merit-parameter shrinkage
heuristic. Once $\vnorm{\cc_k}\leq\epsc$, we shrink the merit parameter according to
\[
    \pi_{k+1} = \max\{0.9\pi_k,\pi_{-1}\}.
\]
This prevents the merit parameter from remaining unnecessarily large after the constraints have been enforced, which can otherwise slow progress toward objective optimality. Although this heuristic is not covered by the monotone merit parameter argument used in the convergence theory, we observe empirically that it does not destabilize the method in these experiments.

\paragraph{Results and Discussion.}
In \cref{fig:poisson_comparison,fig:helmholtz_comparison}, we compare the performance of our method in terms of objective, constraint norm and primal stationarity. Meanwhile, in \cref{fig:poisson error heatmap,fig:helmholtz error heatmap} we compare the pointwise error of the identified solution with the true solution. Note that the SQP method exhibited significantly larger errors than the other methods and was therefore omitted from this comparison to preserve the scale. Finally, in \cref{tab:helmholtz_final_test_error,tab:poisson_final_test_error} we report the final relative $L_2$ and $L_\infty$ test errors. 

In \cref{fig:poisson_comparison}, we see that our method outperforms the alternatives on both the stationarity metric and the constraint norm; indeed, the PDN variants are the only methods to satisfy the termination conditions before the time limit. In \cref{fig:poisson error heatmap} we see that this stronger optimization performance, particularly on the constraints, translates into a noticeably more accurate recovered solution. Indeed, the PDN solutions exhibit significantly smaller error near the boundary, despite the very limited boundary data used in this example. This is reflected in the final test metrics in \cref{tab:poisson_final_test_error}, where the PDN variants attain the best overall errors. Although ALM reaches a slightly lower objective value, i.e., the objective in \cref{eqn:poisson pinn problem}, this does not translate into improved test accuracy. This suggests that a lower objective value alone is not sufficient in this setting: accurate enforcement of the boundary conditions is also critical for obtaining a high-quality solution.

For the Helmholtz problem, our method retains a clear performance advantage on the termination metrics (\cref{fig:helmholtz_comparison}) compared with the ALM and SQP methods, particularly on the constraint norm. This is despite the rank-deficient setting, suggesting that the PDN framework is reasonably robust to violation of the full-rank assumption when combined with rank truncation and the merit-parameter shrink modification. In this example, the PDN variants also outperform the competing methods on the objective value itself, which is likely a consequence of the merit shrink strategy allowing the method to refocus on objective reduction once feasibility has been sufficiently improved. This stronger performance on the optimization and termination metrics is reflected in the recovered solution quality: \cref{fig:helmholtz error heatmap} shows notably better performance for the PDN methods, particularly near the boundary. In contrast, ALM appears to struggle to enforce the constraints, even with the larger number of collocation points used in this problem. Consistent with these visual observations, in \cref{tab:helmholtz_final_test_error} we observe that the PDN methods also achieve the best overall test errors. 

Until this point the zero-dual and least-squares variants of PDN have performed relatively similarly on the termination metrics; however, in this case we see that the least-squares dual variant performs better on the primal stationarity metric in particular. This could be due to the more accurate second-order information obtained when the Lagrangian Hessian is evaluated using the least-squares multipliers. In the large-scale, ill-conditioned setting considered here, that additional curvature accuracy appears to be more important, leading to more effective reduction of the primal optimality measure. It is also plausible that, when combined with the merit-parameter shrink strategy, this more accurate curvature information allows the method to make more sustained progress on the primal objective once feasibility has been largely enforced.

\begin{figure}[ht]
    \centering
    \includegraphics[width=\linewidth]{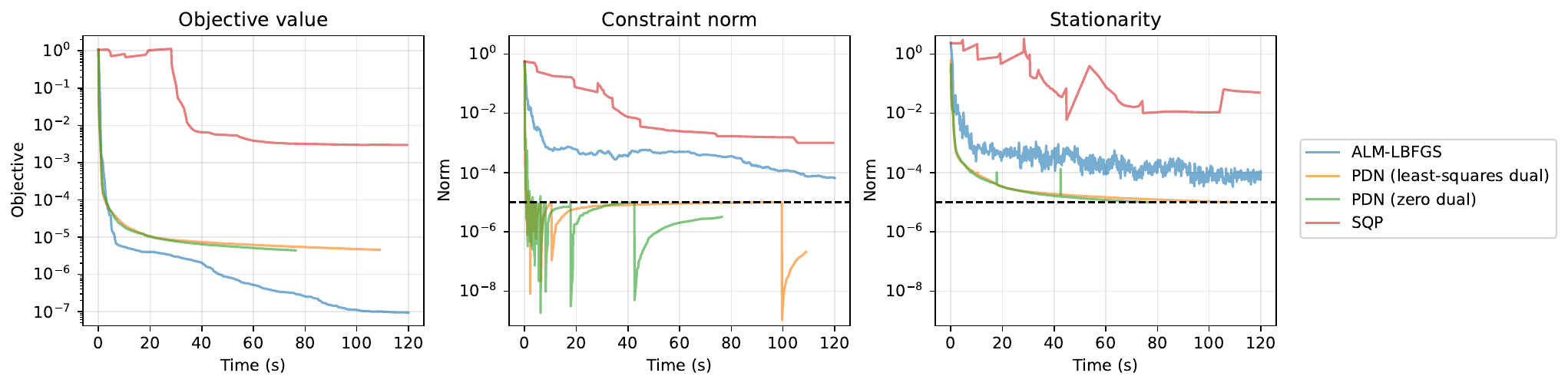}
    \caption{Comparative performance of the algorithms for the Poisson PINN experiment in \cref{sec:PINN}. The PDN methods are the only methods to satisfy the termination conditions within the time limit. They reduce constraint violations to levels orders of magnitude lower than those reached by the comparison methods, and do so in significantly less time.}
    \label{fig:poisson_comparison}
\end{figure}

\begin{figure}[ht]
    \centering
    \includegraphics[width=\linewidth]{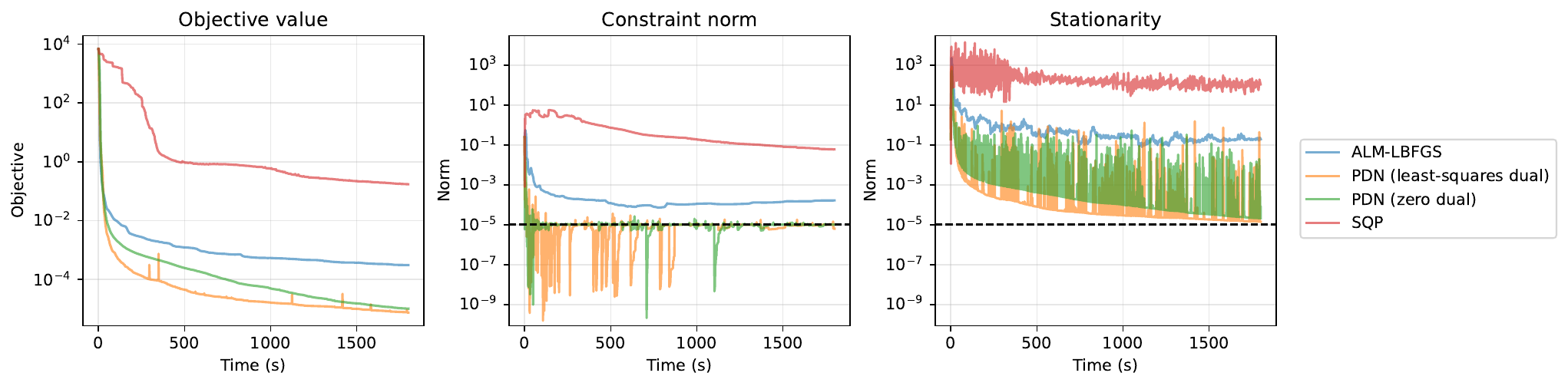}
    \caption{Comparative performance of the algorithms for the Helmholtz PINN experiment in \cref{sec:PINN}. The PDN methods outperform the comparison methods across all performance metrics. In particular, they reduce constraint violations to levels orders of magnitude lower than those reached by the comparison methods, and do so in significantly less time.}
    \label{fig:helmholtz_comparison}
\end{figure}

\begin{figure}[ht]
    \centering
    \includegraphics[width=\linewidth]{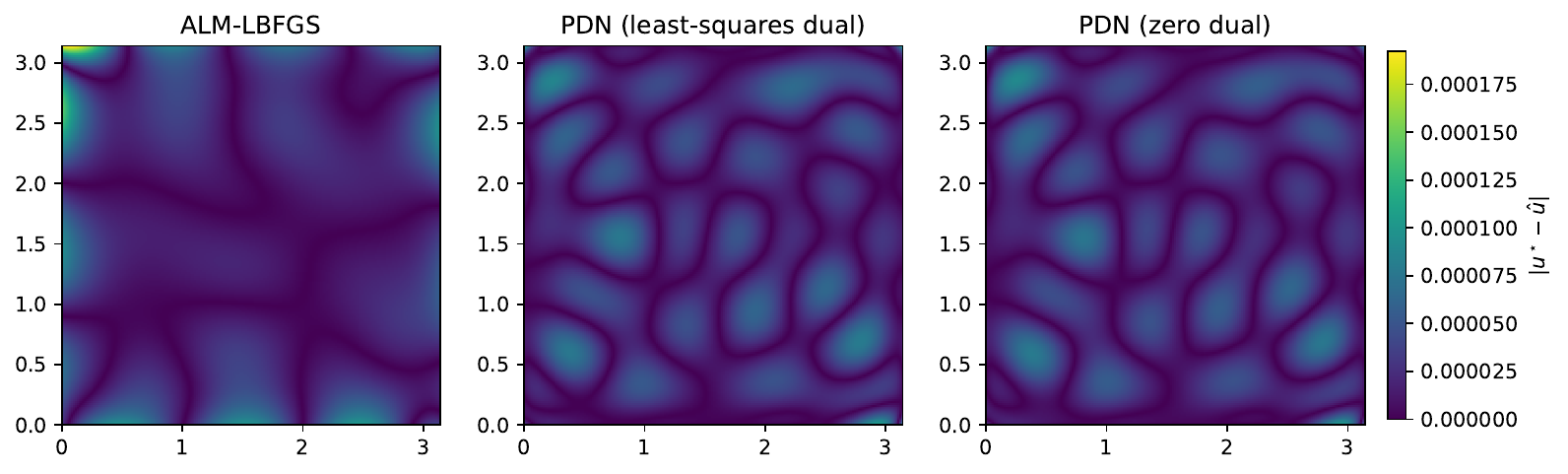}
    \caption{Heatmap of pointwise absolute error, $|\uu^\star - \hat{\uu}|$, for the Poisson PINN experiment in \cref{sec:PINN}. SQP is omitted to preserve the color scale. Compared to our method, ALM exhibits noticeably larger errors along the boundary.}
    \label{fig:poisson error heatmap}
\end{figure}

\begin{figure}[ht]
    \centering
    \includegraphics[width=\linewidth]{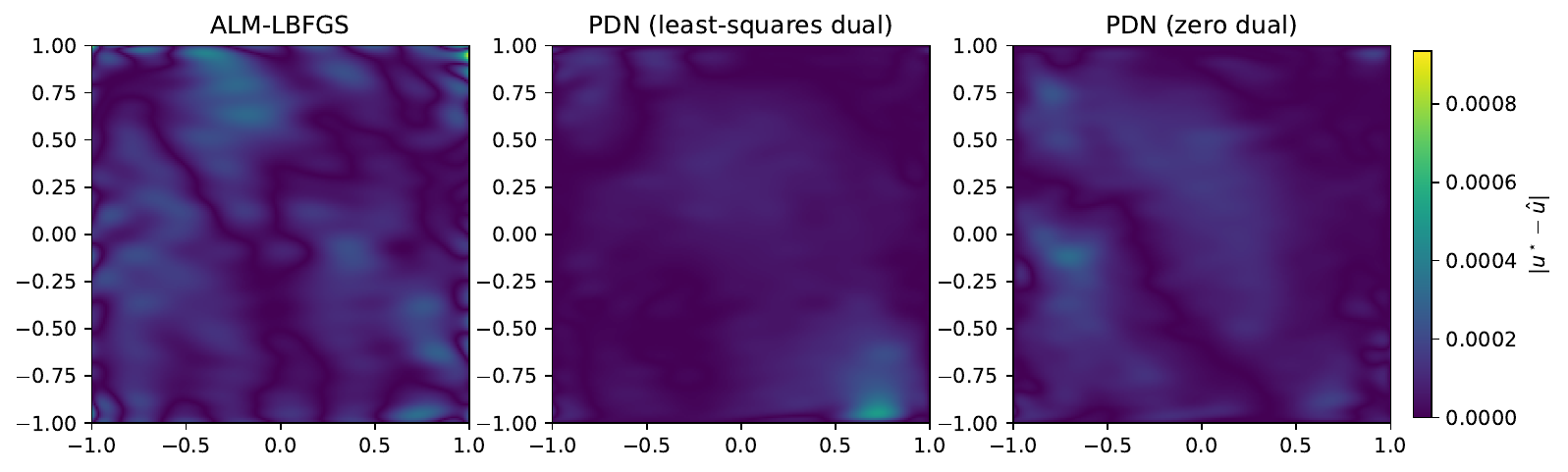}
    \caption{Heatmap of pointwise absolute error, $|\uu^\star - \hat{\uu}|$, for the Helmholtz PINN experiment in \cref{sec:PINN}. SQP is omitted to preserve the color scale. ALM performs notably worse near the boundary than our method.}
    \label{fig:helmholtz error heatmap}
\end{figure}

\begin{table}[H]
    \centering
    \begin{tabular}{|c|c|c|c|c|}
        \hline
         & \textbf{ALM-LBFGS} & \textbf{PDN (least-squares dual)} & \textbf{PDN (zero dual)} & \textbf{SQP} \\
         \hline
         $\vnorm{\hat{\uu}-\uu^\star}/\vnorm{\uu^\star}$ & $6.311\times 10^{-5}$ & $5.655\times 10^{-5}$ & $\mathbf{5.591\times 10^{-5}}$ & $6.363\times 10^{-3}$ \\
         \hline
         $\vnorm{\hat{\uu}-\uu^\star}_\infty$ & $1.925\times 10^{-4}$ & $1.090\times 10^{-4}$ & $\mathbf{1.064\times 10^{-4}}$ & $1.035\times 10^{-2}$ \\
         \hline
    \end{tabular}
    \caption{Final $L_2$ relative and $L_\infty$ test errors for the Poisson PINN example, evaluated on a uniform $64\times64$ test grid, where $\hat{\uu}$ and $\uu^\star$ denote the neural-network predictions and exact solution values on the test grid, respectively. Bold indicates the best-performing method for each metric.}
    \label{tab:poisson_final_test_error}
\end{table}

\begin{table}[H]
    \centering
    \begin{tabular}{|c|c|c|c|c|}
        \hline
         & \textbf{ALM-LBFGS} & \textbf{PDN (least-squares dual)} & \textbf{PDN (zero dual)} & \textbf{SQP} \\
         \hline
         $\vnorm{\hat{\uu}-\uu^\star}/\vnorm{\uu^\star}$ & $2.390\times 10^{-4}$ & $\mathbf{1.382\times 10^{-4}}$ & $1.687\times 10^{-4}$ & $5.484\times 10^{-2}$ \\
         \hline
         $\vnorm{\hat{\uu}-\uu^\star}_\infty$ & $8.026\times 10^{-4}$ & $5.021\times 10^{-4}$ & $\mathbf{3.287\times 10^{-4}}$ & $1.723\times 10^{-1}$ \\
         \hline
    \end{tabular}
    \caption{Final $L_2$ relative and $L_\infty$ test errors for the Helmholtz PINN example, evaluated on a uniform $64\times64$ test grid, where $\hat{\uu}$ and $\uu^\star$ denote the neural-network predictions and exact solution values on the test grid, respectively. Bold indicates the best-performing method for each metric.}
    \label{tab:helmholtz_final_test_error}
\end{table}

\FloatBarrier

\section{Conclusion and Future Directions}

In this work, we developed a fully second-order primal-dual Newton method for equality-constrained optimization. By reformulating the primal-dual system, we were able to exploit the favorable properties of the CR inner solver, yielding an inexact primal-dual Newton method that handles indefiniteness in the Lagrangian Hessian out of the box and without distortionary regularization schemes. Under standard assumptions, we establish the first global complexity guarantees of their kind for an inexact primal-dual Newton method. At the same time, our numerical results show that the method is not merely theoretically appealing but also practically effective on challenging problems with highly nonconvex objectives and nonlinear constraints.

A natural next step is to extend the theory of the method to settings where the constraint Jacobian is rank deficient, thereby formalizing the robustness that we observe empirically in the PINN experiments. Another promising direction, in line with recent developments in SQP theory, is to generalize the framework to the stochastic-objective setting. Taken together, these two extensions would substantially broaden the scope of the method and enable experiments on larger-scale PINN problems, including settings where interior collocation points are sampled during training. Carrying out such more extensive large-scale experiments is a particularly compelling direction for future work.

\section*{Acknowledgements}

Fred Roosta was partially supported by the Australian Research Council through a Discovery Project (DP250101036).

\bibliography{bib}
\bibliographystyle{apalike}

\newpage 

\appendix 

\section[
    Proof of Lemma~\ref{lemma:descent lemma}
]{
    Proof of \cref{lemma:descent lemma}
} \label{apx:proof of descent lemma}

\begin{proof}
    Applying Lipschitz smoothness of the objective, we have for any $\pp \in \real^d$
    \begin{align*}
        |f(\xx+\pp)- f(\xx) - \dotprod{\grad f(\xx), \pp}| &\leq \frac{L_f}{2} \vnorm{\pp}^2.
    \end{align*}
    Furthermore, by the Lipschitz smoothness of the constraints and the reverse triangle inequality
    \begin{align*}
        \vnorm{\cc(\xx + \pp)} - \vnorm{\cc(\xx) + \JJ(\xx)\pp}  \leq \vnorm{\cc(\xx + \pp) - \cc(\xx) - \JJ(\xx) \pp} \leq \frac{L_c}{2}\vnorm{\pp}^2,
    \end{align*}
    where the second inequality follows from a standard second-order upper bound, e.g., \citep[Theorem 3.1.6]{connTrustRegionMethods2000}.
    Applying these results to $\alpha\pp$ with $\alpha > 0$, we have 
    \begin{align*}
        \phi(\xx+ \alpha \pp; \pi) - \phi(\xx; \pi) &= f(\xx+ \alpha \pp) - f(\xx) + \pi(\vnorm{\cc(\xx + \alpha \pp)} - \vnorm{\cc(\xx)})\\
        &\leq  \alpha \dotprod{\grad f(\xx), \pp} + \pi(\vnorm{\cc(\xx) + \alpha \JJ(\xx)\pp} - \vnorm{\cc(\xx)}) + \frac{\alpha^2 \pi L_c  }{2} \vnorm{\pp}^2 + \frac{\alpha^2 L_f }{2}\vnorm{\pp}^2.
    \end{align*}
    Next, we bound the linearization of the constraint for $\alpha \in [0,1]$ as
    \begin{align*}
        \pi(\vnorm{\cc(\xx) + \alpha \JJ(\xx)\pp} - \vnorm{\cc(\xx)}) &= \pi\vnorm{\alpha\cc(\xx) + (1-\alpha)\cc(\xx) + \alpha \JJ\pp} - \vnorm{\cc(\xx)} \\
        &\leq \pi(\alpha \vnorm{\cc(\xx) + \JJ(\xx)\pp} + (1-\alpha)\vnorm{\cc(\xx)} - \vnorm{\cc(\xx)}) \\
        &= \alpha \pi(\vnorm{\cc(\xx) + \JJ(\xx)\pp} - \vnorm{\cc(\xx)}),
    \end{align*}
    which establishes \cref{eqn:merit function upper bound}. Rearranging \cref{eqn:merit function upper bound}, dividing by $\alpha$ and taking the limit $\alpha \downarrow 0$ we have 
    \begin{align*}
        D\phi(\xx; \pi)[\pp] = \lim_{\alpha \downarrow 0} \frac{\phi(\xx + \alpha \pp; \pi) - \phi(\xx; \pi)}{\alpha} \leq \dotprod{\grad f(\xx), \pp} + \pi(\vnorm{\cc(\xx) + \JJ(\xx) \pp} - \vnorm{\cc(\xx)}).
    \end{align*}
\end{proof}

\section{Projected CR Algorithm Additional Details} \label{apx:CR algorithm details}

In this section, we show how the projected CR algorithm (\cref{alg:conjugate residual}) can be implemented without access to a factorization of the projection matrix $\PP$. In particular, access to the mapping $\vv \mapsto \PP\vv$ is sufficient.

Let $\HH$, $\HHb$, $\bgg$, $\bbgg$ and $\PP=\ZZ\ZZ^\transpose$ be as in \cref{eqn:projected CR formulation}. Our goal is to demonstrate that, despite the reduced-space formulation in \cref{eqn:projected CR formulation} and the presentation of \cref{alg:conjugate residual}, in practice, we do not require explicit access to the factors of the projection matrix, i.e., $\ZZ$. 

Specifically, we show that the projected CR updates can be implemented by tracking several fundamental ``primitive vectors'' in the ambient space, to which only the unreduced Hessian-vector product, $\vv \mapsto \HH\vv$, and the projection operator, $\vv \mapsto \PP\vv$, are applied. Suppose that a reduced-space quantity $\yy$ arising in the projected CR iteration admits the representation $\yy = \ZZ^\transpose \xx$ for some $\xx$. We refer to the particular ambient-space vector $\xx$ tracked by the algorithm as the associated primitive vector of $\yy$, and denote it by $\primeb{\yy} = \xx$. Thus, the primitive is not obtained by inverting the relation $\yy = \ZZ^\transpose \xx$. Rather, it is a concrete vector generated during the iteration for which multiplication by $\ZZ^\transpose$ has simply been deferred. The corresponding reduced-space quantity is recovered, when needed, through $\yy = \ZZ^\transpose \primeb{\yy}$. By expressing the projected CR recurrences directly in terms of these ambient-space representatives, we avoid applying $\ZZ$ or $\ZZ^\transpose$ individually and require only applications of $\HH$ and $\PP = \ZZ\ZZ^\transpose$.

We begin with the initialization of the algorithm and how each key update vector within \cref{alg:conjugate residual} can be written as a primitive. Examining the initialization within \cref{alg:conjugate residual} we see that \begin{align*}
    \primeb{\yyzero} = \primeb{\brrzero} = - \bgg, \ \primeb{\bddzero} = 0, 
\end{align*}
and 
\begin{align*}
    \primeb{\HHb\yyzero}=\primeb{\HHb\brrzero} = \primeb{ -\ZZ^\transpose \HH \ZZ \ZZ^\transpose \bgg} = -\HH\PP \bgg.
\end{align*}
Notably, each of these vectors can be expressed in terms of $\HH$, $\PP$, and $\bgg$. With these initializations in hand, it follows from \cref{alg:conjugate residual} that, at every iteration $t\geq 0$, the vectors $\bddtt$, $\yytt$, $\brrtt$, and $\HHb\yytt$ can each be written in the form $\ZZ^\transpose\primeb{\cdot}$, where $\primeb{\cdot}$ denotes the corresponding primitive vector. This follows directly from the linearity of the updates, the above initialization, and the identity $\HHb=\ZZ^\transpose\HH\ZZ$, via a straightforward induction argument. Alternatively, one may observe that these vectors belong to Krylov subspaces of the form $\sK_t(\ZZ^\transpose\HH\ZZ,\ZZ^\transpose\bgg)$ or $\ZZ^\transpose\HH\ZZ\sK_t(\ZZ^\transpose\HH\ZZ,\ZZ^\transpose\bgg)$.

Next, we show that the CR iteration from $t-1$ to $t$ can be performed using the primitive vectors from the previous iteration, together with access to mappings $\vv \mapsto \HH\vv$ and $\vv \mapsto \PP\vv$. In particular, let $t>0$ and suppose we have the vectors 
\begin{align*}
    \primeb{\bddttm}, \primeb{\HHb \yyttm}, \primeb{\HHb\brrttm}, \primeb{\brrttm}, \primeb{\yyttm}, \tageq\label{eqn:projected CR storage vectors}
\end{align*}
stored in memory. Firstly, observe that
\begin{align*}
    \zeta_{t-1} &= \frac{\dotprod{\brrttm, \HHb \brrttm}}{\vnorm{\HHb\yyttm}^2} = \frac{\dotprod{\ZZ^\transpose \primeb{\brrttm}, \ZZ^\transpose \primeb{\HHb \brrttm}}}{\dotprod{\ZZ^\transpose\primeb{\HHb\yyttm}, \ZZ^\transpose\primeb{\HHb\yyttm}}} = \frac{ \dotprod{\primeb{\brrttm}, \PP \primeb{\HHb \brrttm}}}{\dotprod{\primeb{\HHb\yyttm}, \PP\primeb{\HHb\yyttm}}}, 
\end{align*}
and hence
\begin{align*}
    \primeb{\bddtt} &= \primeb{\bddttm} + \zeta_{t-1} \primeb{\yyttm}, \\
    \primeb{\brrtt} &= \primeb{\brrttm} - \zeta_{t-1} \primeb{\HHb \yyttm}.
\end{align*}
The key step in moving from iteration $t-1$ to $t$ is computing the product $\HHb \brrtt$ (in practice, this product is used to update $\HHb \yytt$ and hence $\yytt$). We compute its corresponding primitive by noting that
\begin{align*}
    \HHb\brrtt = \ZZ^\transpose \HH \ZZ \ZZ^\transpose \primeb{\brrtt},
\end{align*}
which implies
\begin{align*}
    \primeb{\HHb\brrtt} = \HH\PP \primeb{\brrtt}.
\end{align*}
With this vector in hand, we have
\begin{align*}
    \beta_{t-1} &=\frac{\dotprod{ \brrtt, \HHb \brrtt}}{\dotprod{\brrttm, \HHb \brrttm}} =  \frac{\dotprod{\ZZ^\transpose \primeb{\brrtt}, \ZZ^\transpose \primeb{\HHb \brrtt}}}{\dotprod{\ZZ^\transpose \primeb{\brrttm}, \ZZ^\transpose \primeb{\HHb \brrttm}}} = \frac{\dotprod{\primeb{\brrtt}, \PP \primeb{\HHb \brrtt}}}{\dotprod{ \primeb{\brrttm}, \PP \primeb{\HHb \brrttm}}},
\end{align*}
and hence
\begin{align*}
    \primeb{\HHb\yytt} &= \primeb{\HHb\brrtt} + \beta_{t-1} \primeb{\HHb \yyttm},\\
    \primeb{\yytt} &= \primeb{\brrtt} + \beta_{t-1} \primeb{\yyttm}.
\end{align*}
Thus, using only the mappings $\vv \mapsto \PP\vv$ and $\vv \mapsto \HH\vv$, each of the vectors in \cref{eqn:projected CR storage vectors} can be updated to iteration $t$, effectively carrying out the CR update from iteration $t-1$ to iteration $t$.

We now show that the termination conditions and return values in \cref{alg:conjugate residual} can also be computed using primitive vectors. In particular, the LPC detection condition \cref{eqn:LPC condition} can be verified for $t\geq0$ using
\begin{align*}
\dotprod{\brrtt, \HHb \brrtt}
&= \dotprod{\ZZ^\transpose \primeb{\brrtt}, \ZZ^\transpose \primeb{\HHb \brrtt}}
= \dotprod{\primeb{\brrtt}, \PP \primeb{\HHb \brrtt}},
\end{align*}
and
\begin{align*}
\vnorm{\brrtt}^2
= \dotprod{\ZZ^\transpose \primeb{\brrtt}, \ZZ^\transpose \primeb{\brrtt}}
= \dotprod{\primeb{\brrtt}, \PP \primeb{\brrtt}}.
\end{align*}
Meanwhile, \cref{eqn:tangent component termination condition} can be evaluated using
\begin{align*}
\vnorm{\HHb\brrtt}^2
= \dotprod{\primeb{\HHb\brrtt}, \PP \primeb{\HHb\brrtt}},
\end{align*}
and
\begin{align*}
\vnorm{\HHb \bddtt}^2
= \vnorm{\brrtt + \bbgg}^2
= \vnorm{\ZZ^\transpose \left(\primeb{\brrtt} + \bgg\right)}^2
= \dotprod{\primeb{\brrtt} + \bgg, \PP\left(\primeb{\brrtt} + \bgg\right)}.
\end{align*}
It remains to verify, that when \cref{alg:conjugate residual} terminates the return values are available from the primitive vectors. It is easy to see that if  $\FLAG=\SOL$, then  $\ZZ \bddtt = \ZZ \ZZ^\transpose \primeb{\bddtt} = \PP \primeb{\bddtt}$. Otherwise if $\FLAG=\LPC$,  $\ZZ \brrtt = \ZZ \ZZ^\transpose \primeb{\brrtt} = \PP \primeb{\brrtt}$. 

The preceding construction reveals that significant computational savings can be obtained by storing certain additional vectors. First, in a similar manner to standard CR, only a single Hessian-vector product, $\primeb{\HHb\brrtt} = \HH\PP \primeb{\brrtt}$, is required to perform an iteration of CR from $t-1$ to $t$. A similar reduction in the cost of applying the mapping $\vv \mapsto \PP\vv$ can be achieved by storing additional vectors.
In particular, let $t>0$ and suppose that, in addition to \cref{eqn:projected CR storage vectors}, we store the following projected vectors:
\begin{align*}
\PP\primeb{\bddttm}, \PP\primeb{\HHb \yyttm}, \PP\primeb{\HHb\brrttm}, \PP\primeb{\brrttm}, \PP\primeb{\yyttm},
\tageq\label{eqn:projected CR projected storage vectors}
\end{align*}
and perform the CR iteration from $t-1$ to $t$. By the above analysis, we can compute $\zeta_{t-1}$, $\PP\primeb{\bddtt}$, and $\PP\primeb{\brrtt}$ using \cref{eqn:projected CR storage vectors} and \cref{eqn:projected CR projected storage vectors}. Furthermore, by computing and storing $\PP \primeb{\HHb\brrtt} = \PP (\HH \PP \primeb{\brrtt})$, we can also compute $\beta_{t-1}$ and subsequently update $\PP\primeb{\HHb\yytt}$ and $\PP\primeb{\yytt}$, thereby completing the CR update from $t-1$ to $t$.
Moreover, with these quantities available, we can also verify the termination conditions and, if necessary, compute the corresponding return values. Thus, in practice, \cref{alg:conjugate residual} requires only a single application of the mapping $\vv \mapsto \PP\vv$ per iteration.

\newpage 

\section{ALM Algorithm Statement}

In this section, we state the full ALM algorithm we compare against in \cref{sec:numerical}.

\begin{algorithm}[ht]
\caption{One-loop augmented Lagrangian method with L-BFGS inner iterations.}
\label{alg:alm-lbfgs}
\begin{algorithmic}[1]
\Require $\xx_0$, $\blambda_0 =0$, $\rho=1$, $\tau_{\mathrm{ALM}}>0$, $\tau_{\min}>0$.
\State $\tau = \tau_{\mathrm{ALM}} \vnorm{\grad_\xx \sL_{\rho}(\xx_0, \blambda_0)}$, $r = \vnorm{\cc_0}$.
\For{$k=0,1,2,\dots$}
    \If{$\vnorm{\grad f_k+ \JJ^\transpose_k \blambda_{k}} \le \epsp$ \textbf{and} $\vnorm{\cc_k} \le \epsc$}
            \Return{$(\xx_k,\blambda_{k})$.}
        \EndIf
    \State Compute $\xx_{k+1}$ using an L-BFGS step with a zoom line search applied to the current augmented Lagrangian $\sL_\rho$.
    \If{$\vnorm{\grad_\xx \sL_{\rho}(\xx_{k+1}, \blambda_k)} \leq \tau$}
        \State $\blambda_{k+1} \gets \blambda_k + \rho \cc_{k+1}$
        \If{$\vnorm{\cc_{k+1}} > 0.25 \, r$ \textbf{and} $\vnorm{\cc_{k+1}} > \epsc$}
            \State $\rho \gets \min(2\rho,10^{4})$
        \EndIf
        \State $r \gets \vnorm{\cc_{k+1}}$; $\tau \gets \max(0.5\tau,\tau_{\min})$; reset L-BFGS memory.
    \Else
        \State $\blambda_{k+1} \gets \blambda_k$
    \EndIf
\EndFor
\end{algorithmic}
\end{algorithm}

\end{document}